\documentclass[10pt,a4paper]{article}
\pdfoutput=1 

\usepackage{lineno,hyperref}
\modulolinenumbers[5]

\usepackage[english]{babel}
\usepackage{graphicx}
\usepackage{marginnote}

\usepackage{tikz}
\usetikzlibrary{patterns}
\usepackage{bbm}
\usetikzlibrary{arrows}
\usepackage{mathtools}
 \usepackage{multirow}
\usepackage{amssymb}
\usepackage{amsfonts}
\usepackage{amsthm}
\usepackage{subfig}
\usepackage[scr]{rsfso} 
\usepackage{algorithm}     
\usepackage{algpseudocode} 

\newcommand{\email}[1]{\hspace*{\stretch{1}}\emph{\texttt{#1}}}

\makeatletter
\def\blfootnote{\xdef\@thefnmark{$\star$}\@footnotetext}
\makeatother
\newenvironment{Authors}%
  {\begin{center}\begin{bfseries}}%
  {\end{bfseries}\end{center}}
\newenvironment{Addresses}%
  {\begin{flushleft}\begin{itshape}}%
  {\end{itshape}\end{flushleft}}

 \usepackage{fancyhdr}  
  
\fancypagestyle{plain}{
	\fancyhead{}
	\fancyhead[C]{\hfill submitted to SIAM/ASA Journal on Uncertainty Quantification, June 2019}
}
  
\newtheorem{theorem}{Theorem}[section]
\newtheorem{theorem1}{Theorem}[section]
\newtheorem{theorem2}{Theorem}[section]
\newtheorem{proposition}[theorem]{Proposition}
\newtheorem{lemma}[theorem1]{Lemma}
\newtheorem{remark}[theorem2]{Remark}
\newtheorem{corollary}[theorem]{Corollary}

  \newcommand{\vertiii}[1]{{\left\vert\kern-0.25ex\left\vert\kern-0.25ex\left\vert #1 
    \right\vert\kern-0.25ex\right\vert\kern-0.25ex\right\vert}}

\begin{document}

\thispagestyle{plain}

\title{PBDW method  for state estimation:  error analysis for noisy data   and nonlinear formulation}
 \date{}
 
 \maketitle
\vspace{-50pt} 
 
\begin{Authors}
Helin Gong$^1$,
{Yvon Maday}$^{2,3}$,
Olga Mula$^{4}$, 
Tommaso Taddei$^{5}$
\end{Authors}

\begin{Addresses}
$^1$   Science and Technology on Reactor System Design Technology Laboratory, Nuclear Power Institute of China; 610041, Chengdu, China.
\email{gonghelin06@qq.com} \\ 
$^{2}$ 
Sorbonne Universit{\'e}, Universit{\'e}  Paris-Diderot SPC, CNRS, Laboratoire Jacques-Louis Lions (LJLL);  75005, Paris, France. 
\email{maday@ann.jussieu.fr}  \\
$^{3}$ 
Institut Universitaire de France; 75005, Paris, France. \\
$^{4}$ 
Universit{\'e} Paris-Dauphine, PSL Research University, CNRS, UMR 7534, CEREMADE;  75016, Paris, France. ´
\email{mula@ceremade.dauphine.fr}\\
$^5$
IMB, UMR 5251, Univ. Bordeaux;  33400, Talence, France.
Inria Bordeaux Sud-Ouest, Team MEMPHIS;  33400, Talence, France, \email{tommaso.taddei@inria.fr} \\
\end{Addresses}

 \begin{abstract}
We present an error analysis and further numerical investigations of the Parameterized-Background Data-Weak (PBDW) formulation to variational Data Assimilation (state estimation),   proposed in [Y Maday, AT Patera, JD Penn, M Yano, Int J Numer Meth Eng, 102(5), 933-965].
The PBDW algorithm  is a state estimation method involving reduced models. It aims at approximating an unknown function $u^{\rm true}$ living in a high-dimensional Hilbert space from $M$ measurement observations given in the form $y_m = \ell_m(u^{\rm true}),\, m=1,\dots,M$, where $\ell_m$ are linear functionals. The method approximates $u^{\rm true}$ with $\hat{u} = \hat{z} + \hat{\eta}$. The \emph{background} $\hat{z}$  belongs to an $N$-dimensional linear space $\mathcal{Z}_N$ built from reduced modelling of a parameterized mathematical model, and the \emph{update} $\hat{\eta}$ belongs to the space $\mathcal{U}_M$ spanned by the Riesz representers of $(\ell_1,\dots, \ell_M)$.  When the measurements are noisy 
{--- i.e., $y_m = \ell_m(u^{\rm true})+\epsilon_m$ with $\epsilon_m$ being a noise term --- }
the classical PBDW formulation  is not robust in the sense that, if $N$ increases, the reconstruction accuracy degrades. In this paper, we propose to address this issue with an extension of the classical formulation, 
{which consists in}  searching for the background $\hat{z}$ either on the whole $\mathcal{Z}_N$ in the noise-free case, or on a well-chosen subset $\mathcal{K}_N \subset \mathcal{Z}_N$ in presence of noise. The restriction to $\mathcal{K}_N$ makes the reconstruction be nonlinear and is the key to make the algorithm significantly more robust against noise. We   {further}
present an \emph{a priori} error and stability analysis, and we illustrate the efficiency of the approach on several numerical examples.
\end{abstract}


\section{Introduction}
\label{sec:introduction}
Let $\mathcal{U}$ be a Hilbert space defined over a domain $\Omega \subset \mathbb{R}^d$ and equipped with inner product $(\cdot, \cdot)$ and induced norm $\| \cdot \| = \sqrt{(\cdot, \cdot)}$. In this paper, we consider the following state estimation problem: we want to recover an unknown function $u^{\rm true}\in \mathcal{U}$ that represents the state of a physical system of interest from $M$ measurements given in the form
$$
y_m = \ell_m(u^{\rm true}) + \epsilon_m \quad  {m=1,\dots,M},
$$
where $\ell_1\ldots,\ell_M$ are $M$ independent linear functionals over $\mathcal{U}$ and $\epsilon_1,\ldots,\epsilon_M$ reflect the experimental noise. In the following, we gather in the vector
$$
\mathbf{y} = [y_1,\ldots,y_M]^T \in \mathbb{R}^M
$$
the set of measurement data.

{
Several authors have proposed to exploit}
 Bayesian approaches \cite{lorenc1986analysis,Stuart2010, DS2017} that consist in adding certain prior assumptions and then searching through the most plausible solution through sampling strategies of the posterior density. Since this is very costly in a high-dimensional framework, approaches involving dimensionality reduction techniques have become a very active research field in recent years. Our focus lies on strategies involving reduced modelling of parameterized PDEs for which a number of different approaches have been proposed in recent years, see  \cite{maday2013generalized,maday2015generalized,maday2015parameterized,KGV2018, AGV2019}. Note, however, that other compression approaches are possible and, in particular, we cite the works of \cite{AHP2013, berger2017sampling} in the field of signal processing and compressive sensing, which share similarities with the main ideas propagated in the reduced modelling approach as well.

Our starting point is the so-called Parameterized-Background Data-Weak method (PBDW) that was originally introduced in \cite{maday2015parameterized}. The method has been further developed  and analyzed in several works. We cite \cite{binev2017data,CDDFMN2019} for relevant works on the theoretical front, \cite{binev2018greedy} for works on sensor placement. The methodology has been applied to non-trivial applications in \cite{maday2015parameterized, hammond:hal-01698089, GGLM2019} and an analysis on how the method can be used as a vehicle to find optimal sensor locations can be found in \cite{binev2018greedy}. Our paper is devoted to the topic of the noise in measurements for which previous works are \cite{maday2015pbdw, taddei2017model,taddei2017adaptive}. 
We outline our contribution on this topic in what follows.

The {  PBDW} method  {exploits} the knowledge of a parameterized best-knowledge (bk)  model that  describes the physical system,  {to improve performance}. We denote by $u^{\rm bk}(\mu) \in \mathcal{U}$, the solution to the parameterized model for the parameter value $\mu \in \mathcal{P}^{\rm bk}$, 
$$
G^{\rm bk,\mu}(u^{\rm bk}(\mu))=0.
$$
Here, $G^{\rm bk,\mu}(\cdot)$ denotes the parameterized bk model associated with the system, and $\mathcal{P}^{\rm bk} \subset \mathbb{R}^P$ is a compact set that reflects the  {lack of knowledge} in the value of the parameters of the model.
We further define the bk manifold
$$
\mathcal{M}^{\rm bk} = \{ u^{\rm bk}(\mu): \, \mu \in \mathcal{P}^{\rm bk} \},
$$
which collects the solution to the bk model for all values of the parameter.
 {
Note that here, for simplicity of exposition, the model is defined over $\Omega$: in 
\cite{taddei2018localization}, the authors  considered the case in which the model is defined over a domain $\Omega^{\rm bk}$ that strictly contains the domain of interest $\Omega$.
}
We here intend, but we do not assume, that $u^{\rm true}$ is close to the bk manifold: there exists $\mu^{\rm true} \in \mathcal{P}^{\rm bk}$ such that
$\|  u^{\rm true}  - u^{\rm bk}(\mu^{\rm true})  \|$ is small.
In our state estimation problem, we are given the vector $\mathbf{y}\in \mathbb{R}^M$ of measurement data but the value of $\mu^{\rm true}$ is unknown so we cannot simply run a forward computation to approximate $u^{\rm true}$ with $u^{\rm bk}(\mu^{\rm true})$. That is why we refer to the   lack of knowledge of the value of $\mu^{\rm true}$ as to   \emph{anticipated or parametric ignorance}.
 {On the other hand, we refer to $\inf_{\mu \in \mathcal{P}^{\rm bk}}  \|  u^{\rm true}  - u^{\rm bk}(\mu) \|$ as to  \emph{unanticipated or nonparametric model error}.}

The PBDW  method  seeks an approximation
$$
\hat{u} = \hat{z} + \hat{\eta}
$$
employing projection by data. 
For perfect measurements,  that is $\epsilon_m  = 0,\ m=1,\dots,M$,
the estimate $\hat{u}$ is  built by searching $\hat{\eta}$ of minimum norm subject to the observation constraints $\ell_m (\hat{u}) = y_m$ for $m=1,\ldots,M$. In presence of noise, PBDW  {can be formulated} as  {a} Tikhonov regularization of the perfect-measurement statement  {that} depends on an hyper-parameter $\xi >0$ {which} should be tuned using out-of-sample data. We refer to the above mentioned literature
(see in particular \cite{maday2015parameterized,taddei2017model})
for a detailed discussion of the connections between PBDW and other existing state estimation techniques.

The first contribution to $\hat{u}$, is the \emph{deduced background estimate},
\begin{equation}
\label{eq:background_z}
\hat{z} = \sum_{n=1}^N   \hat{z}_n \zeta_n \in \mathcal{K}_N
=
\left\{
\sum_{n=1}^N \, z_n \zeta_n:
\; \;
\mathbf{z} = [z_1,\ldots,z_N]^T \in \Phi_N
\right\}
 \subset \mathcal{Z}_N: = {\rm span} \{ \zeta_n \}_{n=1}^N,
\end{equation}
where $\mathcal{Z}_N$ is an $N$-dimensional linear space spanned by the basis $\{ \zeta_n \}_{n=1}^N$, and $\mathcal{K}_N \subseteq \mathcal{Z}_N$ is a subset of $\mathcal{Z}_N$.   {The space} $\mathcal{Z}_N$ is built 
 {based on}  the bk manifold $\mathcal{M}^{\rm bk}$ and
summarizes two approximations:
  \begin{itemize}
  \item[(i)] the approximation coming from the model, which suffers from a bias   (\emph{unanticipated model error}), 
  \item[(ii)] the approximation of the elements of $\mathcal{M}^{\rm bk}$ due to the finite dimension $N$ of $\mathcal{Z}_N$.
  \end{itemize}
  Note that, while the second approximation can be systematically improved by increasing $N$, the first one is incompressible and inherent to the choice of the model.
One of the novelties with respect to previous works on noise is that we restrict the search of $\hat z$ to a well-chosen subset $\mathcal{K}_N$ of $\mathcal{Z}_N$.
The information that is encoded in $\mathcal{K}_N$ reflects some ``learning'' acquisition on the behavior of the coefficients $\mathbf{\hat{z}}=[\hat{z}_1,\dots, \hat{z}_N]^T$   of the solutions to the best-knowledge model when the parameter varies. The relevance of this set is a more complete formalization of the decrease of the Kolmogorov thickness and of course depends on the proper choice of the reduced basis $ \{ \zeta_n \}_{n=1}^N$. We will see further how this can be taken into account.
As shown later in the paper, the state estimate $\hat{u}$ is a linear function of the measurements $\mathbf{y}$ if and only if  $\Phi_N = \mathbb{R}^N$ (i.e., $\mathcal{K}_N = \mathcal{Z}_N$): for this reason, we refer to the case $\Phi_N = \mathbb{R}^N$ as linear PBDW, and to the case $\Phi_N \subsetneq \mathbb{R}^N$ 
(i.e., $\mathcal{K}_N \subsetneq \mathcal{Z}_N$)
as nonlinear PBDW. 

The second term in $\hat{u}$,  $\hat{\eta} \in \mathcal{U}_M$ is the \emph{update estimate}:
the linear $M$-dimensional space $\mathcal{U}_M$ is the span of Riesz representers $q_1,\ldots,q_M \in \mathcal{U}$
 of the $M$ observation functionals $\{ \ell_m \}_{m=1}^M$,
$$
\mathcal{U}_M \coloneqq {\rm span} \{  q_m \}_{m=1}^M,
\qquad
(q_m, v) = \ell_m(v) \quad
\forall \, v \in \mathcal{U}.
$$
The background $\hat{z}$  addresses  
 {the lack of knowledge in}
the value of the model parameters,
while the update $\hat{\eta}$ accomodates the non-parametric model error.

The contributions of the present work are twofold.
\begin{itemize}
\item[(i)]
We present a complete \emph{a priori} error analysis of linear PBDW, and we present a stability analysis for the nonlinear case.
More in detail, we 
present an error analysis for general linear recovery algorithms, which  relies on the definition of three computable  constants;
we specialize our analysis to  linear PBDW;
 and we present, once again for  linear PBDW, two optimality results   that motivate our approach. 
Furthermore, for the nonlinear case,  we prove that, if $\Phi_N$ is convex, small perturbations in the measurements lead to small perturbations in the state estimate.  As explained in section \ref{sec:analysis}, in the linear case, our analysis  is based on an extension of the framework presented  in \cite{berger2017sampling} to a broader class of linear recovery algorithms. The extension is necessary since  linear PBDW does not belong to the recovery class of \cite{berger2017sampling}.  For the analysis of the nonlinear case, we use tools  {originally developed} in the inverse problem literature (see, e.g., \cite{engl1996regularization}).
 \item[(ii)]
We  present several numerical results that empirically motivate the introduction of the constraints for the background coefficients $\hat{\mathbf{z}}$ (i.e., $\Phi_N \subsetneq \mathbb{R}^N$).
We consider   the specific case where  $\Phi_N = \bigotimes_{n=1}^N [a_n, b_n]$ and $\{a_n, b_n  \}_n$ are estimated based on the bk manifold. We present numerical investigations of the stability of the formulation as a function of 
(i) the hyper-parameter $\xi$ associated with the regularizer,
(ii) the background dimension $N$, and
(iii) the measurement locations. 
Note that   the idea of introducing box constraints has originally been introduced  in \cite{argaud2017stabilization} to stabilize the Generalized Empirical Interpolation Method in presence of noise (GEIM,  \cite{maday2013generalized}).  In this respect, the present paper can be understood as an extension of the latter methodology to PBDW.
\end{itemize}

The paper is organized as follows.
In section \ref{sec:formulation}, we present the PBDW method:
we discuss the well-posedness of the mathematical formulation, and we present the  actionable algebraic form which is used in the numerical implementation.
In section \ref{sec:analysis}, we present the analysis of the method: we here discuss the error analysis for linear PBDW and the stability bound for the nonlinear case.
To simplify the exposition, in sections \ref{sec:formulation} and \ref{sec:analysis}  we consider real-valued problems; the extension to the complex-valued case is straightforward and is briefly outlined at the end of section \ref{sec:formulation}.
In section  \ref{sec:numerics} we present several numerical results for a two-dimensional and a three-dimensional model problem, and in section \ref{sec:conclusions} we draw some conclusions.


\section{Formulation}
\label{sec:formulation}
\subsection{PBDW statement}
\label{sec:PBDW_statement}

In view of the presentation of the PBDW formulation, we recall the definition of the experimental measurements
\begin{equation}
\label{eq:measurements_gen}
y_m = \ell_m(u^{\rm true})+\epsilon_m,   \quad  m=1,...,M,
\end{equation}
where $\{ \ell_m \}_{m=1}^M \subset \mathcal{U}'$ and $\{ \epsilon_m \}_{m=1}^M$ are unknown disturbances, and of the parameterized bk mathematical model
\begin{equation}
\label{eq:mathematical_model_bk}
G^{\rm bk, \mu}(u^{\rm bk}(\mu)) = 0, 
\end{equation}
where $\mu$ corresponds to the set of uncertain parameters in the model and belongs to the compact set $\mathcal{P}^{\rm bk} \subset \mathbb{R}^P$. We here assume that $G^{\rm bk, \mu}$ is well-posed for all $\mu \in \mathcal{P}^{\rm bk}$ over a domain $\Omega^{\rm bk}$ that contains $\Omega$; we further assume that the restriction of $u^{\rm bk}(\mu)$ to $\Omega$, 
$u^{\rm bk}(\mu)|_{\Omega}$,  belongs to $\mathcal{U}$.
Then, we introduce the rank-$N$ approximation of $u^{\rm bk}|_{\Omega}$, $u_N^{\rm bk}(\mu) |_{\Omega} = \sum_{n=1}^N z_n^{\rm bk} (\mu) \zeta_n$,
 and we denote by $\Phi_N \subset \mathbb{R}^N$ a suitable 
{bounding box}  of the set 
 $\{ \mathbf{z}^{\rm bk}(\mu): \mu \in \mathcal{P}^{\rm bk} \}$. 

 We can now introduce the PBDW statement: find $\hat{u}_{\xi} = \sum_{n=1}^N \left( \hat{\mathbf{z}}_{\xi} \right)_n \zeta_n + \hat{\eta}_{\xi}$ such that $(\hat{\mathbf{z}}_{\xi}, \hat{\eta}_{\xi}) \in \Phi_N \times { \mathcal{U}}$ minimizes 
\begin{equation}
\label{eq:pbdw}
\min_{(\mathbf{z},\eta) \in \Phi_N \times {\mathcal{U}}}  \, 
\mathcal{J}_{\xi}(\mathbf{z}, \eta) :=
\xi \| \eta \|^2 + \Big\| 
\boldsymbol{\ell}   \left(  \sum_{n=1}^N \, z_n \zeta_n +\eta  \right) - \mathbf{y} \Big\|_2^2,
\end{equation}
with $\boldsymbol{\ell}  = [\ell_1,\ldots,\ell_M]^T: \mathcal{U} \to \mathbb{R}^M$, { and where $\| \cdot \|_2$ is the Euclidean $\ell^2$-norm in $\mathbb{R}^N$}. 
For reasons that will become clear soon, we further introduce the limit formulations:
\begin{equation}
\label{eq:pbdw0}
(\hat{\mathbf{z}}_{0}, \hat{\eta}_{0}) \in \,
{\rm arg} \min_{(\mathbf{z},\eta) \in \Phi_N \times \mathcal{U}}  \, 
  \| \eta \|, \quad
  {\rm subject \, to } \; \;
\boldsymbol{\ell} \left( 
\sum_{n=1}^N \, z_n \zeta_n +\eta 
\right)  = \mathbf{y};
\end{equation}
and
\begin{equation}
\label{eq:pbdw_inf}
 \hat{\mathbf{z}}_{\infty}  \in \,
{\rm arg} \min_{ \mathbf{z}  \in \Phi_N  }  \;
\Big\| 
\boldsymbol{\ell}  \left( 
\sum_{n=1}^N \, z_n \zeta_n  
\right) - \mathbf{y} \Big\|_2^2.
\end{equation}
We anticipate that \eqref{eq:pbdw0} and \eqref{eq:pbdw_inf} can be rigorously linked to \eqref{eq:pbdw}: we address this issue in the next section.


We shall now interpret the PBDW statement introduced above. The first term in \eqref{eq:pbdw} penalizes the distance of the state estimate from the set
$\mathcal{K}_N$ defined in  \eqref{eq:background_z},
which is an approximation of the bk solution manifold $\mathcal{M}^{\rm bk}$; 
the second term penalizes the data misfit; finally, the hyper-parameter $\xi>0$ regulates the relative importance of the background compared to the data. We remark that PBDW can be interpreted as a relaxation of the Partial Spline Model presented in \cite{wahba1990spline}: we refer 
 to \cite[section 2]{taddei2017adaptive} and \cite[section 2]{maday2017adaptive} for a detailed derivation. We further observe that in \eqref{eq:pbdw}
 we consider the $\ell^2$ loss,
$V_M (\cdot) = \| \cdot \|_2^2$, to penalize the data misfit: in presence of \emph{a priori} information concerning the properties of the measurement noise, other  loss functions could also be considered.

Model order reduction techniques for data compression are here employed to generate the \emph{background space} $\mathcal{Z}_N = {\rm span} \{  \zeta_n  \}_{n=1}^N$ from the bk manifold. We   refer to \cite{maday2015parameterized} and to the references therein for a detailed discussion; we further refer to \cite{taddei2018localization} for the construction of local approximation spaces when $\Omega$ is strictly contained in $\Omega^{\rm bk}$.
On the other hand, $\Phi_N \subset \mathbb{R}^N$ is built by exploiting  (estimates of) snapshots of the bk solution manifold for selected values of the parameters $\mu^1,\ldots,\mu^{n_{\rm train}} \in \mathcal{P}^{\rm bk}$. In particular, we here consider two choices for $\Phi_N$:
$\Phi_N = \mathbb{R}^N$  and $\Phi_N =  \bigotimes_{n=1}^N  [a_n, b_n]$. In the former case, it is easy to verify that PBDW reduces to {the original linear algorithm of \cite{maday2015parameterized}}, while for the second case we anticipate that computation of the state estimate requires the solution to a quadratic programming problem with box constraints.
We defer the detailed description of the definition of $\{ a_n, b_n \}_n$ to the numerical examples presented in section \ref{sec:numerics}.

%
%
%

\subsection{Finite-dimensional formulation and limit cases}
\label{sec:limit_cases}

We introduce  the matrices 
$$
\mathbf{L}  = (\mathbf{L}_{m,n})_{\substack{1\leq m \leq M \\ 1\leq n \leq N}} \in \mathbb{R}^{M,N}, \quad \mathbf{L}_{m,n} =  \ell_m(\zeta_n),
$$
and
$$
\mathbf{K}  = (\mathbf{K}_{m,m'})_{\substack{1\leq m,\, m' \leq M}} \in \mathbb{R}^{M,M}, \quad
\mathbf{K}_{m,m'} = (q_m, q_{m'})
$$
In the remainder of this work, we assume that 
$$
M \geq N .
$$
Given a symmetric positive definite matrix $\mathbf{W} \in \mathbb{R}^{M,M}$, 
we define the weighted norm $\Vert \cdot \Vert_{\mathbf{W}} $, such that for all $\mathbf{y} \in \mathbb{R}^M$ we have $\Vert \mathbf{y} \Vert_{\mathbf{W}} := \sqrt{\mathbf{y}^T \mathbf{W} \mathbf{y}}$, and we denote by 
$\lambda_{\rm min}(\mathbf{W})$ and 
$\lambda_{\rm max}(\mathbf{W})$  the minimum and maximum eigenvalues of $\mathbf{W}$.
Proposition \ref{th:representer_theorem} summarizes key properties of the PBDW formulation stated in the previous section.
The proof is provided in Appendix \ref{sec:proof_long}.

\begin{proposition}
\label{th:representer_theorem}
Let $\ell_1, \ldots, \ell_M \in \mathcal{U}'$ be 
{linear independent.}
Let $\hat{u}_{\xi} = \sum_{n=1}^N$  $\left(\hat{\mathbf{z}}_{\xi} \right)_n$ 
$\zeta_n + \hat{\eta}_{\xi}$ be a solution to  \eqref{eq:pbdw} for $\xi >0$, and let $\hat{u}_{0} = \sum_{n=1}^N  \left(\hat{\mathbf{z}}_{0} \right)_n \zeta_n + \hat{\eta}_0$ be a solution to  \eqref{eq:pbdw0}.
Then, the following hold.
  
\begin{enumerate}
\item
The updates $ \hat{\eta}_{\xi}$ and  $\hat{\eta}_{0}$ belong to the space $\mathcal{U}_M = {\rm span} \{ q_m \}_{m=1}^M$.
\item
The vector of coefficients $\hat{\mathbf{z}}_{\xi}$ associated with the  deduced background   solves the least-squares problem:
\begin{subequations}
\label{eq:two_stage_regularized}
\begin{equation}
\label{eq:two_stage_regularized_a}
\min_{\mathbf{z} \in \Phi_N} 
\,
\|  \mathbf{L} \mathbf{z}  - \mathbf{y} \|_{\mathbf{W}_{\xi} },
\quad
{\rm where}
\; \;
\mathbf{W}_{\xi} := 
\left(
\xi \mathbf{Id} + \mathbf{K}
\right)^{-1},
\end{equation}
where $\mathbf{Id}$ is the identity matrix; 
$\hat{\eta}_{\xi}$ is the unique solution to 
\begin{equation}
\label{eq:two_stage_regularized_b}
\min_{\eta \in \mathcal{U}_M} \,
\xi \| \eta \|^2 
\, + \, 
\big\|   \boldsymbol{\ell}(  \eta   )       - 
\mathbf{y}^{\rm err}( \hat{\mathbf{z}}_{\xi}   )         \big\|_2^2,
\quad
{\rm where} \; \;
\mathbf{y}^{\rm err}(\mathbf{z}   ):=
\mathbf{y}  - \mathbf{L} \mathbf{z}. 
\end{equation}
\end{subequations}
In addition, the solution  $(\hat{\mathbf{z}}_0, \hat{\eta}_0)$ to \eqref{eq:pbdw0} solves
\begin{equation}
\label{eq:two_stage_xi0}
\left\{
\begin{array}{l}
\displaystyle{
\min_{\mathbf{z} \in \Phi_N} \|  \mathbf{L} \mathbf{z}  - \mathbf{y} \|_{\mathbf{K}^{-1} },
}
\\[3mm]
\displaystyle{
\min_{\eta \in  \mathcal{U}_M} 
\,
\|  \eta \|,
\quad
{\rm subject \, to} \;
\boldsymbol{\ell}(\eta) = \mathbf{y}^{\rm err}(     \hat{\mathbf{z}}_0       ).
}
\\
\end{array}
\right.
\end{equation}
\item
Any solution $\hat{\mathbf{z}}_{\xi} $ to \eqref{eq:pbdw} satisfies 
\begin{subequations}
\label{eq:bounds_limits}
\begin{equation}
\label{eq:bounds_limits_a}
\|  \mathbf{L} \hat{\mathbf{z}}_{\xi} - \mathbf{y}   \|_2^2
\, \leq 
\,
\frac{\xi + \lambda_{\rm max}(\mathbf{K})}{\xi + \lambda_{\rm min}(\mathbf{K})} 
\,
\min_{\mathbf{z} \in \Phi_N} \, \|  \mathbf{L} \mathbf{z} - \mathbf{y} \|_2^2,
\end{equation}
and
\begin{equation}
\label{eq:bounds_limits_b}
\|  \mathbf{L} \hat{\mathbf{z}}_{\xi} - \mathbf{y}   \|_{\mathbf{K}^{-1}}^2 \, \leq 
\,
\frac{\lambda_{\rm max}(\mathbf{K})} {\lambda_{\rm min}(\mathbf{K})}
\left(
\frac{\xi + \lambda_{\rm min}(\mathbf{K})}{\xi + \lambda_{\rm max}(\mathbf{K})}
\right)
\,
\min_{\mathbf{z} \in \Phi_N} \, 
\|  \mathbf{L} \mathbf{z} - \mathbf{y} \|_{\mathbf{K}^{-1}}^2.
\end{equation}
\end{subequations}
Furthermore,  the optimal update  $\hat{\eta}_{\xi} $ satisfies
\begin{equation}
\label{eq:bounds_limits_2}
\| \hat{\eta}_{\xi}  \|^2
\leq
\frac{1}{\xi} \, \min_{\mathbf{z} \in \Phi_N} \, \|  \mathbf{L} \mathbf{z} - \mathbf{y} \|_2^2;
\qquad
 \|  \boldsymbol{\ell}(\hat{\eta}_{\xi} ) + \mathbf{L} \hat{\mathbf{z}}_{\xi} - \mathbf{y}        \|_2^2
 \leq
 \xi \, \|  \hat{\eta}_0 \|^2.
\end{equation}
\item
If $\mathbf{L}$ is full rank, any solution to \eqref{eq:pbdw} 
 is bounded  for any choice of $\Phi_N$ and for any $\xi >0$.
\item
If $\Phi_N$ is convex and $\mathbf{L}$ is full rank, then the solution  to \eqref{eq:pbdw} is  unique for any $\xi >0$.
 \end{enumerate}  
 \end{proposition}

Estimates \eqref{eq:bounds_limits} can be used to prove rigorous links between \eqref{eq:pbdw} and the limit cases 
\eqref{eq:pbdw0} and \eqref{eq:pbdw_inf}: we state the formal result in the following corollary, which is an extension of 
\cite[Proposition 2.9]{taddei2017adaptive}. 
Motivated by this corollary, with some abuse of notation, 
we extend the PBDW formulation \eqref{eq:pbdw} to $\xi \in [0,\infty]$, with the understanding that
$\xi = 0$ corresponds to \eqref{eq:pbdw0} and $\xi = \infty$ corresponds to \eqref{eq:pbdw_inf}.

\begin{corollary}
\label{th:limit_cases}
Given the sequence $\{ \xi_i \}_{i=1}^{\infty}$ such that $\xi_i >0$, we define the sequence of solutions
$\{ (\hat{\mathbf{z}}_i,  \hat{\eta}_i  ):= ( \hat{\mathbf{z}}_{\xi_i},  \hat{\eta}_{\xi_i} ) \}_{i=1}^{\infty}$ to 
\eqref{eq:pbdw}. 
Then, if $\mathbf{L}$ is full rank with $M \geq N$, 
the following hold:
\textbf{(i)}
 if $\xi_i \to \infty$,  then $\hat{\eta}_i \to 0$ and any  limit point of 
$\{ \hat{\mathbf{z}}_i \}_{i=1}^{\infty}$ is a solution to \eqref{eq:pbdw_inf};
\textbf{(ii)}
if $\xi_i \to 0$, then any limit point of $\{ (\hat{\mathbf{z}}_i,  \hat{\eta}_i  ) \}_{i=1}^{\infty}$
is a solution to \eqref{eq:pbdw0}; and 
\textbf{(iii)}
if $\Phi_N$ is convex, then the  solution map $\xi \mapsto ( \hat{\mathbf{z}}_{\xi},  \hat{\eta}_{\xi} )$ is continuous in $[0,\infty]$.
\end{corollary}

\begin{proof}
We here prove the first statement. The proofs of the second and of the third statements follow similar ideas.
Since $\mathbf{L}$ is full rank and $\mathbf{K}$ is invertible, exploiting \eqref{eq:bounds_limits_a} and \eqref{eq:bounds_limits_2}, there exists $C < \infty$ such that
$$
\sup_i \, \| \hat{\mathbf{z}}_i  \|_2 , \; \; 
\sup_i \, \xi_i \, \| \hat{\eta}_i  \|_2 \leq C.
$$
This implies that $\hat{\eta}_i \to 0$, while, applying  Bolzano-Weierstrass theorem, we find that $\{ \hat{\mathbf{z}}_i  \}_i$ admits convergent subsequences. Let $\hat{\mathbf{z}}^{\star}$ be a limit of point of  $\{ \hat{\mathbf{z}}_i  \}_i$; then  by taking the limit in \eqref{eq:bounds_limits_a}, we obtain
$$
\|   \mathbf{L}\hat{\mathbf{z}}^{\star} - \mathbf{y} \|_2^2 
\leq \,
\limsup_{i \to \infty} \, \|   \mathbf{L}\hat{\mathbf{z}}_i - \mathbf{y} \|_2^2  \; 
\leq
\;
\min_{\mathbf{z} \in \Phi_N} \, \| \mathbf{L} \mathbf{z}  - \mathbf{y}  \|_2^2,
$$
which proves the first statement.
\end{proof}

For non-convex domains $\Phi_N$, the solution to \eqref{eq:pbdw} is not in general unique: for this reason, we here restrict our attention to the case in which $\Phi_N$ is convex.
We thus specialize \eqref{eq:two_stage_regularized} to the two choices of $\Phi_N$ considered in this work.
For  $\Phi_N = \mathbb{R}^N$, the vector $\hat{\mathbf{z}}_{\xi}$ solves the linear problem: 
\begin{subequations}
\label{eq:algebraic_pbdw_std}
\begin{equation}
\label{eq:algebraic_pbdw_std_a}
\mathbf{L}^T \mathbf{W}_{\xi} \mathbf{L} \, 
\hat{\mathbf{z}}_{\xi} \, =  \,
\mathbf{L}^T \mathbf{W}_{\xi} \mathbf{y},
\end{equation}
while the vector $\hat{\boldsymbol{\eta}}_{\xi}$ associated with 
the update $\hat{\eta}_{\xi}$, 
$\hat{\eta}_{\xi} = \sum_{m=1}^M (\hat{\boldsymbol{\eta}}_{\xi})_m q_m$, satisfies
\begin{equation}
\label{eq:algebraic_pbdw_std_b}
\left( \mathbf{K} + \xi \mathbf{Id} \right) \, \hat{\boldsymbol{\eta}}_{\xi} = \mathbf{y}  - \mathbf{L} \, \hat{\mathbf{z}}_{\xi}.
\end{equation}
Note that in this case there exists a linear map between the data $\mathbf{y}$ and the solution $\hat{u}_{\xi}$. 
We further observe that the update $\hat{\eta}_{\xi}$ belongs to $\mathcal{Z}_N^{\perp} \cap \mathcal{U}_M$ (see \cite[Proposition 2.2.2]{taddei2017model}), where $\mathcal{Z}_N^{\perp}$ is the orthogonal complement of $\mathcal{Z}_N$. 
\end{subequations}
On the other hand, for $\Phi_N =  \bigotimes_{n=1}^N  [a_n, b_n]$,  $\hat{\mathbf{z}}_{\xi}$ solves the quadratic programming problem:
\begin{equation}
\label{eq:algebraic_pbdw_constrained}
\min_{\mathbf{z} \in \mathbb{R}^N} \,
\frac{1}{2} \mathbf{z}^T \left( \mathbf{L}^T \mathbf{W}_{\xi} \mathbf{L} \right) \mathbf{z}
\, - \,
\mathbf{z}^T \mathbf{L}^T \mathbf{W}_{\xi} \mathbf{y},
\quad
{\rm subject \, to} \;
a_n \leq z_n \leq b_n,
\; \; \;
n=1,\ldots,N;
\end{equation}
which can be  {easily solved with classical optimization methods.} The update $\hat{\eta}_{\xi}$ can be computed using \eqref{eq:two_stage_regularized_b}  as for the linear case.
Note that in this case the map between data and state estimate is nonlinear, 
and  the update $\hat{\eta}_{\xi}$ does not in general belong\footnote{We found empirically that explicitly adding the constraint $\eta \in \mathcal{Z}_N^{\perp}$ does not improve reconstruction performance, and can even deteriorate the accuracy of the PBDW estimate in presence of substantial model bias and moderate experimental noise.} to $\mathcal{Z}_N^{\perp} \cap \mathcal{U}_M$.
As anticipated in the introduction, we refer to  \eqref{eq:pbdw} with $\Phi_N = \mathbb{R}^N$ as to linear   PBDW, and we refer to \eqref{eq:pbdw} with $\Phi_N =  \bigotimes_{n=1}^N  [a_n, b_n]$ as to nonlinear   PBDW.

\begin{remark}
\label{remark_complex}
We can easily extend the previous developments to complex-valued problems.
If $\mathcal{U}$ is a space of complex-valued functions, the measurements $y_1,\ldots,y_M \in \mathbb{C}$ {  and}
$\Phi_N = \mathbb{C}^N$,  we can find the counterpart of \eqref{eq:algebraic_pbdw_std}:
\begin{equation}
\label{eq:algebraic_PBDW_complex_std}
\mathbf{L}^H \mathbf{W}_{\xi} \mathbf{L} \,  \hat{\mathbf{z}}_{\xi} \, =  \,
\mathbf{L}^H \mathbf{W}_{\xi} \mathbf{y},
\qquad
\left( \mathbf{K} + \xi \mathbf{Id} \right) \, \hat{\boldsymbol{\eta}}_{\xi} = \mathbf{y}  - \mathbf{L} \, \hat{\mathbf{z}}_{\xi};
\end{equation}
where $(\cdot )^H$ denotes the Hermitian conjugate. For the nonlinear case, if we set
$$
\Phi_N = \left\{  
\mathbf{z}^{\rm re} + {\rm i} \mathbf{z}^{\rm im}
\, : \,   
\mathbf{z}^{\rm re} \, \in  \bigotimes_{n=1}^N  [a_n, b_n],
\; \;
\mathbf{z}^{\rm im} \, \in  \bigotimes_{n=1}^N  [a_{n+N}, b_{n+N}],
\right \},
$$
for  some $ \{ a_n, b_n \}_{n=1}^{2N}$,
it is easy to obtain that $\widehat{\mathbf{z}}_{\xi}^{\star} = [ {\rm Re}\{  \widehat{\mathbf{z}}_{\xi} \}              ,
 {\rm Im}\{  \widehat{\mathbf{z}}_{\xi} \}             
 ] \in \mathbb{R}^{2N}$ solves
\begin{subequations}
\label{eq:algebraic_PBDW_complex_complex}
\begin{equation}
\min_{\mathbf{z}  \in \mathbb{R}^{2N}} \,
\frac{1}{2} \mathbf{z}^T \,  \mathbf{H} \,  \mathbf{z} 
\, - \,
 \mathbf{f}^T \mathbf{z} \quad 
{\rm subject \, to} \;
a_n \leq z_n \leq b_n,
\; \; \;
n=1,\ldots, 2 N;
 \end{equation}
 where
\begin{equation}
\mathbf{H} = 
\left[
\begin{array}{cc}
{\rm Re} \{  \mathbf{L}^H \mathbf{W}_{\xi} \mathbf{L}   \} &
-{\rm Im} \{  \mathbf{L}^H \mathbf{W}_{\xi} \mathbf{L}   \} \\[1mm]
{\rm Im} \{  \mathbf{L}^H \mathbf{W}_{\xi} \mathbf{L}   \}  &
{\rm Re} \{  \mathbf{L}^H \mathbf{W}_{\xi} \mathbf{L}   \}  \\
\end{array}
\right],
\quad
\mathbf{f} =
\left[
\begin{array}{c}
{\rm Re} \{  \mathbf{L}^H \mathbf{W}_{\xi} \mathbf{y}   \} \\
{\rm Im} \{  \mathbf{L}^H \mathbf{W}_{\xi} \mathbf{y}    \}  \\
\end{array}
\right].
 \end{equation}
 \end{subequations}
\end{remark}

%
%
%
%
%
%
%

\section{Analysis}
\label{sec:analysis}
We present below a mathematical analysis of the PBDW formulation for noisy measurements. 
In section \ref{sec:linear_PBDW}, we extend the analysis presented in \cite{berger2017sampling} to general linear recovery algorithms, and we apply it to PBDW.
In section \ref{sec:nonlinear_PBDW}, we prove that the solution to nonlinear PBDW  depends   
continuously on data. To conclude, in section \ref{sec:DOE}, we briefly discuss how the analysis presented in this section could be exploited to choose measurement locations.

\subsection{Analysis for linear PBDW: \emph{a priori} error bounds and optimality}
\label{sec:linear_PBDW}
\subsubsection{A general result for linear recovery algorithms}
\label{sec:stability_measures}

We first introduce some notation.
Given the  closed linear subspace $\mathcal{Q} \subset \mathcal{U}$, we denote by $\Pi_{\mathcal{Q}} : \mathcal{U} \to \mathcal{Q}$ the orthogonal projection operator onto $\mathcal{Q}$, and we denote by $\mathcal{Q}^{\perp}$ its orthogonal complement. 
We also denote by $\mathcal{L}(\mathcal{X},\mathcal{Y})$ the space of linear bounded operators  from the Hilbert space $\mathcal{X}$ to the Hilbert space $\mathcal{Y}$, equipped with the norm
$\| A  \|_{\mathcal{L}(\mathcal{X}, \mathcal{Y})} = \sup_{ v \in \mathcal{X} }  \frac{\| A(v)  \|_{\mathcal{Y}}}{\|  v \|_{\mathcal{X}}}$.
Given the algorithm $A: \mathbb{R}^M \to \mathcal{U}$,
we define the image of $A$, ${\rm Im}(A):= \{ A(\mathbf{y}): \; \mathbf{y} \in \mathbb{R}^M \}$;
we denote by $Q$ the dimension of the space ${\rm Im}(A)$, $Q \leq M$, and we denote by $\{ \psi_q \}_{q=1}^Q$ an orthonormal basis of  ${\rm Im}(A)$. We further denote by $A_{\boldsymbol{\ell}}: \mathcal{U} \to \mathcal{U}$ the $\mathcal{L}(\mathcal{U}, \mathcal{U})$ operator such that $A_{\boldsymbol{\ell}}(u) = A(\boldsymbol{\ell}(u))$. 

We can now introduce the stability constants associated with $A$:
\begin{equation}
\label{eq:Lebesgue_constant_2norm}
\Lambda_{2}(A):=
\|  A  \|_{\mathcal{L}(\mathbb{R}^M ,  \mathcal{U})}  
\, = \,
\sup_{\mathbf{y} \in \mathbb{R}^M} \, \frac{\|  A(\mathbf{y} ) \|}{\|\mathbf{y} \|_2}~,
\end{equation}
and
\begin{equation}
\label{eq:Lebesgue_constant_Unorm}
\Lambda_{\mathcal{U}}(A) :=
\|  Id  -  A_{\boldsymbol{\ell}}  \|_{\mathcal{L}(\mathcal{U},  \mathcal{U})}  
\, = \,
\sup_{u \in \mathcal{U}} \, \frac{\| u  -   A_{\boldsymbol{\ell}}(u)) \|}{\| u \|}~.
\end{equation}
We further define the biasing constant
\begin{equation}
\label{eq:Lebesgue_constant_bias}
\Lambda_{\mathcal{U}}^{\rm bias}(A):=
\|  Id  -  A_{\boldsymbol{\ell}}|_{{\rm Im}(A)}  \|_{\mathcal{L}({\rm Im}(A),  \mathcal{U})}  
\, = \,
\sup_{u \in {\rm Im}(A)} \, \frac{\| u  -   A_{\boldsymbol{\ell}}(u) \|}{\| u \|}~.
\end{equation}
Note that $\Lambda_{\mathcal{U}}^{\rm bias}(A) = 0$ if and only if
$A_{\boldsymbol{\ell}}(u) = u$ for all $u \in {\rm Im}(A)$. Next Lemma summarizes important properties of the constants introduced above.
We remark that   if $\Lambda_{\mathcal{U}}^{\rm bias}(A) \neq  0$, exact computation of 
$\Lambda_{\mathcal{U}} (A) $ is in general not possible. In the numerical experiments, we consider the approximation 
$\Lambda_{\mathcal{U}}(A) \approx \sup_{u \in \mathcal{U}_{\mathcal{N}}} \, \frac{\| u  -   A_{\boldsymbol{\ell}}(u)) \|}{\| u \|}$,
where $ \mathcal{U}_{\mathcal{N}}$ is the $\mathcal{N}$-dimensional approximation of the   space $ \mathcal{U}$, based  on a high-fidelity 
(spectral, Finite Element,...) discretization, and then we resort to an Arnoldi iterative method to (approximately) solve the corresponding eigenvalue problem.

\begin{lemma}
\label{th:properties_lebesgue_constants}
Given the linear algorithm $A: \mathbb{R}^M \to \mathcal{U}$, the following hold.
\begin{enumerate}
\item
The constants $\Lambda_2(A)$  and $\Lambda_{\mathcal{U}}^{\rm bias}(A)$ can be computed as follows:
\begin{equation}
\label{eq:algebraic_lebesgue}
\Lambda_{2}(A) = s_{\rm max}(\mathbf{A}), \quad
 \Lambda_{\mathcal{U}}^{\rm bias}(A) = s_{\rm max}(\mathbf{Id}  - {\mathbf{A}}_{\boldsymbol{\ell}} ),
\end{equation}
where $s_{\rm max}(\mathbf{W})$ denotes the maximum singular value of $\mathbf{W}$.
Here,
$\mathbf{A} \in \mathbb{R}^{Q,M}$ and ${\mathbf{A}}_{\boldsymbol{\ell}} \in \mathbb{R}^{Q,Q}$
are such that 
$\mathbf{A}_{q,m} = (A(\mathbf{e}_m),  \psi_q)$ and 
$({\mathbf{A}}_{\boldsymbol{\ell}}   )_{q,q'} = (A_{\boldsymbol{\ell}}(\psi_{q'}),  \psi_q)$, 
where $\{ \mathbf{e}_m \}_m$ is the canonical basis in $\mathbb{R}^M$.
\item
Suppose that $\Lambda_{\mathcal{U}}^{\rm bias}(A) = 0$. Then, $A_{\boldsymbol{\ell}}$ is idempotent (i.e.,
 $A_{\boldsymbol{\ell}}^2 = A_{\boldsymbol{\ell}}$), and 
$\Lambda_{\mathcal{U}}(A) = \|  A_{\boldsymbol{\ell}} \|_{\mathcal{L}(\mathcal{U}, \mathcal{U})}$.
Furthermore, we have $\Lambda_{\mathcal{U}}(A) = s_{\rm max} (\mathbf{A} \mathbf{K}^{1/2})$.
\end{enumerate}
\end{lemma}

\begin{proof}
Proof  of the identities in \eqref{eq:algebraic_lebesgue} is tedious but straightforward. We omit the details.

To prove the second statement, we recall that 
 $\Lambda_{\mathcal{U}}^{\rm bias}(A) = 0$ if and only if
$A_{\boldsymbol{\ell}}(u) = u$ for all $u \in {\rm Im}(A)$. 
The latter implies that $A_{\boldsymbol{\ell}}^2 = A_{\boldsymbol{\ell}}$. Recalling   \cite[Corollary 3]{rakovcevic2000norm}, we then obtain that
$\Lambda_{\mathcal{U}}(A) =
\| Id  -   A_{\boldsymbol{\ell}}\|_{\mathcal{L}(\mathcal{U}, \mathcal{U})} 
=
\|   A_{\boldsymbol{\ell}}\|_{\mathcal{L}(\mathcal{U},  \mathcal{U} )}$.
Finally, we observe 
$$
\Lambda_{\mathcal{U}}(A) = 
\sup_{u \in \mathcal{U}} \, \frac{\| A_{\boldsymbol{\ell}}(u)  \|}{\|  u \|}
=
\sup_{u \in \mathcal{U}_M} \, \frac{\| A_{\boldsymbol{\ell}}(u)  \|}{\|  u \|}
=
\sup_{\mathbf{q} \in \mathbb{R}^M} \, \frac{\| \mathbf{A} \mathbf{K} \mathbf{q}  \|_2}{\|  \mathbf{K}^{1/2} \mathbf{q} \|_2}
\, = \,
s_{\rm max} (\mathbf{A} \mathbf{K}^{1/2}),
$$
which completes the proof.
\end{proof}

Proposition \ref{th:measures_stability} links the previously-defined quantities to the state estimation error.
We observe that $\Lambda_2(A)$ measures the sensitivity of $A$ to measurement error, while $\Lambda_{\mathcal{U}}(A)$ measures the sensitivity to the approximation error --- given by $\| \Pi_{{\rm Im}(A)^{\perp}} u^{\rm true}  \|$.
Finally, $\Lambda_{\mathcal{U}}^{\rm bias}(A)$ should be interpreted as the maximum possible relative error for perfect measurements (i.e., $\mathbf{y} = \boldsymbol{\ell}(u^{\rm true})$) and perfect approximation (i.e., $u^{\rm true} \in {\rm Im}(A)$).

\begin{proposition}
\label{th:measures_stability}
Given the linear algorithm $A: \mathbb{R}^M \to \mathcal{U}$, the following estimate holds:
\begin{equation}
\label{eq:inf_norm_err}
\| A(\mathbf{y}) - u^{\rm true}  \|
\leq
\Lambda_{2}(A)
\|  \mathbf{y} - \boldsymbol{\ell}(u^{\rm true}) \|_2
\, + \,
\Lambda_{\mathcal{U} }(A)  \, 
\| \Pi_{{\rm Im}(A)^{\perp}} u^{\rm true} \| 
\, + \,
\Lambda_{\mathcal{U} }^{\rm bias}(A)  \, 
\|   u^{\rm true} \| 
~.
\end{equation}
Furthermore, if $\mathbf{y} = \boldsymbol{\ell}(u^{\rm true}) + \boldsymbol{\epsilon}$ with $\epsilon_m \overset{\rm iid}{\sim} (0,\sigma^2)$, the mean-square error is bounded by
\begin{equation}
\label{eq:avg_norm_err}
\begin{array}{rl}
\displaystyle{
\mathbb{E} \left[
\| A(\mathbf{y}) - u^{\rm true}  \|^2
\right] }  \leq &
\displaystyle{
\left( 
\Lambda_{\mathcal{U} }(A)  \, 
\| \Pi_{{\rm Im}(A)^{\perp}} u^{\rm true} \| 
\, + \,
\Lambda_{\mathcal{U} }^{\rm bias}(A)  \, 
\|   u^{\rm true} \| 
 \right)^2 \, 
}
\\[3mm]
&
\displaystyle{
\, + \,
\sigma^2 
\, {\rm trace} 
\left(
\mathbf{A}^T \mathbf{A}
\right),
}
\\
\end{array}
\end{equation}
where $\mathbf{A} $ was introduced in Lemma \ref{th:properties_lebesgue_constants}.
\end{proposition} 

\begin{proof}
Exploiting the definition of $\Lambda_{\mathcal{U} }(A) $ and $\Lambda_{\mathcal{U} }^{\rm bias}(A) $, we find
$$
\begin{array}{rl}
\displaystyle{ \|  u   - A_{\boldsymbol{\ell}}(u) \| \, \leq }
&
\displaystyle{
\|  (I  - A_{\boldsymbol{\ell}}) \Pi_{{\rm Im}(A)^{\perp}} u  \|
\, + \,
\|  (I  - A_{\boldsymbol{\ell}}) \Pi_{{\rm Im}(A)} u  \| 
}
\\[3mm]
\leq &
\displaystyle{
\Lambda_{\mathcal{U} }(A) \| \Pi_{{\rm Im}(A)^{\perp}} u  \|
\, + \,
\Lambda_{\mathcal{U} }^{\rm bias}(A) \| u  \|.
}
\end{array}
$$
{  Then, we obtain \eqref{eq:inf_norm_err}: }
$$
\begin{array}{l}
\displaystyle{
\| A(\mathbf{y}) - u^{\rm true}  \|
\leq
\| A\left( \mathbf{y} -  \boldsymbol{\ell} (u^{\rm true}) \right)  \|
+
\|   A_{\boldsymbol{\ell}}(u^{\rm true}) - u^{\rm true}  \|
}
\\[3mm]
\displaystyle{
\hspace{0.8in} \leq
\Lambda_{2}(A)
\|  \mathbf{y} - \boldsymbol{\ell}(u^{\rm true}) \|_2
\, + \,
\Lambda_{\mathcal{U} }(A) \| \Pi_{{\rm Im}(A)^{\perp}} u^{\rm true}  \|
+
\Lambda_{\mathcal{U} }^{\rm bias}(A) \| u^{\rm true}  \|.}
\end{array}
$$

To prove \eqref{eq:avg_norm_err}, 
we first define  $\mathbf{y}^{\rm true} := \boldsymbol{\ell}(u^{\rm true})$; then, { if we exploit the definition of $\mathbf{A}$,}
we find
$$
\begin{array}{l}
\displaystyle{
\mathbb{E} \left[
\| A(\mathbf{y}) - u^{\rm true}  \|^2
\right]  \le
\mathbb{E} \left[ 
\big\|
 A\left( \mathbf{y} -  \mathbf{y}^{\rm true} \right) 
\, + \,
\left( A_{\boldsymbol{\ell}}  (u^{\rm true}) -  u^{\rm true}  \right) 
 \big\|^2 
\right]
}
\\[3mm]
\hspace{0.1in} 
\displaystyle{
=
\mathbb{E} \left[  \big\|  A\left( \boldsymbol{\epsilon}  \right)   \big\|^2  \right]
\, + \,
\big\| A_{\boldsymbol{\ell}}  (u^{\rm true}) -  u^{\rm true}  \big\|^2
=
\sigma^2 {\rm trace} (\mathbf{A}^T \mathbf{A})
\, + \,
\big\| A_{\boldsymbol{\ell}}  (u^{\rm true}) -  u^{\rm true}  \big\|^2}
\\[3mm]
\hspace{0.1in}
\displaystyle{
\leq
\sigma^2 {\rm trace} (\mathbf{A}^T \mathbf{A})
\, + \,
\left( 
\Lambda_{\mathcal{U} }(A) \| \Pi_{{\rm Im}(A)^{\perp}} u^{\rm true}  \|
+
\Lambda_{\mathcal{U} }^{\rm bias}(A) \| u^{\rm true}  \|
\right)^2,}
\end{array}
$$
which is the thesis.
In the second-to-last step we used the identity
 $\mathbb{E} \big[$
 $  \big\|  A\big( \boldsymbol{\epsilon}  \big)$ 
 $ \big\|^2  \big] = \mathbb{E} \left[  \big\| 
\mathbf{A}  \,  \boldsymbol{\epsilon}    \big\|_2^2  \right]$, and then we applied
\cite[Theorem C, Chapter 14.4]{rice2006mathematical} .
\end{proof}

\begin{remark}
\label{remark_berger}
\textbf{Perfect algorithms.}
In \cite{berger2017sampling}, the authors restrict their attention to \emph{perfect algorithms}, that is algorithms satisfying   $A_{\boldsymbol{\ell}}(u) = u$ for all $u \in {\rm Im}(A)$.
Clearly, a linear algorithm $A$  is perfect if and only if 
$\Lambda_{\mathcal{U}}^{\rm bias}(A) = 0$. If $A$ is perfect, estimate \eqref{eq:inf_norm_err} reduces to 
$$
\| A(\mathbf{y}) - u^{\rm true}  \|
\leq
\Lambda_{2}(A)
\|  \mathbf{y} - \boldsymbol{\ell}(u^{\rm true}) \|_2
\, + \,
\| A_{\boldsymbol{\ell}}  \|_{\mathcal{L}(\mathcal{U}, \mathcal{U})} \, 
\| \Pi_{{\rm Im}(A)^{\perp}} u^{\rm true} \| 
~,
$$
which is the error bound proved in \cite{berger2017sampling}. 
We recall that in \cite{berger2017sampling} $\Lambda_{2}(A)$ is referred to as \emph{reconstruction operator norm}, while $\Lambda_{\mathcal{U}}(A)$ is called \emph{quasi-optimality constant}. As observed in the next section, for $\xi \in (0,\infty)$ 
PBDW is not perfect; therefore, the analysis in \cite{berger2017sampling} cannot be applied.
\end{remark}

\subsubsection{Application to Ridge regression}
\label{sec:remark_ridge}
Before applying the error analysis to  PBDW, we specialize our analysis to the recovery algorithm associated with the following optimization statement:
\begin{equation}
\label{eq:spline_smoothing}
\min_{u \in \mathcal{U}} \, \xi \|  u  \|^2 \, + \, \|   \boldsymbol{\ell}(u) - \mathbf{y} \|_2^2.
\end{equation}
We denote by   $A_{\xi}$ 
the recovery algorithm associated to  \eqref{eq:spline_smoothing}.
We remark that \eqref{eq:spline_smoothing} has been widely studied in the context of spline smoothing
and learning theory:
more in detail,  \eqref{eq:spline_smoothing} is typically referred to as Ridge regression in the statistics literature, and as Tikhonov regularization in the inverse problem literature; we refer to 
 \cite{wendland2005approximate} and to the references therein for a thorough discussion.
We observe that  ${\rm Im}(A_{\xi}) = \mathcal{U}_M$; furthermore,  the constants in Proposition \ref{th:measures_stability} are given by
$$
\Lambda_2(A_{\xi}) = \max_{m=1,\ldots, M} \, \frac{\sqrt{\lambda_m(\mathbf{K})}}{\xi + \lambda_m(\mathbf{K})},
\quad
\Lambda_{\mathcal{U}}(A_{\xi}) = 1,
\quad
\Lambda_{\mathcal{U}}^{\rm bias}(A_{\xi}) = 1 \, - \, \frac{  \lambda_{\rm min}(\mathbf{K})    }{\xi + \lambda_{\rm min}(\mathbf{K}) }.
$$
Since  
$\Lambda_{\mathcal{U}}^{\rm bias}(A_{\xi}) \neq  0$, the algorithm does not belong to the class of methods studied in  \cite{berger2017sampling}. On the other hand,  applying \eqref{eq:inf_norm_err}, we obtain the estimate:
$$
\|  A_{\xi}(\mathbf{y}) - u^{\rm true}  \| \, \leq 
\,
\frac{1}{2\sqrt{\xi}}
\,
\|  \mathbf{y} -  \boldsymbol{\ell}(u^{\rm true}) \|_2
\, + \,
\| \Pi_{\mathcal{U}_M^{\perp}} u^{\rm true}  \|
\, +\,
\left(
1 \, - \, \frac{  \lambda_{\rm min}(\mathbf{K})    }{\xi + \lambda_{\rm min}(\mathbf{K}) }
\right)
\,
\| u^{\rm true}  \|,
$$
where we used the identity
$\max_{ x \in (0, \infty)} \frac{x}{\xi + x^2} = \frac{1}{2\sqrt{\xi}}$ to bound $\Lambda_2(A_{\xi})$. 
Note that in presence of noise the optimal value $\xi^{\rm opt}$ of $\xi$ that minimizes the right-hand side of the error bound satisfies $0 < \xi^{\rm opt}  < \infty$.

\subsubsection{Application to linear PBDW}
We now specialize the analysis to the linear PBDW recovery algorithm, $A^{\rm pbdw, \xi}$.
First,  we observe that  $A^{\rm pbdw, \xi}$ satisfies:
$$
{\rm Im} \left( A^{\rm pbdw, \xi}  \right)
=
\left\{
\begin{array}{ll}
\mathcal{Z}_N \oplus \left(\mathcal{Z}_N^{\perp} \cap \mathcal{U}_M  \right)
&
\xi \in [0,\infty) 
\\[3mm]
\mathcal{Z}_N
&
\xi = \infty
\\
\end{array}
\right.
$$
Moreover, for all values of $\xi$, 
$(Id - A_{\boldsymbol{\ell}}^{\rm pbdw, \xi}) z = 0 $ for all $z \in \mathcal{Z}_N$; as a result, we can specialize \eqref{eq:inf_norm_err} as
$$
\begin{array}{l}
\displaystyle{ \| A^{\rm pbdw, \xi}(\mathbf{y}) - u^{\rm true}  \|  
\, \leq \,
\Lambda_{2}(A^{\rm pbdw, \xi})
\|  \mathbf{y} - \boldsymbol{\ell}(u^{\rm true}) \|_2
\,
}
\\[3mm]
\hspace{0.7in}
\displaystyle{
+ \,
\Lambda_{\mathcal{U} }(A^{\rm pbdw, \xi})  \, 
\| \Pi_{{\rm Im}(A^{\rm pbdw, \xi})^{\perp}} u^{\rm true} \| 
\, + \,
\Lambda_{\mathcal{U} }^{\rm bias}(A^{\rm pbdw, \xi})  \, 
\|  \Pi_{\mathcal{Z}_N^{\perp} \cap \mathcal{U}_M  } u^{\rm true} \| 
~.}
\\
\end{array}
$$
We can further bound the latter as
$$
\begin{array}{rl}
\displaystyle{
\| A^{\rm pbdw, \xi}(\mathbf{y}) - u^{\rm true}  \|
\leq}
&
\displaystyle{
\Lambda_{2}(A^{\rm pbdw, \xi})
\|  \mathbf{y} - \boldsymbol{\ell}(u^{\rm true}) \|_2
}
\\[3mm]
&
\displaystyle{
\, + \,
\left(
\Lambda_{\mathcal{U} }(A^{\rm pbdw, \xi})  \, 
\, + \,
\Lambda_{\mathcal{U} }^{\rm bias}(A^{\rm pbdw, \xi})  
\right)
\,
\|  \Pi_{\mathcal{Z}_N^{\perp}}  u^{\rm true} \| 
~.
}\\
\end{array}
$$
Therefore, we can interpret the sum $\left(
\Lambda_{\mathcal{U} }  \, + \, \Lambda_{\mathcal{U} }^{\rm bias}\right)$
as the sensitivity to the model mismatch\footnote{
More precisely, as explained in \cite[section 2.7]{maday2015parameterized},
$\|  \Pi_{\mathcal{Z}_N^{\perp}}  u^{\rm true} \| $ should be interpreted as the sum of a \emph{nonparametric model error}
$\inf_{w \in \mathcal{M}^{\rm bk}} \| u^{\rm true} - w  \|$ and of a \emph{discretization error} 
$\sup_{w \in \mathcal{M}^{\rm bk}} \|  \Pi_{\mathcal{Z}_N^{\perp}} w  \|$ associated with the compression of the solution manifold.
}
 $\|  \Pi_{\mathcal{Z}_N^{\perp}}  u^{\rm true} \| $, and 
 $\Lambda_{2}$ as the sensitivity to experimental noise.

 It is easy to show that 
$\Lambda_{\mathcal{U} }^{\rm bias}(A^{\rm pbdw, \xi}) \to 0$ as $\xi \to 0^+$ and
 $\Lambda_{\mathcal{U} }^{\rm bias}(A^{\rm pbdw, \xi}) \to 1$ as $\xi \to \infty$.
For the model problems considered in section \ref{sec:numerics}, we further empirically demonstrate that 
$\Lambda_{2}(A^{\rm pbdw, \xi})$ is monotonic decreasing in $\xi$, while 
$\Lambda_{\mathcal{U} }(A^{\rm pbdw, \xi})$ and $\Lambda_{\mathcal{U} }^{\rm bias}(A^{\rm pbdw, \xi})$
are monotonic increasing in $\xi$: this suggests that the optimal value of $\xi$ should depend on the ratio
$\|  \mathbf{y} - \boldsymbol{\ell}(u^{\rm true}) \|_2 / \|  \Pi_{\mathcal{Z}_N^{\perp}}  u^{\rm true} \| $.
In the  numerical experiments, we also find that $\Lambda_2(A^{\rm pbdw, \xi}) $ and  $\Lambda_{\mathcal{U}}(A^{\rm pbdw, \xi})$ increase as $N$ increases while $\|  \Pi_{\mathcal{Z}_N^{\perp}} u^{\rm true} \|$ decreases as $N$ increases;
as a result, the choice  of $N$  should also reflect the amount of noise and the behavior of $\|  \Pi_{\mathcal{Z}_N^{\perp}} u^{\rm true} \|$ with $N$.
Since the noise level and $\|  \Pi_{\mathcal{Z}_N^{\perp}} u^{\rm true} \|$ are  typically unknown, 
the choice of $\xi$ and $N$ should be performed online based on out-of-sample data.
We anticipate that the constrained formulation is significantly less sensitive to the choice of $N$ than the standard unconstrained approach; on the other hand,  both formulations are nearly equally sensitive to the choice of $\xi$.
We also emphasize that the present discussion for the choice of $\xi$ is in good  agreement with the conclusions drawn  in  \cite{taddei2017adaptive,maday2017adaptive}.


We further observe that the biasing constant satisfies:
$$
\Lambda_{\mathcal{U}}^{\rm bias}  \left( A^{\rm pbdw, \xi}  \right)
=
\left\{
\begin{array}{ll}
0
&
\xi \in \{ 0,\infty \},
\\[3mm]
>0
&
\xi  \in (0,\infty),
\\
\end{array}
\right.
$$
and is continuous in $[0,\infty)$.
Finally, we observe that for $\xi = 0$ we can relate $  \Lambda_{\mathcal{U}}$ to the inf-sup constant $ \beta_{N,M} = \inf_{z \in \mathcal{Z}_N} \sup_{q \in \mathcal{U}_M} \, \frac{(z,q)}{\| z \|  \| q\|}$ introduced in \cite{maday2015pbdw} to measure stability with respect to model mismatch for \eqref{eq:pbdw0}:
  \begin{equation}
  \label{eq:identity_inf_sup}
  \Lambda_{\mathcal{U}}(A^{\rm pbdw, \xi=0} ) = \frac{1}{\beta_{N,M}}.
  \end{equation}
Identity  \eqref{eq:identity_inf_sup}  implies that 
\eqref{eq:inf_norm_err} reduces to the estimate proved in \cite{binev2017data}
for $A=A^{\rm pbdw, \xi=0}$ and perfect measurements.

\begin{proof}
(Identity  \eqref{eq:identity_inf_sup})
The state estimate  $\hat{u}_0 = A^{\rm pbdw,\xi =0}( \boldsymbol{\ell}(u) ) \in {\rm Im}(A) = \mathcal{Z}_N \oplus (\mathcal{Z}_N^{\perp} \cap \mathcal{U}_M)$ satisfies  
$(\hat{u}_0 , q   )= (u, q)$ for all $q \in \mathcal{U}_M$. As a result, 
recalling  standard results in Functional Analysis and Lemma \ref{th:properties_lebesgue_constants},
we find $\Lambda_{\mathcal{U}}( A^{\rm pbdw,\xi =0}) =  \frac{1}{\tilde{\beta}}$ where $\tilde{\beta} = \inf_{w \in {\rm Im}(A)} \sup_{q \in \mathcal{U}_M}
\frac{(w,q)}{\| w \| \| q \|  }$. It remains to prove that
$\tilde{\beta} = \beta_{N,M}$:
$$
\begin{array}{rl}
\tilde{\beta}^2 = 
&
\displaystyle{
\inf_{(z,\eta) \in \mathcal{Z}_N \times (\mathcal{Z}_N^{\perp} \cap \mathcal{U}_M) }
\, \left(
\sup_{q \in \mathcal{U}_M}
\frac{(z+\eta, q)}{\sqrt{ \| z \|^2 + \| \eta \|^2   } \| q \|  }
\right)^2}
\\[4mm]
&
\displaystyle{
\,  \overset{\rm (i)}{=} \,
\inf_{(z,\eta) \in \mathcal{Z}_N \times (\mathcal{Z}_N^{\perp} \cap \mathcal{U}_M) } \,
\frac{\|  \Pi_{\mathcal{U}_M} (z) +\eta   \|^2}{  \| z \|^2 + \| \eta \|^2  }}
\\[4mm]
&
\displaystyle{
\overset{\rm (ii)}{=}  \,
\inf_{(z,\eta) \in \mathcal{Z}_N \times (\mathcal{Z}_N^{\perp} \cap \mathcal{U}_M) } \,
\frac{\|  \Pi_{\mathcal{U}_M} (z)   \|^2 + \|   \eta   \|^2}{ \| z \|^2 + \| \eta \|^2}
}
\\[4mm]
\,  \, &
\displaystyle{
\overset{\rm (iii)}{=} 
\inf_{ z \in \mathcal{Z}_N } \,
\frac{\|  \Pi_{\mathcal{U}_M} (z)   \|^2  }{ \| z \|^2}
=
\inf_{ z \in \mathcal{Z}_N } \,
\sup_{q \in \mathcal{U}_M} 
\left(
\frac{(z,q)  }{ \| z \|  \| q \| }
\right)^2
=
\beta_{N,M}^2
}
\\[3mm]
\end{array}
$$ 
where we used 
(i)  $\eta \in  \mathcal{U}_M$, 
(ii) 
$(\eta, \Pi_{\mathcal{U}_M} z) = (\eta, z) = 0$
(which exploits the fact that $\eta \in \mathcal{Z}_N^{\perp}  \cap \mathcal{U}_M$), and
(iii) $\|  \Pi_{\mathcal{U}_M} z  \|^2 \leq \|  z   \|^2$.
\end{proof}

\subsubsection{Optimality of PBDW algorithms}
\label{sec:optimality}

In the next two Propositions, we prove  two optimality statements satisfied by PBDW for the limit cases $\xi = 0$ and $\xi = \infty$.

The first result --- which was proved in {  Theorem 2.9 of}  \cite{binev2017data} ---
illustrates the connection between PBDW and the problem of optimal recovery (\cite{micchelli1977survey}), for perfect measurements. We recall that in \cite[Chapter 2.2.2]{taddei2017model} a similar optimality statement is proved for the case $\xi > 0$. Note that another relevant result on the optimality of PBDW can be found in the recent work \cite{CDDFMN2019}, where it is proven that the optimal affine algorithm that is possible to build among all state estimation algorithms can be expressed as a PBDW algorithm.

The second result shows that the algorithm
$\Pi_{\mathcal{Z}_N} A^{\rm pbdw, \xi}: \mathbf{y} \mapsto \hat{z}_{\xi} = \sum_{n=1}^N \, 
\left(  \hat{\mathbf{z}}_{\xi} \right)_n $ $\zeta_n $ for $\Phi_N = \mathbb{R}^N$ minimizes $\Lambda_2$ for $\xi = \infty$ and
$\Lambda_{\mathcal{U}}$ for $\xi = 0$ over all linear algorithms $A: \mathbb{R}^M \to \mathcal{Z}_N$ satisfying $\Lambda_{\mathcal{U}}^{\rm bias}(A) = 0$. As mentioned in the introduction, Proposition \ref{th:optimality_gauss_markov} has been proved in 
\cite{berger2017sampling}. The proof of Proposition \ref{th:optimality_binev} is omitted, while the proof of Proposition \ref{th:optimality_gauss_markov} --- which exploits a different argument from the one in \cite{berger2017sampling} --- is contained in Appendix \ref{sec:appendix_gauss_markov}.

\begin{proposition}
\label{th:optimality_binev}
Given the space $\mathcal{Z}_N = {\rm span} \{ \zeta_n \}_{n=1}^N \subset \mathcal{U}$ and the set of linear observation functionals $\boldsymbol{\ell}: \mathcal{U} \to \mathbb{R}^M$, we introduce the compact set
$$
\mathcal{K}_{N,M}(\delta, \mathbf{y}) :=
\left\{
u \in \mathcal{U}: \,
\| \Pi_{\mathcal{Z}_N^{\perp}} u  \| \leq \delta, \; \;
\boldsymbol{\ell}(u) = \mathbf{y}
\right\}
$$
where $\delta>0$ is a given constant. Then, for all $\mathbf{y} \in \mathbb{R}^M$ and $\delta>0$ such that $\mathcal{K}_{N,M}(\delta, \mathbf{y})$ is not empty, the linear  PBDW algorithm $A^{\rm pbdw, \xi=0}: \mathbb{R}^M \to \mathcal{U}$ satisfies
$$
A^{\rm pbdw, \xi=0}(\mathbf{y} )  = {\rm arg} \inf_{w \in \mathcal{U}} \, \sup_{u \in \mathcal{K}_{N,M}(\delta, \mathbf{y})} \,
\|  u  - w\|.
$$
Note that the PBDW algorithm does not depend on the value of $\delta$.
\end{proposition}

\begin{proposition}
\label{th:optimality_gauss_markov}
Given the space $\mathcal{Z}_N = {\rm span} \{ \zeta_n \}_{n=1}^N \subset \mathcal{U}$, and the set of linear functionals 
$\boldsymbol{\ell}: \mathcal{U} \to \mathbb{R}^M$,
let  $A$ be a linear algorithm 
such that ${\rm Im}(A) = \mathcal{Z}_N$ and $\Lambda_{\mathcal{U}}^{\rm bias}(A) = 0$.
Then, 
\begin{equation}
\label{eq:optimality_PBDW}
\Lambda_2(A) \geq \Lambda_2(A^{\rm pbdw, \xi = \infty}),
\qquad
\Lambda_{\mathcal{U}}(A) \geq \Lambda_{\mathcal{U}}
\left( 
 \Pi_{\mathcal{Z}_N}A^{\rm pbdw, \xi = 0} \right).
\end{equation}
\end{proposition}

\subsection{Analysis of nonlinear PBDW: a stability estimate}
\label{sec:nonlinear_PBDW}

We here show that if $\Phi_N$ is convex  the deduced background
$ \hat{\mathbf{z}}_{\xi}$ associated with the (nonlinear) PBDW solution satisfies the stability estimate  $\|  \hat{\mathbf{z}}_{\xi}(\mathbf{y}_1)  -   \hat{\mathbf{z}}_{\xi}(\mathbf{y}_2)  \|_2  \leq C  \| \mathbf{y}_1  - \mathbf{y}_2 \|_2$ for some constant $C$ and for any data $\mathbf{y}_1,  \mathbf{y}_2 \in \mathbb{R}^M$. Since the update is a linear function of the residual  $\mathbf{y} - \mathbf{L}  \hat{\mathbf{z}}_{\xi}$ (cf. Proposition \ref{th:representer_theorem}), this implies that the whole PBDW estimate depends continuously on data.
Towards this end, we recall that  the deduced background  $\hat{\mathbf{z}}_{\xi} \in \mathbb{R}^N$  satisfies
(cf. Proposition \ref{th:representer_theorem})
\begin{equation}
\label{eq:PBDW_once_again}
\hat{\mathbf{z}}_{\xi}  \in   {\rm arg} \min_{\mathbf{z} \in \Phi_N    } 
\, 
\| \mathbf{L}  \mathbf{z} - \mathbf{y}  \|_{\mathbf{W}_{\xi}}^2
\, = \,
{\rm arg} \min_{\mathbf{z} \in \Phi_N    } 
\, 
\frac{1}{2}
\mathbf{z}^T \mathbf{Q}_{\xi} \mathbf{z} + \mathbf{z}^T \mathbf{c}_{\xi},
\quad
{\rm where}
\; 
\left\{
\begin{array}{l}
\mathbf{Q}_{\xi}  = \mathbf{L}^T \mathbf{W}_{\xi} \mathbf{L}, \\[2mm]
\mathbf{c}_{\xi} = - \mathbf{L}^T \mathbf{W}_{\xi} \mathbf{y}.\\
\end{array}
\right.
\end{equation}
Furthermore,  we define the constant 
\begin{equation}
\label{eq:stability_constant_nonlinear}
\Lambda_{\xi}^{\rm nl}(\Phi_N)  = \left(
\min_{  \mathbf{z}_1,\mathbf{z}_2 \in \Phi_N, \,   \mathbf{z}_1 \neq \mathbf{z}_2    }
\,
\frac{\|   \mathbf{z}_1 -  \mathbf{z}_2     \|_{\mathbf{Q}_{\xi}}^2}{\|  \mathbf{z}_1 -  \mathbf{z}_2 \|_2^2}
\right)^{-1}.
\end{equation}
 Next Lemma lists a number of properties of $\Lambda_{\xi}^{\rm nl}(\Phi_N)$.
 
\begin{lemma}
 \label{th:properties_stability_nonlinear}
Let $\{ \zeta_n  \}_{n=1}^N$ be an orthonormal basis of $\mathcal{Z}_N$ and 
$\mathbf{L} \in \mathbb{R}^{M,N}$ be full rank with $M \geq N$. Then, the constant  
 $\Lambda_{\xi}^{\rm nl}(\Phi_N) $   \eqref{eq:stability_constant_nonlinear}
 satisfies the following.
 \begin{enumerate}
 \item
 If $\Phi_N \subset \Phi_N'$, then $\Lambda_{\xi}^{\rm nl}(\Phi_N) \leq \Lambda_{\xi}^{\rm nl}(\Phi_N')$.
 \item
 If $\Phi_N = \mathbb{R}^N$, then 
 $\Lambda_{\xi}^{\rm nl}(\Phi_N)  = 
 \|  \mathbf{Q}_{\xi}^{-1} \|_2 $.
 In particular, for $\xi=0$ (i.e., $\mathbb{Q}_{\xi} = \mathbf{L}^T \mathbf{K}^{-1} \mathbf{L}$), we have
  $\Lambda_{\xi}^{\rm nl}(\Phi_N)  = \Lambda_{\mathcal{U}}(A^{\rm pbdw, \xi=0}) = \frac{1}{\beta_{N,M}}$.
 \end{enumerate}
 \end{lemma}
 \begin{proof}
 The first statement follows directly from the definition of minimum. On the other hand, for $\Phi_N= \mathbb{R}^N$,   $\Lambda_{\xi}^{\rm nl}$ can be rewritten as:
 $$
 \Lambda_{\xi}^{\rm nl}(\mathbb{R}^N)
 \, = \,
 \left(
 \min_{\mathbf{z} \in \mathbb{R}^N} 
 \frac{\| \mathbf{L} \mathbf{z}  \|_{\mathbf{W}_{\xi}}^2 }{  \| \mathbf{z}  \|_2^2  }
  \right)^{-1} \,  = \,
 \frac{1}{\lambda_{\rm min}(\mathbf{L}^T \mathbf{W}_{\xi}  \mathbf{L})}
 =
  \|  \mathbf{Q}_{\xi}^{-1} \|_2,
  $$
where in the second identity we exploited the relationship between eigenvalues of symmetric matrices and minimum Rayleigh quotients, and in the third identity we used a standard property of the $\| \cdot \|_2$ norm  of symmetric matrices.
Finally, for $\xi=0$ exploiting \cite[Lemma 3.3]{maday2015generalized}  and the fact that $\{ \zeta_n  \}_{n=1}^N$ is an orthonormal basis,
we find that 
  $\beta_{N,M} = \lambda_{\rm min}(\mathbf{L}^T \mathbf{K}^{-1}  \mathbf{L})$. Thesis then follows recalling \eqref{eq:identity_inf_sup}.
 \end{proof}

Next Proposition motivates the definition of $ \Lambda_{\xi}^{\rm nl}$.

\begin{proposition}
\label{th:nonlinear_stability}
Let $\Phi_N$ be convex, and let the hypotheses of Lemma \ref{th:properties_stability_nonlinear} hold. Given $\mathbf{y}_1,\mathbf{y}_2 \in \mathbb{R}^M$, 
we denote by   $\hat{\mathbf{z}}_{\xi,i}$ the solution to \eqref{eq:PBDW_once_again} for  $\mathbf{y} = \mathbf{y}_i$, for $i=1,2$.
Then, we have
\begin{equation}
\label{eq:stability_estimate_nonlinear}
\|  \hat{\mathbf{z}}_{\xi,2}  - \hat{\mathbf{z}}_{\xi,1}     \|_2
\leq
\|  \mathbf{L}^T \mathbf{W}_{\xi}  \|_2 \, 
\Lambda_{\xi}^{\rm nl}(\Phi_N) \, \|  \mathbf{y}_1 - \mathbf{y}_2  \|_2,
\end{equation}
where $\Lambda_{\xi}^{\rm nl}$ is defined in \eqref{eq:stability_constant_nonlinear}.
\end{proposition}

In view of the proof, we state the following standard result 
(see, e.g., \cite[Lemma 5.13]{engl1996regularization}).

\begin{lemma}
\label{th:lemma_nonlinear}
Let $f: \mathbb{R}^N \to \mathbb{R}$ be  convex and differentiable with gradient $\nabla f$, and let $\Phi_N \subset \mathbb{R}^N$ be a closed convex set. Then, 
$$
\mathbf{z}^{\star} \in {\rm arg } \min_{\mathbf{z} \in \Phi_N} f(\mathbf{z}) \; \Leftrightarrow \;
\left(
\nabla f(\mathbf{z}^{\star} ), \, \mathbf{h} - \mathbf{z}^{\star}  
\right)_2 \geq 0
\; \; \forall \, \mathbf{h} \in \Phi_N
$$
where $(\cdot, \cdot)_2$ denotes the Euclidean inner product.
\end{lemma}

\begin{proof}
(Proposition \ref{th:nonlinear_stability}).
For $i=1,2$, problem \eqref{eq:PBDW_once_again} for $\mathbf{y} = \mathbf{y}_i$ can 
be restated as
$$
\min_{\mathbf{z} \in \mathbb{R}^N} \;
f_i( \mathbf{z}) :=
\frac{1}{2} \mathbf{z}^T \mathbf{Q}_{\xi} \mathbf{z}
+ 
\mathbf{z}^T \mathbf{c}_{\xi,i},
\qquad
\mathbf{Q}_{\xi}  = \mathbf{L}^T \mathbf{W}_{\xi} \mathbf{L},
\; \;
\mathbf{c}_{\xi,i} = - \mathbf{L}^T \mathbf{W}_{\xi} \mathbf{y}_i.
$$
Exploiting Lemma \ref{th:lemma_nonlinear}, we find for any $\mathbf{h}_1, \mathbf{h}_2 \in \Phi_N$
$$
\left(
\mathbf{Q}_{\xi} \hat{\mathbf{z}}_{\xi,i} + \mathbf{c}_{\xi,i} , \; \mathbf{h}_i -  \hat{\mathbf{z}}_{\xi,i}    \right)_2 \geq 0,
\qquad
i=1,2.
$$
If we consider $\mathbf{h}_1 = \hat{\mathbf{z}}_{\xi,2}$  and  $\mathbf{h}_2 = \hat{\mathbf{z}}_{\xi,1}$ and we sum the two inequalities, we obtain
$$
\begin{array}{l}
\displaystyle{
\left(
\mathbf{Q}_{\xi} (\hat{\mathbf{z}}_{\xi,1} -  \hat{\mathbf{z}}_{\xi,2}) \, + \,  \mathbf{c}_{\xi,1} -   \mathbf{c}_{\xi,2}, \; 
\hat{\mathbf{z}}_{\xi,2} -  \hat{\mathbf{z}}_{\xi,1}
\right)_2 \geq 0
}
\\[3mm]
\displaystyle{
\Rightarrow \, 
\|  \hat{\mathbf{z}}_{\xi,1} -  \hat{\mathbf{z}}_{\xi,2}  \|_{\mathbf{Q}_{\xi}}^2
\leq
\left(
 \mathbf{c}_{\xi,1} -   \mathbf{c}_{\xi,2},
 \hat{\mathbf{z}}_{\xi,2} -  \hat{\mathbf{z}}_{\xi,1}
\right)_2
\leq
\|   \mathbf{c}_{\xi,1} -   \mathbf{c}_{\xi,2} \|_2 
\|    \hat{\mathbf{z}}_{\xi,2} -  \hat{\mathbf{z}}_{\xi,1}  \|_2
}
\\[3mm]
\displaystyle{
\Rightarrow \, 
\|  \hat{\mathbf{z}}_{\xi,1} -  \hat{\mathbf{z}}_{\xi,2}  \|_2
\leq
\Lambda_{\xi}^{\rm nl}(\Phi_N) \;
\|   \mathbf{c}_{\xi,1} -   \mathbf{c}_{\xi,2} \|_2 
\leq \;
\Lambda_{\xi}^{\rm nl}(\Phi_N) \,
\|   \mathbf{L}^T  \,   \mathbf{W}_{\xi}  \|_2 
\,
\|   \mathbf{y}_{1} -   \mathbf{y}_{2} \|_2 
}
\\[3mm]
\end{array}
$$
 which is the thesis.
\end{proof}

\begin{remark}
\label{remark_linear_nonlinear}
\textbf{Comparison with  linear PBDW.}
Recalling the properties of the constant $\Lambda_{\xi}^{\rm nl}$ (cf. Lemma \ref{th:lemma_nonlinear}), estimate \eqref{eq:stability_estimate_nonlinear} suggests
that considering $\Phi_N \subsetneq \mathbb{R}^N$ reduces the sensitivity of \eqref{eq:PBDW_once_again} to perturbations in the data. However, if we restrict the bound \eqref{eq:stability_estimate_nonlinear} to linear algorithms, we find
$$
\|  \hat{\mathbf{z}}_{\xi,2}  - \hat{\mathbf{z}}_{\xi,1}     \|_2
\leq
\| \left( \mathbf{L}^T \mathbf{W}_{\xi} \mathbf{L} \right)^{-1}  \|_2 \, 
\|  \mathbf{L}^T \mathbf{W}_{\xi}   \|_2 \,
\|  \mathbf{y}_1 - \mathbf{y}_2  \|_2,
$$
which is clearly suboptimal compared to the optimal  bound\footnote{The optimal  bound can be trivially derived using
\eqref{eq:algebraic_pbdw_std_a}.
}
$$
\|  \hat{\mathbf{z}}_{\xi,2}  - \hat{\mathbf{z}}_{\xi,1}     \|_2
\leq
\| \left( \mathbf{L}^T \mathbf{W}_{\xi}  \mathbf{L} \right)^{-1}  \, \mathbf{L}^T \mathbf{W}_{\xi}  \|_2 \, 
\|  \mathbf{y}_1 - \mathbf{y}_2  \|_2.
$$
Although we cannot rigorously prove that reducing the admissible set $\Phi_N$ for $\hat{\mathbf{z}}_{\xi}$ always improves the stability of the formulation, several numerical results presented in the next section confirm this intuition, and ultimately motivate the nonlinear approach.
\end{remark}


\subsection{Selection of the observation functionals}
\label{sec:DOE}
In  \cite{maday2015pbdw},  a greedy algorithm called SGreedy  was proposed  to adaptively select the observation centers\footnote{
As explained in the original paper, SGreedy can also be used to choose $\ell_1^o,\ldots,\ell_M^o$ from a dictionary of available functionals
$\mathcal{L}  \subset \mathcal{U}'$.
}
 $\{ x_m  \}_{m=1}^M \subset \Omega$ for functionals of the form 
$$
\ell_m (v) = C_m \int_{\Omega} \, \omega(\| x -  x_m   \|_2) \, v(x) \, dx,
$$
where $\omega: \mathbb{R}_+ \to \mathbb{R}_+$ is a suitable convolutional kernel associated with the physical transducer.  A convergence analysis of the algorithm can be found in reference \cite{binev2018greedy}, which is a general study on greedy algorithms for the optimal sensor placement using reduced models.

SGreedy aims at maximizing the inf-sup constant $\beta_{N,M}$: recalling \eqref{eq:identity_inf_sup}, maximizing  $\beta_{N,M}$ is equivalent to minimizing  $\Lambda_{\mathcal{U}}$ for $\xi = 0$. In  \cite{taddei2017model,maday2017adaptive}, a variant of the SGreedy algorithm is proposed: first, the SGreedy routine is executed until $\beta_{N,M}$  exceeds a certain  threshold, then the remaining points are chosen to minimize the fill distance $h_M = \sup_{x \in \Omega} \, \min_{1\leq m\leq M} \, \| x - x_m \|_2$, which is empirically found to be correlated with $\Lambda_2$.

While for perfect measurements  and linear PBDW these Greedy routines are mathematically sounding and have been successfully tested, their performance for noisy measurement and nonlinear PBDW has not been fully investigated yet. In section \ref{sec:numerics}, we present numerical results for two model problems. Our numerical results suggest that SGreedy  is effective --- if compared to standard \emph{a priori} selections  --- also in presence of noise for linear methods; on the other hand, the introduction of box constraints reduces the sensitivity of the method to measurement locations.

\section{Numerical results}
\label{sec:numerics}
\subsection{A two-dimensional  problem}
\label{sec:2D_numerics}

\subsubsection{Problem statement}
We first investigate the performance of PBDW using the following two-dimensional advection-diffusion  model problem:
\begin{equation}
\label{eq:2D_model_problem}
\left\{
\begin{array}{ll}
\displaystyle{-\Delta  u_g(\mu) + \mathbf{b}(\mu) \cdot \nabla \, u_g(\mu)   \, = \, 
x_1 \, x_2
\, + \, g_1
} &
\displaystyle{ x \in \Omega := (0,1)^2 }
\\[3mm]
u_g(\mu) =  4 \, x_2 (1 - x_2) \, (1 + g_2)
&
\displaystyle{ x \in \Gamma := \{ 0 \} \times (0,1) }
\\[3mm]
\partial_n u_g(\mu ) = 0
&
\displaystyle{ x \in \partial \Omega \setminus \Gamma  }
\\ 
\end{array}
\right.
\end{equation}
where $\mathbf{b}(\mu) = \mu_1 \left[ \cos(\mu_2), \sin(\mu_2) \right]$ with 
$\mu = [\mu_1,\mu_2] \in \mathcal{P}^{\rm bk} = [0.1,10]\times [0,\pi/4]$. 
We here define the bk manifold as
$$
\mathcal{M}^{\rm bk}:= 
\left\{
u^{\rm bk}(\mu):= u_{g=0} (\mu) : \quad
\mu \in \mathcal{P}^{\rm bk}
\right\}
$$
and we consider 
$$
u^{\rm true} = u_{g= (\bar{g}_1,\bar{g}_2)}(\mu^{\rm true})
\qquad
{\rm for \, some} \; \mu^{\rm true} \in \mathcal{P}^{\rm bk},
\quad
\left\{
\begin{array}{l}
\bar{g}_1(x) = 0.2 x_1^2 \\
\bar{g}_2(x) = 0.1 \sin (2\pi x_2 ). \\
\end{array}
\right.
$$
{The lack of knowledge of } the value of $\mu^{\rm true}$ constitutes the anticipated parametric 
{ ignorance}
 in the model, while uncertainty in $g$ constitutes the unanticipated  non-parametric {  ignorance}.

We consider Gaussian observation functionals with standard deviation $r_{\rm w}=0.01$:
\begin{equation}
\label{eq:exp_observations}
\ell_m(v) =
\ell \left(
v; x_m, \, r_{\rm w}
\right)
=
 \, C(x_m) \, \int_{\Omega} \; {\rm exp} 
\left(
-\frac{1}{2 r_{\rm w}^2} \| x  - x_m  \|_2^2
\right) \,
v(x) \; dx
\end{equation}
where $C(x_m)$ is a normalization constant such that $\ell_m(1) = 1$.
To assess performance for imperfect observations, we pollute the measurements by adding 
Gaussian homoscedastic random disturbances $\epsilon_1,\ldots,\epsilon_M$:
$$
y_m = \ell_m( u^{\rm true} ) + \epsilon_m,
\quad
{\rm where} \; \;
\epsilon_m \overset{\rm iid}{\sim} \mathcal{N}(0, \sigma^2), 
\; \; 
\sigma = \frac{1}{\rm SNR} \; {\rm std} \left( \{ 
\ell \left( u^{\rm true}; \tilde{x}_j, \, r_{\rm w} \right)
  \}_{j=1}^{100}  \right),
$$
for given signal-to-noise ratio ${\rm SNR}>0$ and uniformly-randomly chosen  observation points
$\{  \tilde{x}_j \}_j \subset  \Omega$.

We define the ambient space $\mathcal{U} = H^1(\Omega)$ endowed with the inner product 
$$
(u,v) = \int_{\Omega} \,  \nabla  u \, \cdot \, \nabla  v \, +  \, u  \, v \, dx.
$$
Then, we generate the background space $\mathcal{Z}_N$ using Proper Orthogonal Decomposition (POD, \cite{volkwein2011model})
based on the $\mathcal{U}$ inner product: we compute the solution to \eqref{eq:2D_model_problem} for $g= 0$ for $n_{\rm train} = 10^3$ choices of the parameters $\{ \mu^i \}_{i=1}^{n_{\rm train}}$ in $\mathcal{P}^{\rm bk}$, then we use POD to build the background expansion $\{ \zeta_n \}_{n=1}^N$. 
Furthermore, in view of the application of nonlinear PBDW ($\Phi_N \subsetneq \mathbb{R}^N$), we set
$$
a_n := \min_{i=1,\ldots,n_{\rm train}} \left( u^{\rm bk}(\mu^i), \, \zeta_n \right),
\quad
b_n := \max_{i=1,\ldots,n_{\rm train}} \left( u^{\rm bk}(\mu^i), \, \zeta_n \right),
\quad
n=1,\ldots,N.
$$
{The property of the POD construction (i.e., the decay rate of the POD eigenvalues) gives some intuition of the fact that these bounds encode some valuable information.}

\subsubsection{Results}

\subsubsection*{Linear    PBDW}

Figure \ref{fig:constants_2D_xi} shows the behavior of $\Lambda_2$, $\Lambda_{\mathcal{U}}$ and $\Lambda_{\mathcal{U}}^{\rm bias}$ with respect to the hyper-parameter $\xi>0$, for  several values of $M$ and $N$. As anticipated in section \ref{sec:analysis}, 
 $\Lambda_2$  is monotonic decreasing in $\xi$  and increases as $N$ increases;
 $\Lambda_{\mathcal{U}}$ is monotonic increasing in $\xi$ and $N$, and decreases as $M$ increases;
 finally, $\Lambda_{\mathcal{U}}^{\rm bias}$ is monotonic increasing in $\xi$, and weakly depends on $M$ and $N$.
 We remark that we observed the same qualitative behavior for 
several other choices of $M,N$. 

\begin{figure}[h!]
\centering
 \subfloat[] {\includegraphics[width=0.3\textwidth]
 {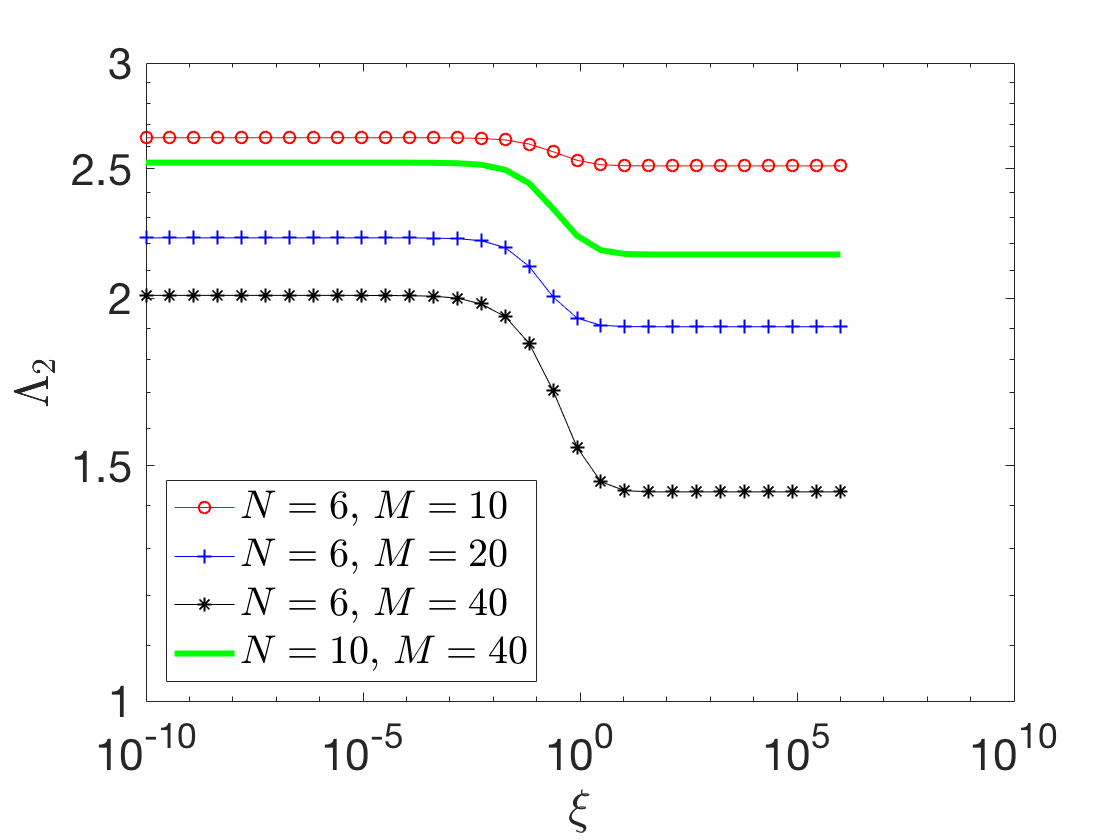}}  
    ~ ~
\subfloat[] {\includegraphics[width=0.3\textwidth]
 {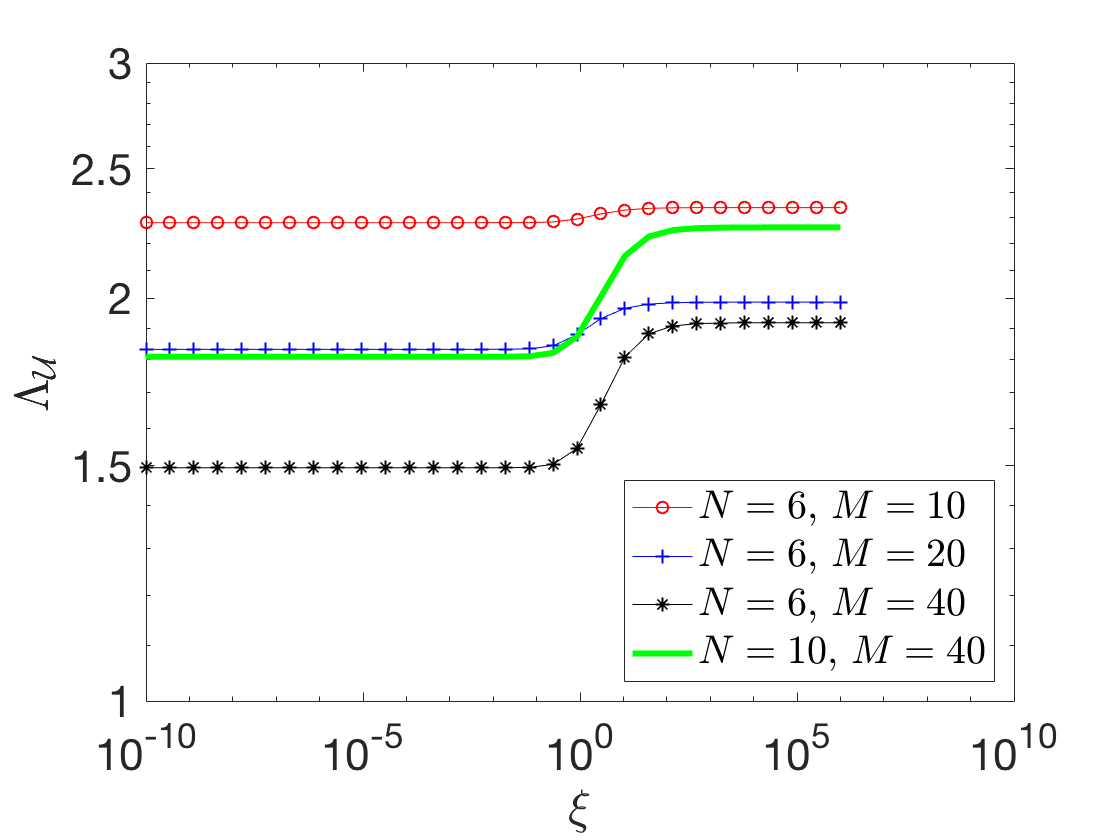}}  
  ~ ~
\subfloat[] {\includegraphics[width=0.3\textwidth]
 {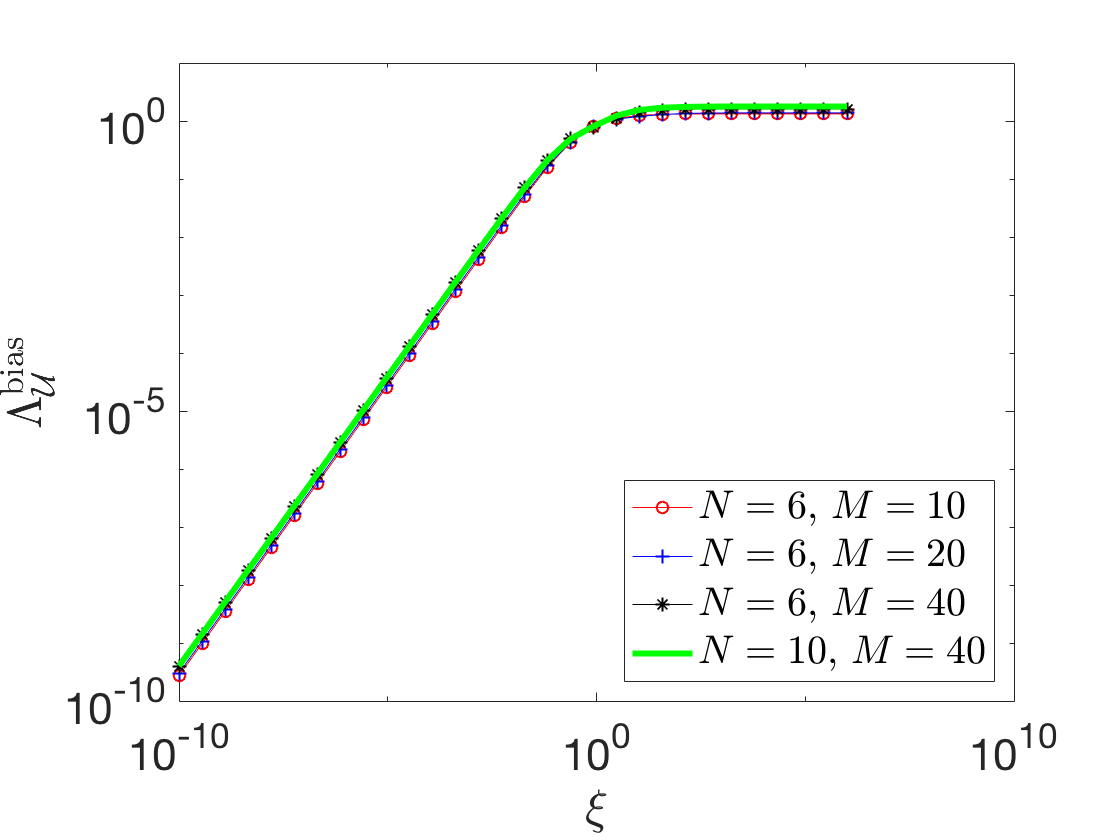}}
 
\caption{two-dimensional problem.
Behavior of $\Lambda_2$, $\Lambda_{\mathcal{U}}$ and $\Lambda_{\mathcal{U}}^{\rm bias}$ for several choices of $N$ and  $M$.
}
 \label{fig:constants_2D_xi}
\end{figure}

Figure \ref{fig:err_2D_unconstrained_xi}   shows the behavior of the average relative error
\begin{equation}
\label{eq:Eavg}
E_{\rm avg}(\xi) = 
\frac{1}{K \times n_{\rm test}} \, \sum_{j=1}^{n_{\rm test}} \,
 \sum_{k=1}^K \,\frac{\| u^{\rm true}(\mu^j)  - \hat{u}_{\xi}^{(k,j)} \|_{L^2(\Omega)}}{ 
\| u^{\rm true}    \|_{L^2(\Omega)}
},
\quad
\mu^j \overset{\rm iid}{\sim} 
{\rm Uniform}(\mathcal{P}^{\rm bk}),
\end{equation}
where 
$K=50$, $n_{\rm test} = 10$, and 
$\hat{u}_{\xi}^{(k,j)}$ is the (linear)  PBDW estimate associated with the $k$-th realization of the random disturbance $\boldsymbol{\epsilon}$, and the $j$-th true field  $u^{\rm true}(\mu^j)$ considered.
To compute $u^{\rm true}(\mu^j)$,  
we both consider  the solution to \eqref{eq:2D_model_problem} for 
$g \equiv 0$ (unbiased) and  $g \neq 0$ (biased).
We further consider two different signal-to-noise levels, 
${\rm SNR}= \infty$, ${\rm SNR}= 10$:
 the choice ${\rm SNR} = \infty$ corresponds to the case of perfect measurements. 
As expected, for perfect measurements, the optimal  value of $\xi$ is equal to zero, while for noisy measurements   optimal performance  is achieved for $\xi = \infty$ in the case of unbiased model, and for $\xi \in (0,\infty)$ in presence  of bias.
These results are in good qualitative agreements with the discussion in section \ref{sec:analysis}, and with the  results in 
\cite{taddei2017adaptive,maday2017adaptive}.

\begin{figure}[h!]
\centering
\subfloat[unbiased, ${\rm SNR} = \infty$] {\includegraphics[width=0.3\textwidth]
 {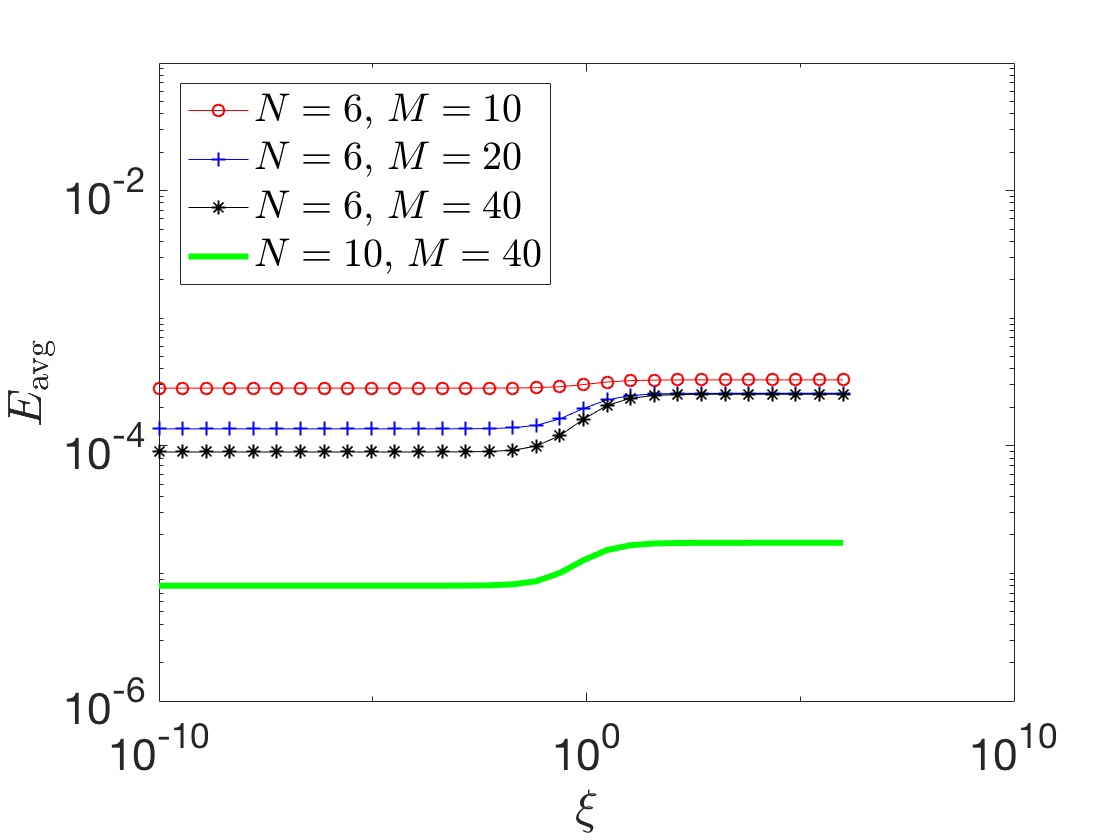}}  
  ~  
\subfloat[biased, ${\rm SNR} = \infty$] {\includegraphics[width=0.3\textwidth]
 {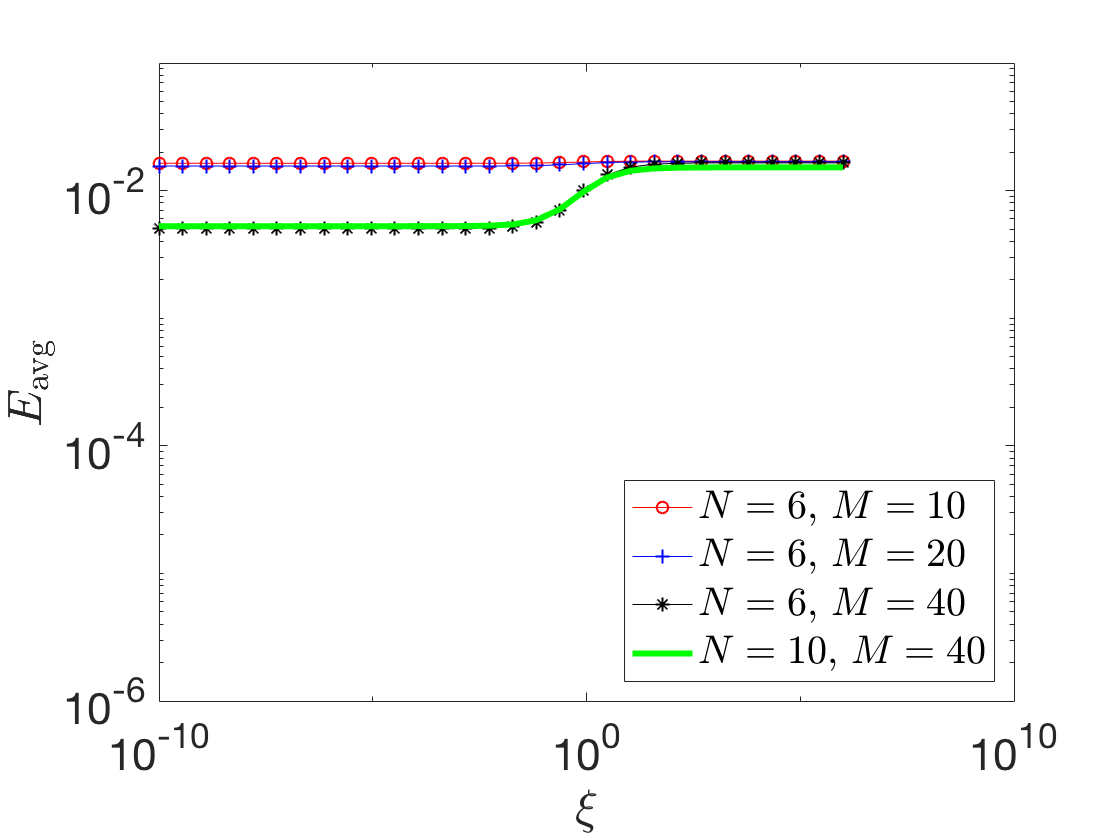}}

 \subfloat[unbiased, ${\rm SNR} = 10$] {\includegraphics[width=0.3\textwidth]
 {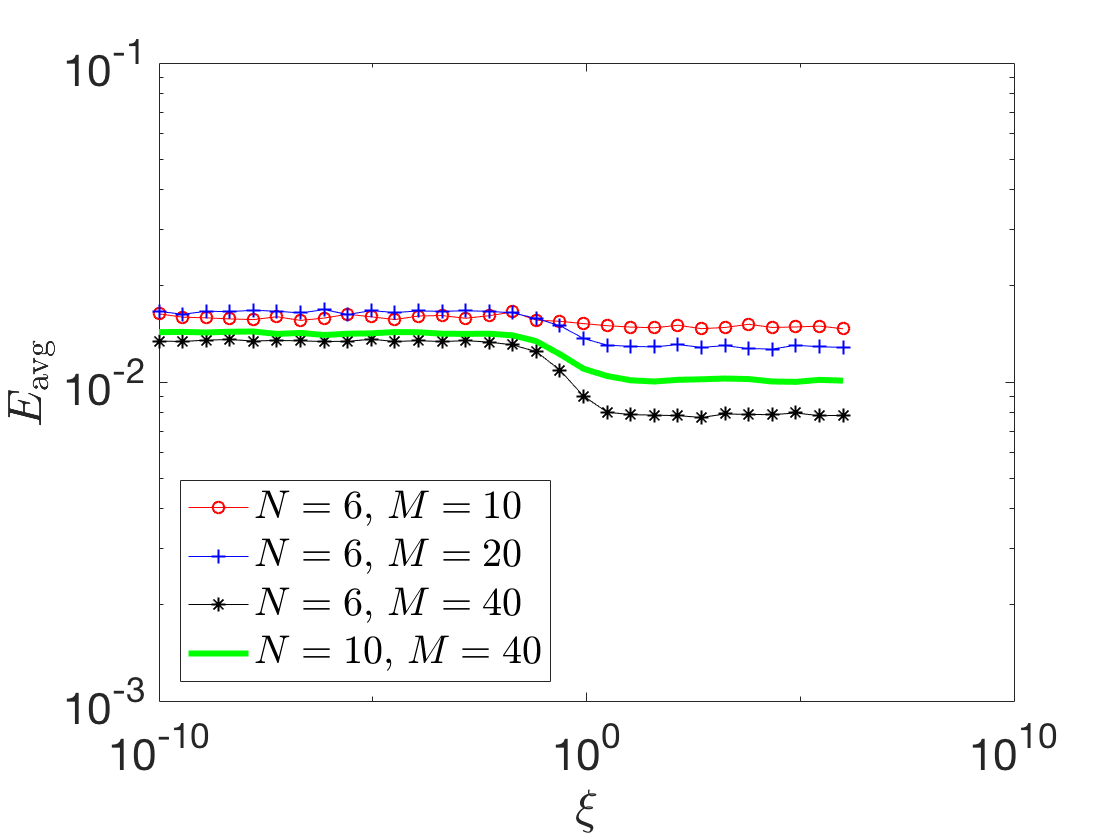}}  
  ~  
\subfloat[biased, ${\rm SNR} = 10$] {\includegraphics[width=0.3\textwidth]
 {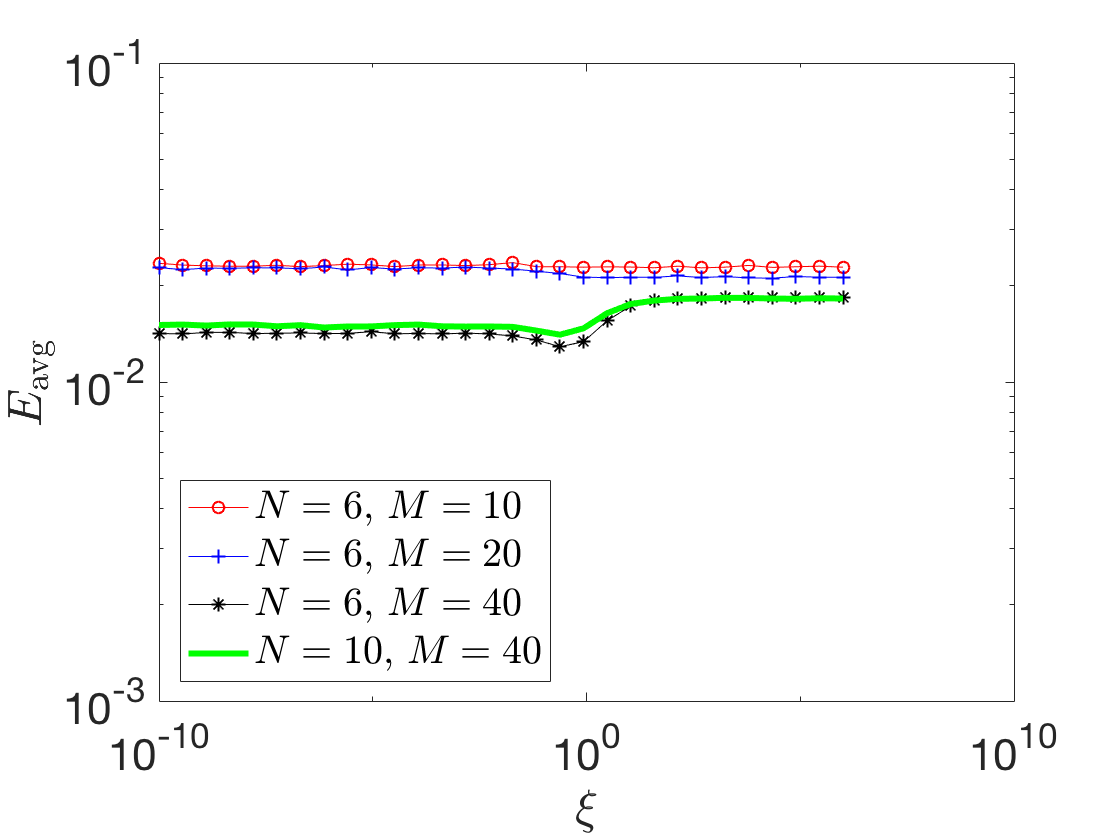}}
 
\caption{two-dimensional problem. Behavior of  $E_{\rm avg}$ for several values of $\xi$ and three choices of $N,M$, for linear PBDW
($\Phi_N = \mathbb{R}^N$).
}
 \label{fig:err_2D_unconstrained_xi}
\end{figure} 

Figure \ref{fig:greedy_2D}(a) shows the first $M=N+5$ points selected by the SGreedy-procedure for $N=5$,
while Figures \ref{fig:greedy_2D}(b) and (c) show
the behavior of the stability constants $\Lambda_{\mathcal{U}}$ and $\Lambda_2$ for different choices of the observation centers and $\xi = 0$.
Figures \ref{fig:greedy_2D}(d)-(e)-(f) show analogous results for $N=15$.
We observe that
for small values of $M$  the SGreedy procedure reduces the constants  $\Lambda_{\mathcal{U}}$ and $\Lambda_2$   compared to a equispaced grid of  observations  and to a grid associated with Gaussian quadrature nodes. We further observe that the application of the variant SGreedy + approximation (with threshold $tol = 0.4$) leads to more moderate values of $\Lambda_2$ compared to the plain SGreedy, at the price of a slight deterioration in $\Lambda_{\mathcal{U}}$.

\begin{figure}[h!]
\centering
\subfloat[$N=5$] {\includegraphics[width=0.3\textwidth]
 {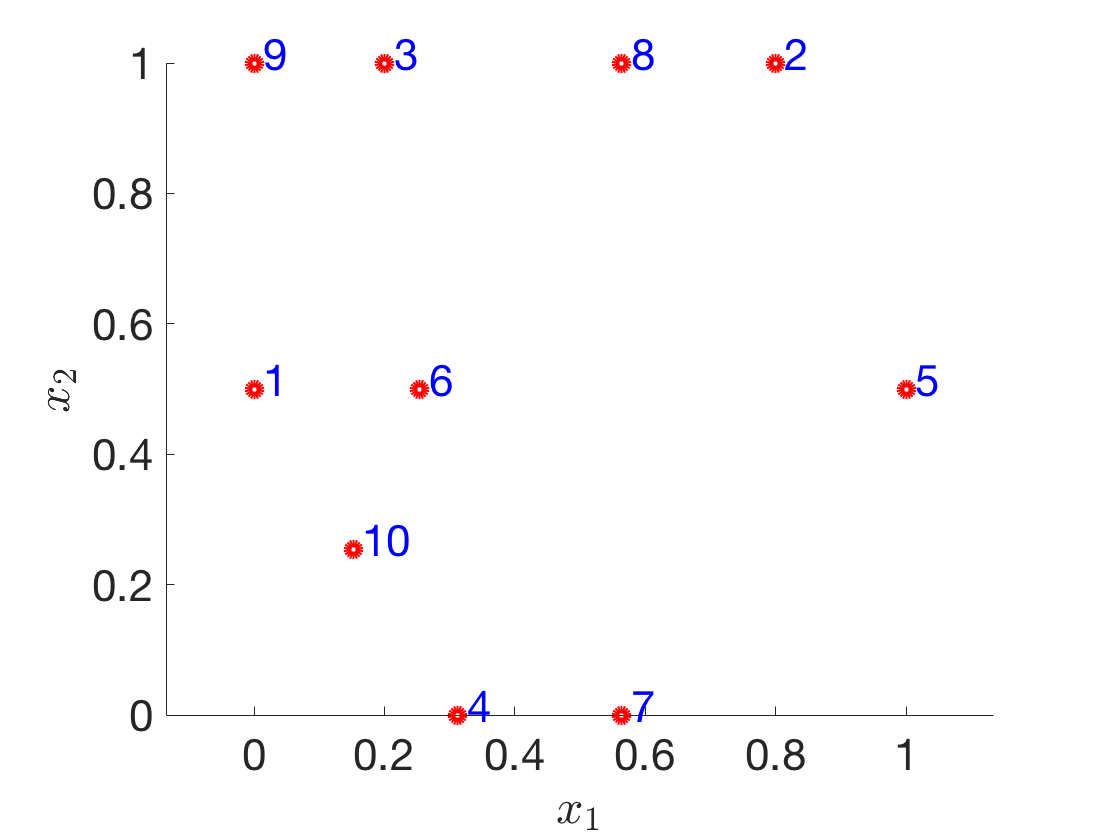}}  
  ~ ~
\subfloat[$N=5$ ] {\includegraphics[width=0.3\textwidth]
 {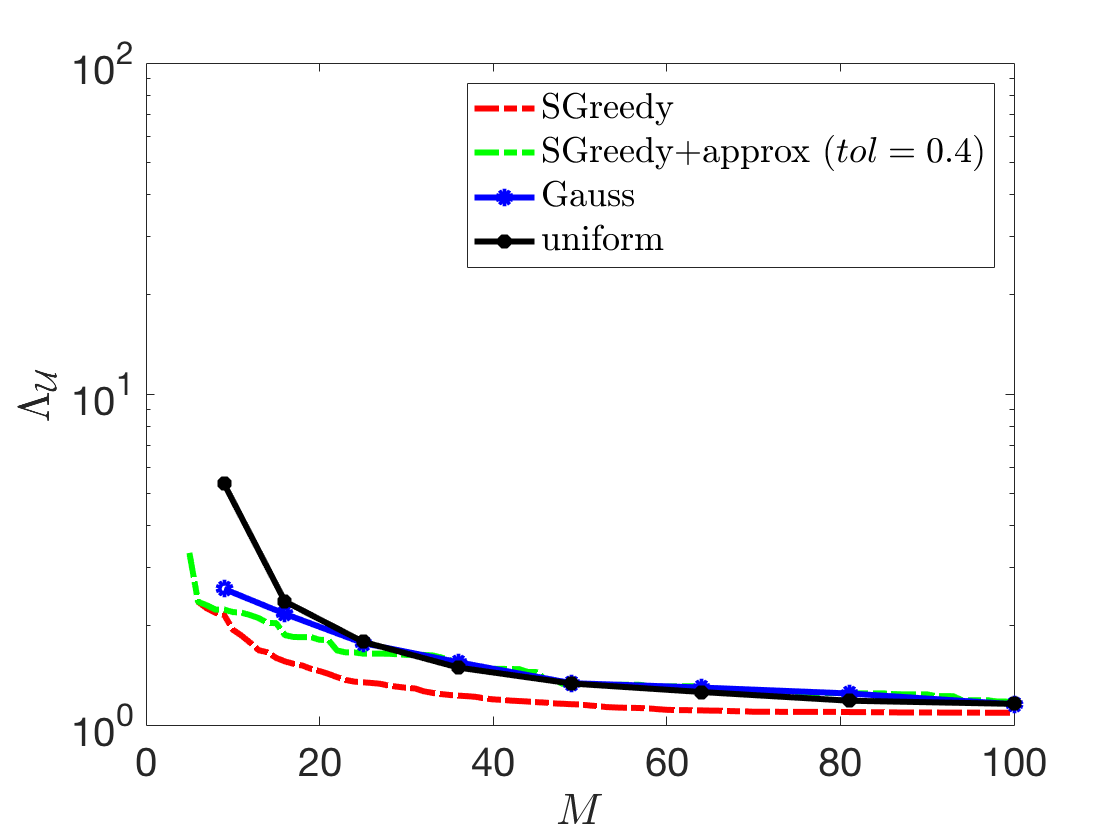}}
   ~ ~
\subfloat[$N=5$ ] {\includegraphics[width=0.3\textwidth]
 {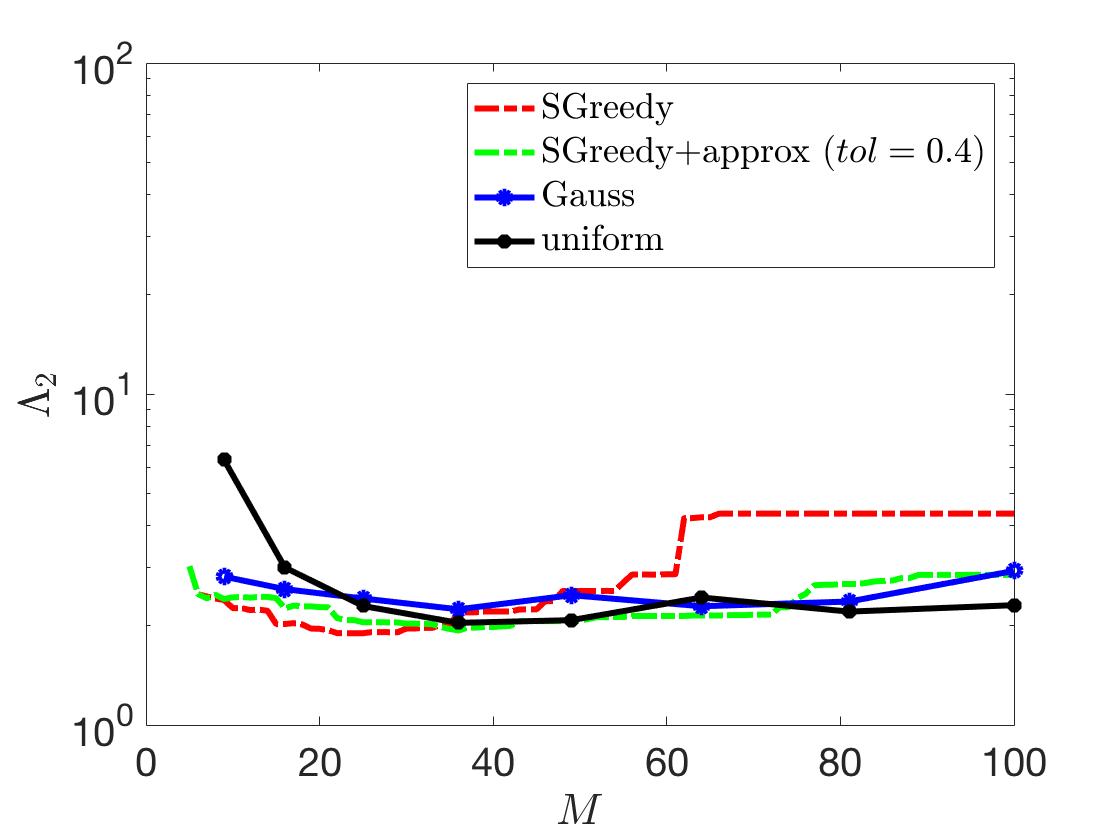}}
 
\subfloat[$N=15$] {\includegraphics[width=0.3\textwidth]
 {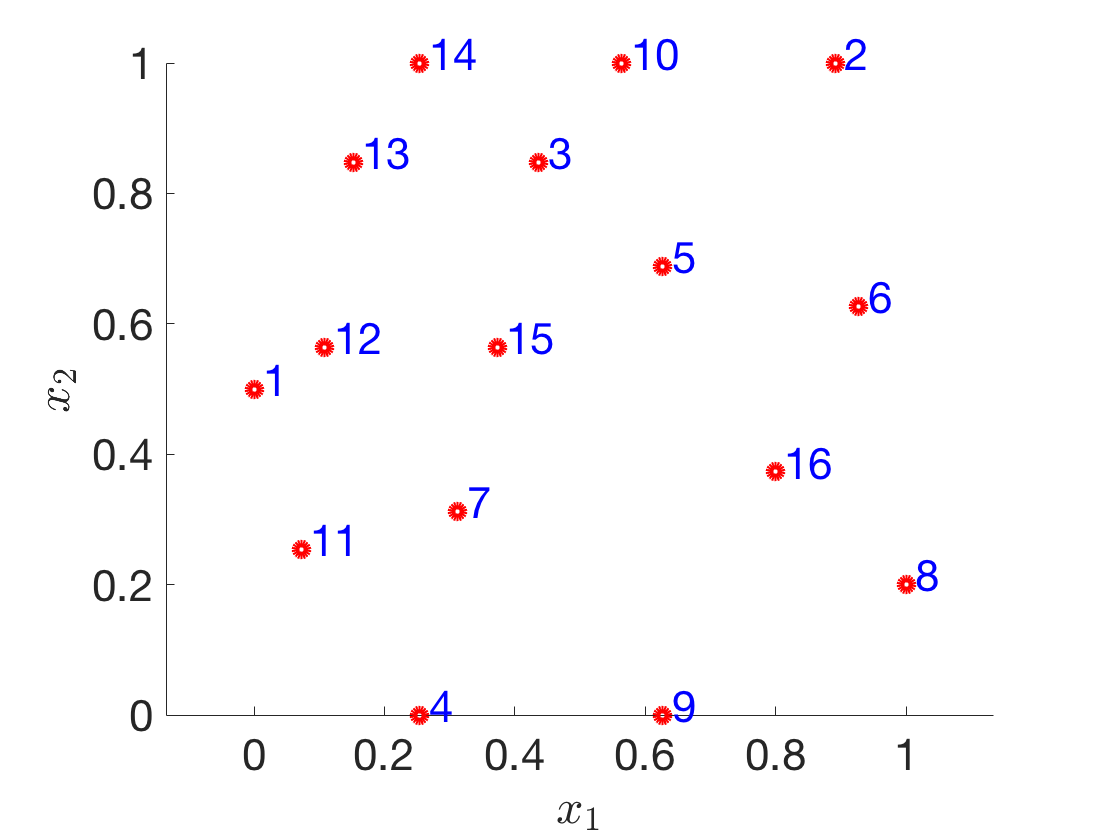}}  
  ~ ~
\subfloat[$N=15$ ] {\includegraphics[width=0.3\textwidth]
 {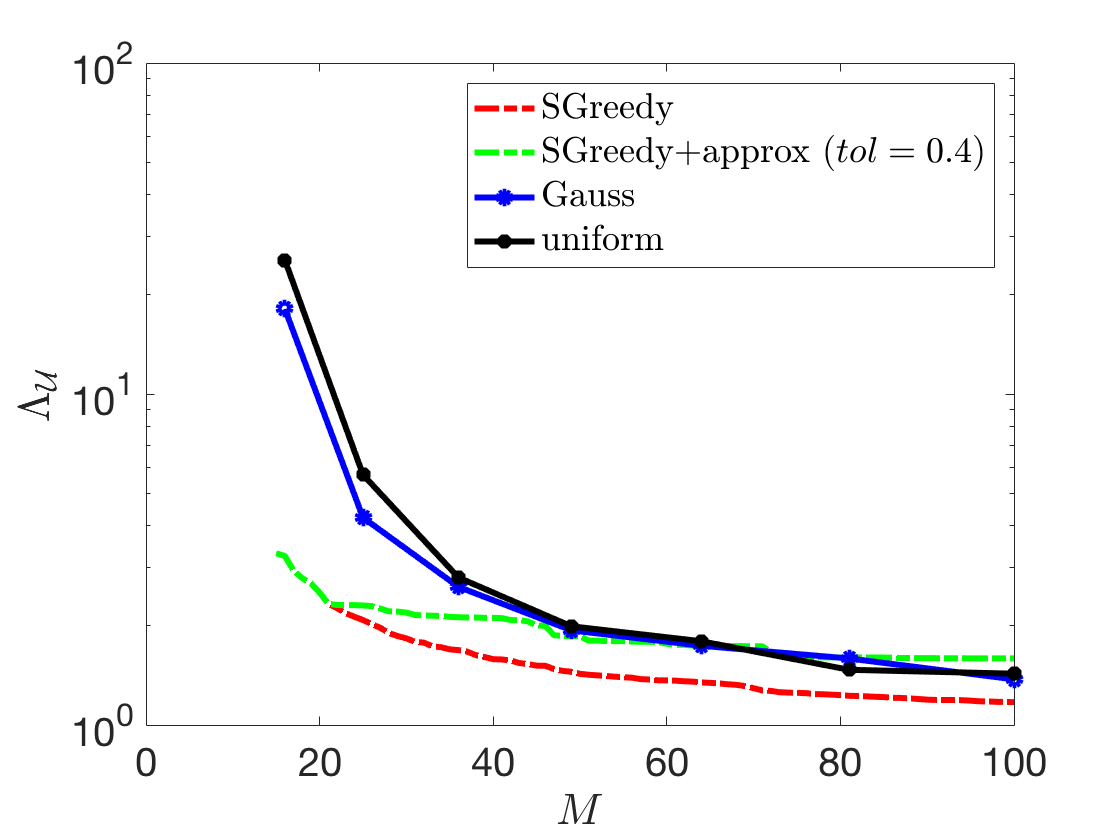}}
   ~ ~
\subfloat[$N=15$ ] {\includegraphics[width=0.3\textwidth]
 {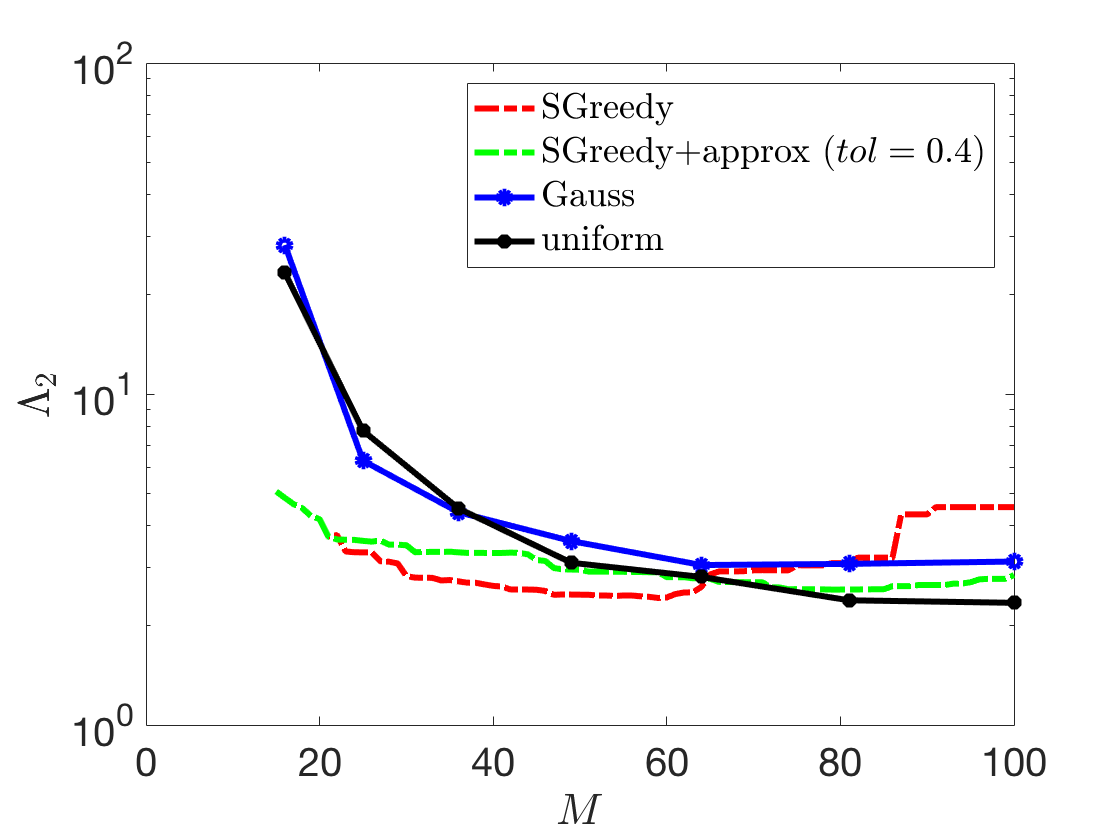}} 
 
\caption{two-dimensional problem. 
Figures (a)-(d): location of the observation centers selected by SGreedy. 
Figures (b)-(e): behavior of $\Lambda_{\mathcal{U}}$ with $M$ for different choices of the observation centers.
Figures (c)-(f): behavior of $\Lambda_{2}$ with $M$ for different choices of the observation centers.
}
 \label{fig:greedy_2D}
\end{figure}

Figure \ref{fig:Nadapt_unconstrained} shows the behavior of $E_{\rm avg}$ defined in \eqref{eq:Eavg} with $N$,
for several choices of $M$. Observations are chosen using the SGreedy+approximation algorithm with threshold $tol = 0.4$. Here, the  value of $\xi$ is   chosen using holdout validation based on $I= M/2$ additional measurements (see \cite{taddei2017adaptive,maday2017adaptive} for further details).
We observe that for noisy measurements the error reaches a minimum for an intermediate  value of $N$, which depends on $M$ and on the noise level. These results are  consistent with the interpretation  --- stated in section \ref{sec:analysis} ---  of $N$ as a regularization parameter. 

\begin{figure}[h!]
\centering
\subfloat[unbiased, ${\rm SNR} = \infty$] {\includegraphics[width=0.3\textwidth]
 {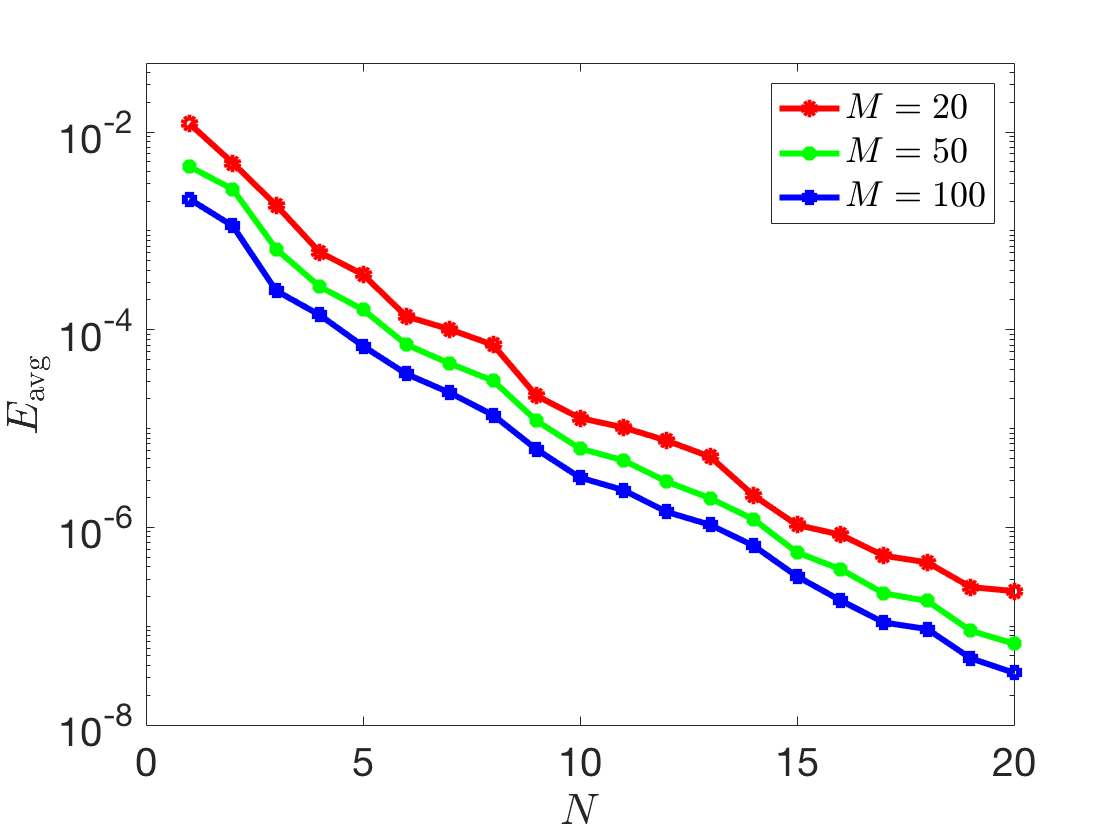}}  
  ~  
\subfloat[biased, ${\rm SNR} = \infty$] {\includegraphics[width=0.3\textwidth]
 {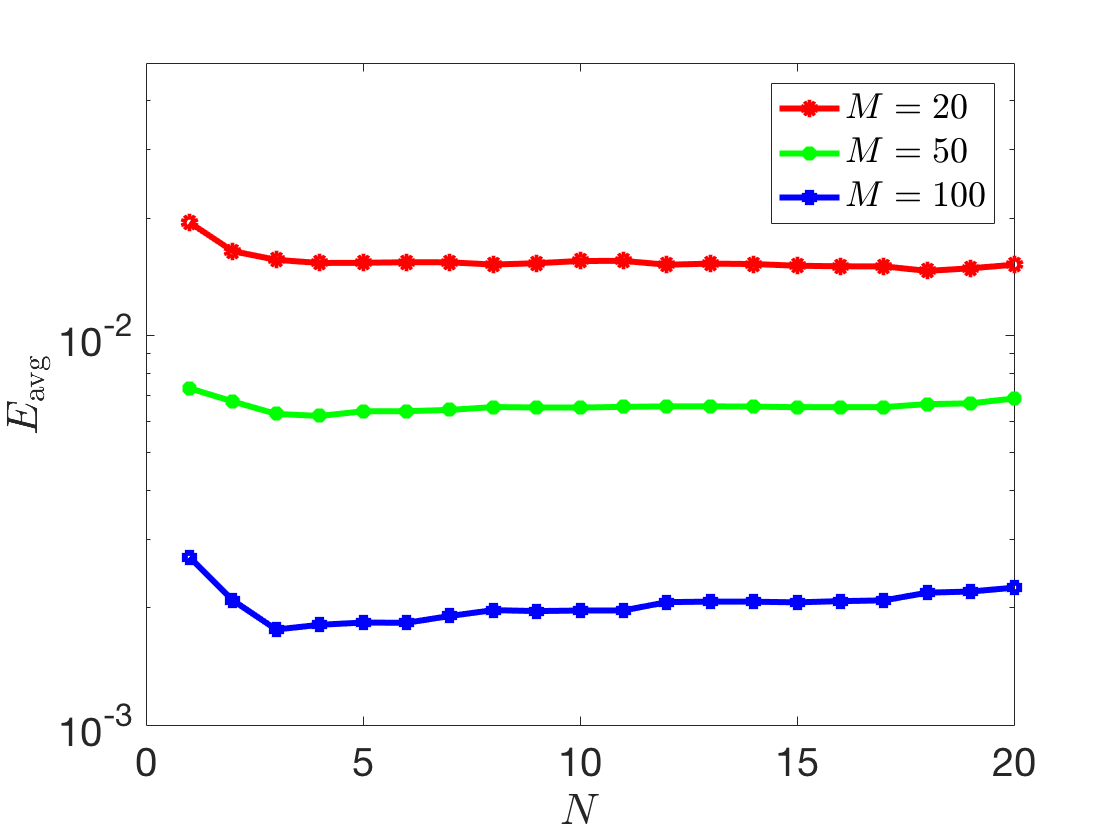}}

 \subfloat[unbiased, ${\rm SNR} = 3$] {\includegraphics[width=0.3\textwidth]
 {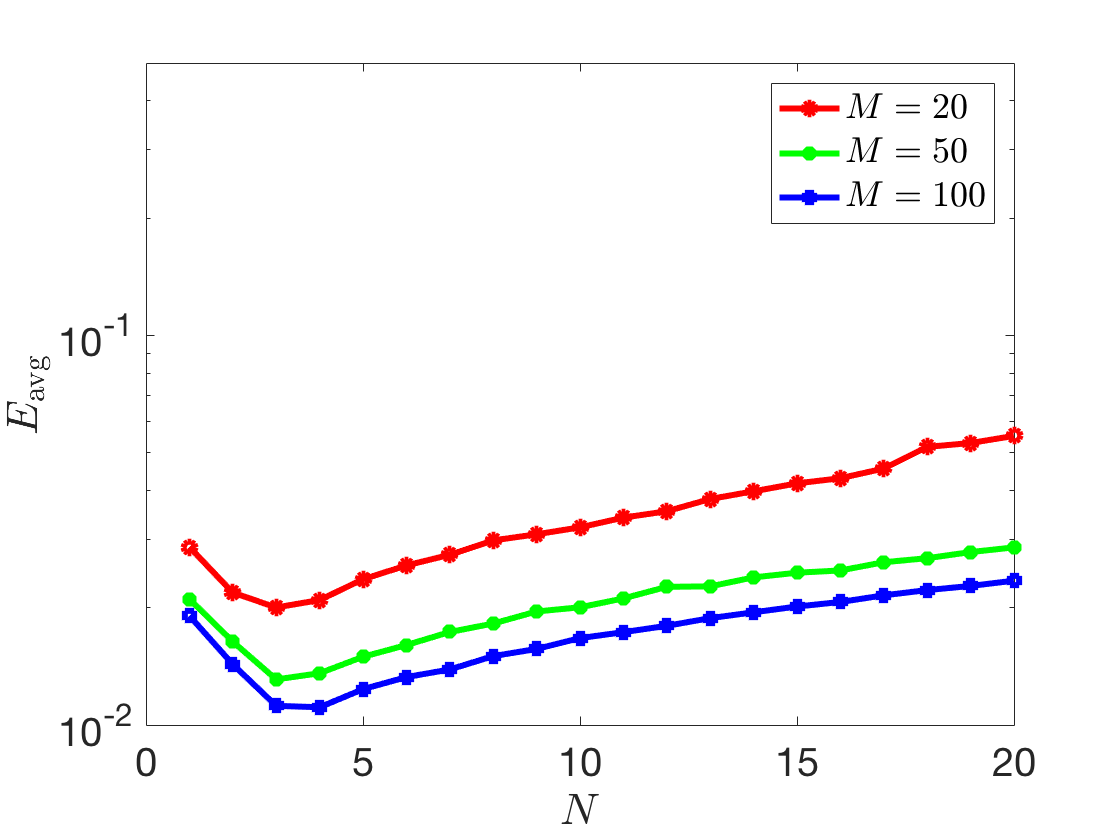}}  
  ~  
\subfloat[biased, ${\rm SNR} = 3$] {\includegraphics[width=0.3\textwidth]
 {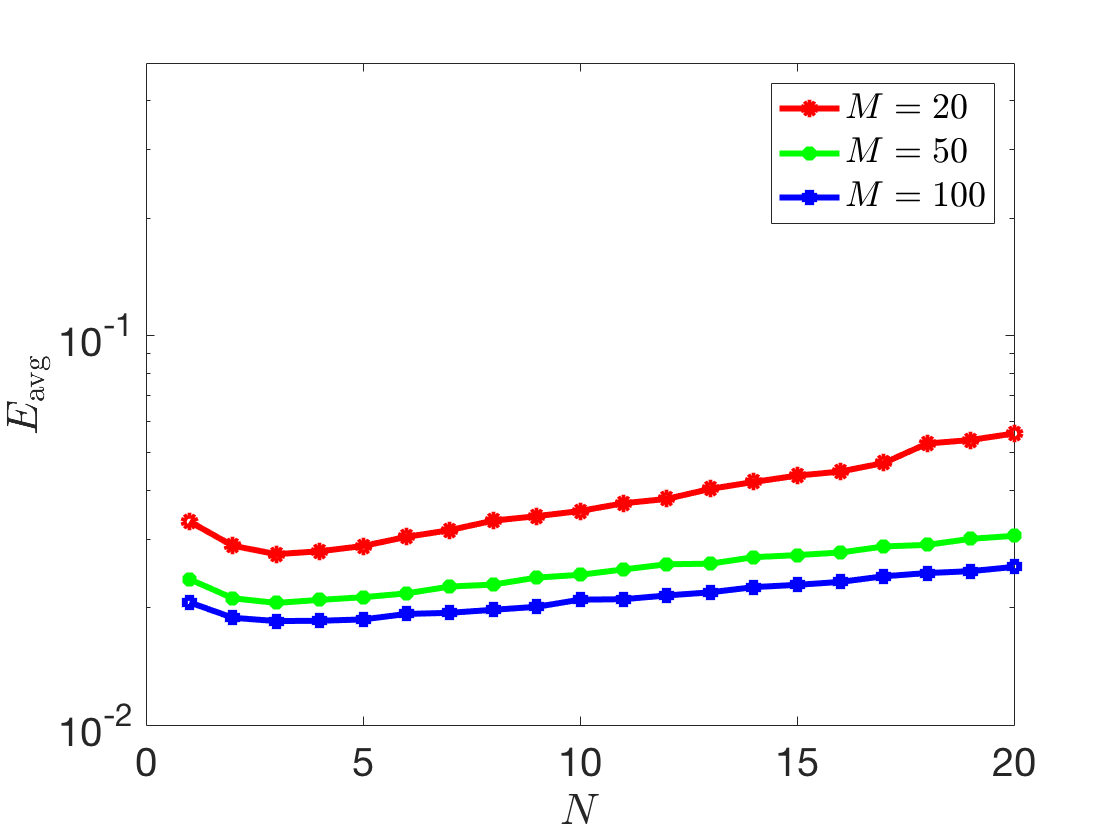}}
 
\caption{two-dimensional problem.
Behavior of $E_{\rm avg}$ with $N$ for several values of $M$, for linear PBDW ($\Phi_N = \mathbb{R}^N$).
}
 \label{fig:Nadapt_unconstrained}
\end{figure}

\subsubsection*{Nonlinear  PBDW}

Figure \ref{fig:Nadapt_constrained} replicates the  experiment of Figure \ref{fig:Nadapt_unconstrained}  for the nonlinear formulation.
We observe that the nonlinear formulation is significantly more robust to the choice of $N$ compared to the linear formulation, particularly in the presence of noise.

\begin{figure}[h!]
\centering
\subfloat[unbiased, ${\rm SNR} = \infty$] {\includegraphics[width=0.3\textwidth]
 {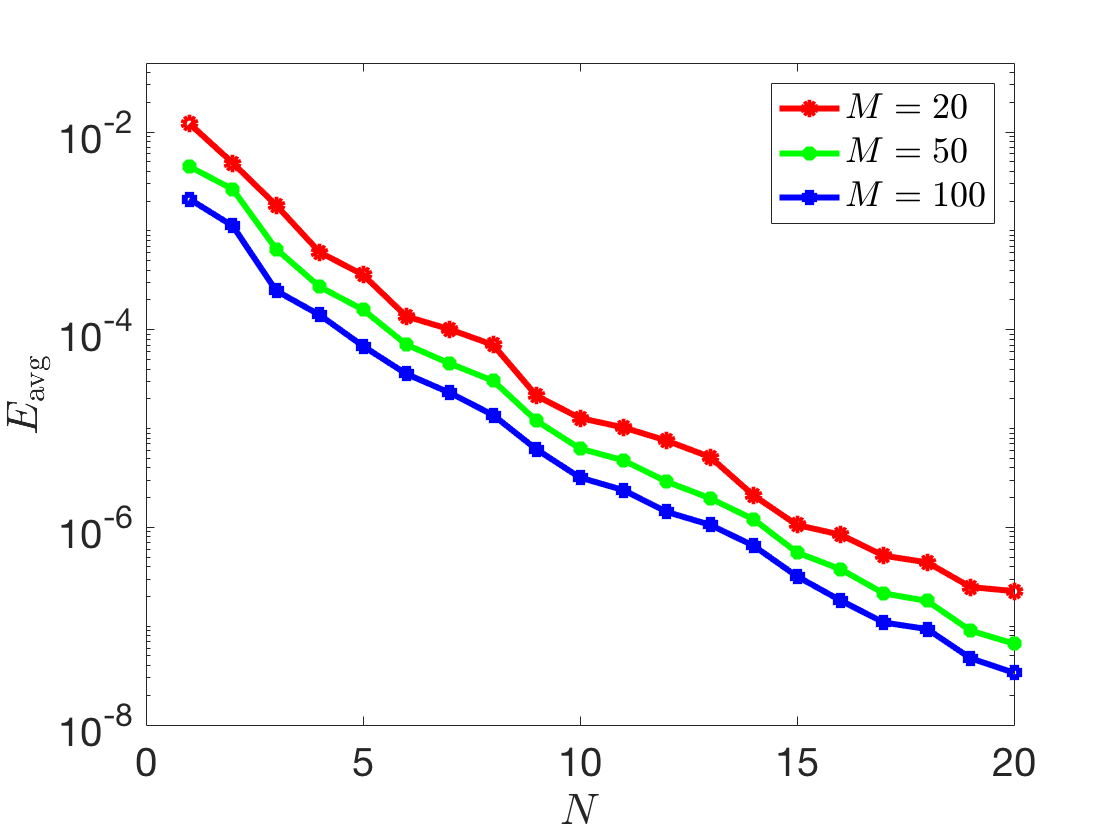}}  
  ~  
\subfloat[biased, ${\rm SNR} = \infty$] {\includegraphics[width=0.3\textwidth]
 {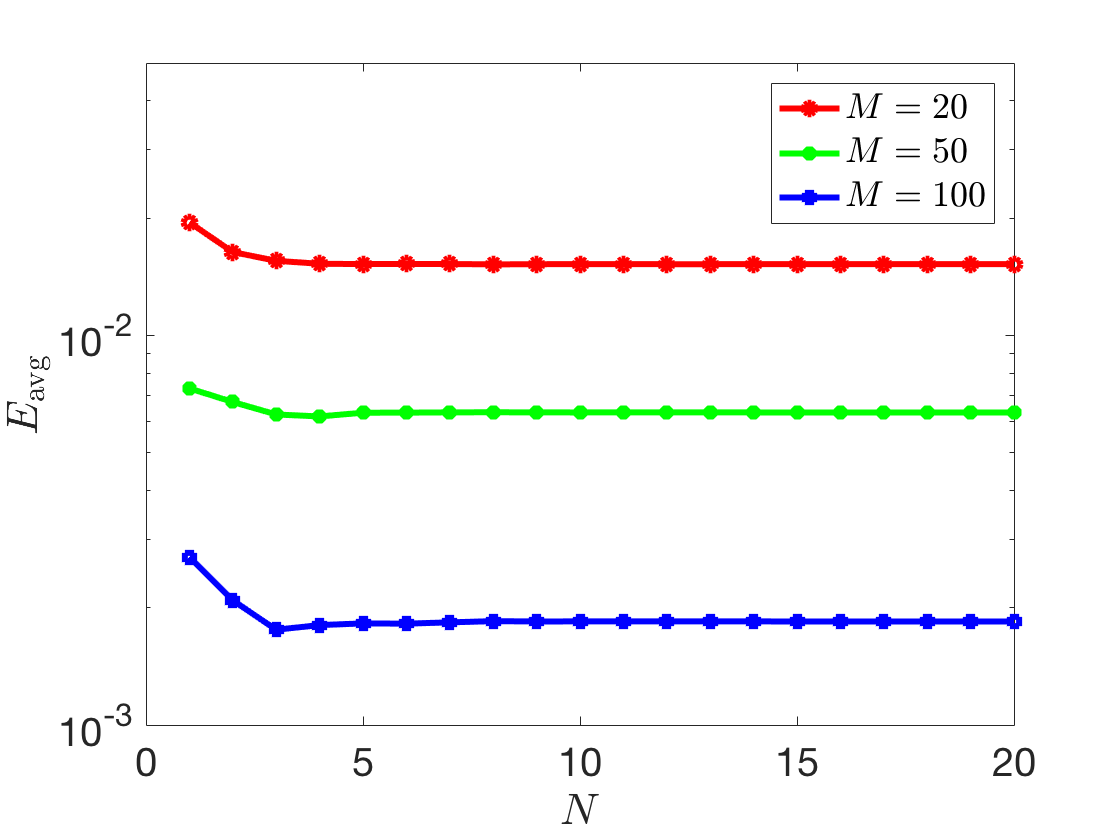}}

 \subfloat[unbiased, ${\rm SNR} = 3$] {\includegraphics[width=0.3\textwidth]
 {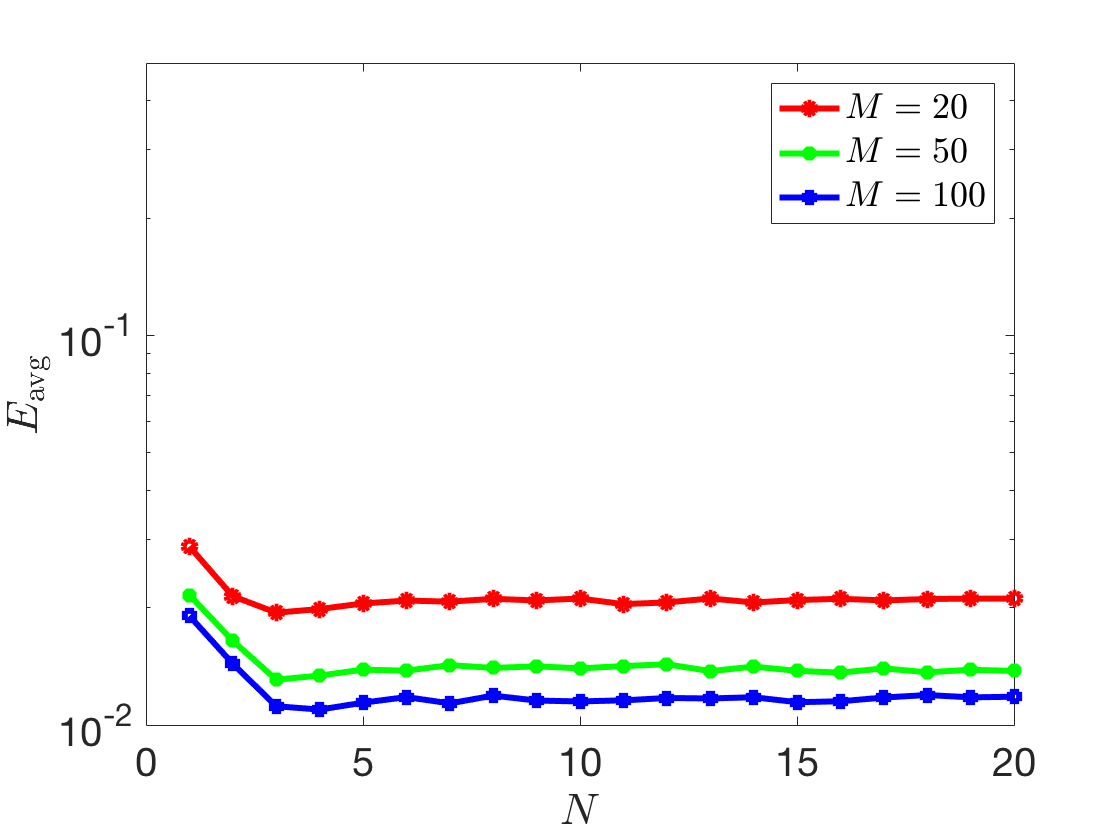}}  
  ~  
\subfloat[biased, ${\rm SNR} = 3$] {\includegraphics[width=0.3\textwidth]
 {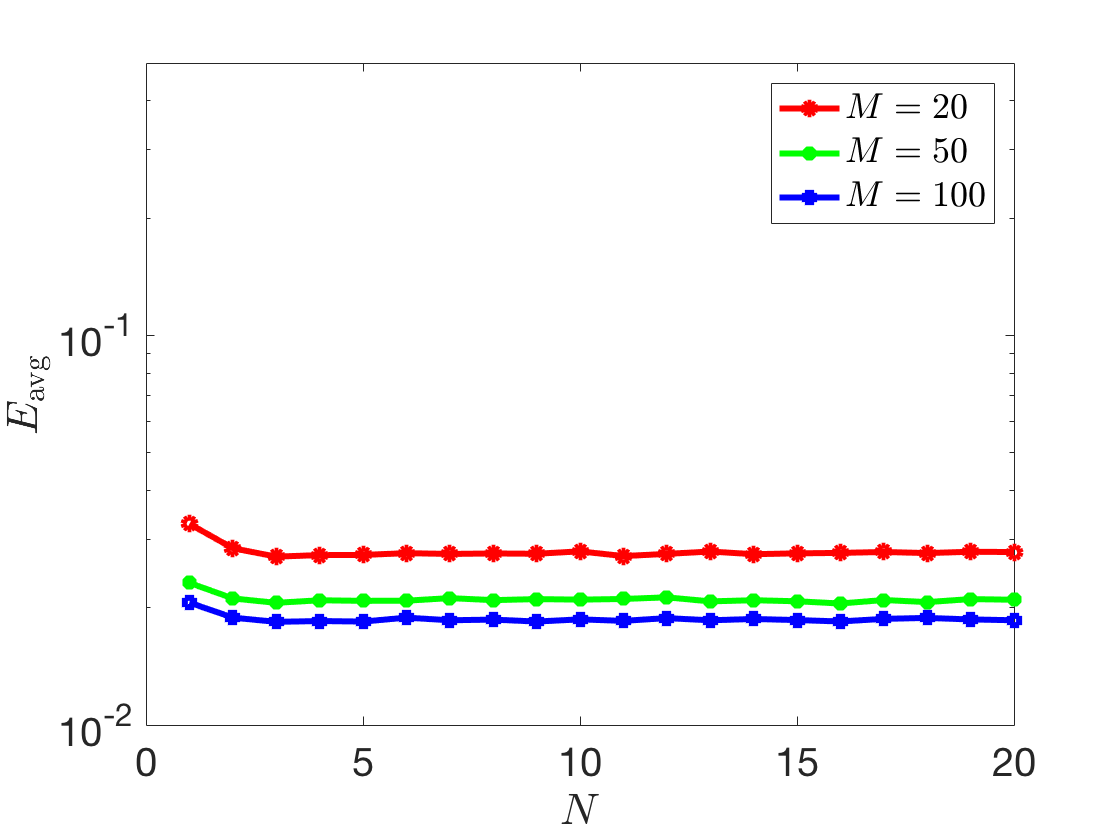}}
 
\caption{two-dimensional problem.
Behavior of $E_{\rm avg}$ with $N$ for several values of $M$, for nonlinear PBDW 
 ($\Phi_N \subsetneq \mathbb{R}^N$).
}
 \label{fig:Nadapt_constrained}
\end{figure} 

In Figures  \ref{fig:Nconv_constrained_unbiased} and \ref{fig:Nconv_constrained_biased}, we investigate the behavior of $E_{\rm avg}$ with $N$, for $M=N+3$, for both linear and nonlinear formulations, and two noise levels. We consider three strategies for the selection of the observation centers: uniform points, Gaussian points, and adaptive points (based on SGreedy). For the problem at hand,  the nonlinear formulation improves reconstruction performance,  particularly in presence of noise and for non-adaptive selections of measurement locations.

\begin{figure}[h!]
\centering
\subfloat[linear, ${\rm SNR} = \infty$] {\includegraphics[width=0.3\textwidth]
 {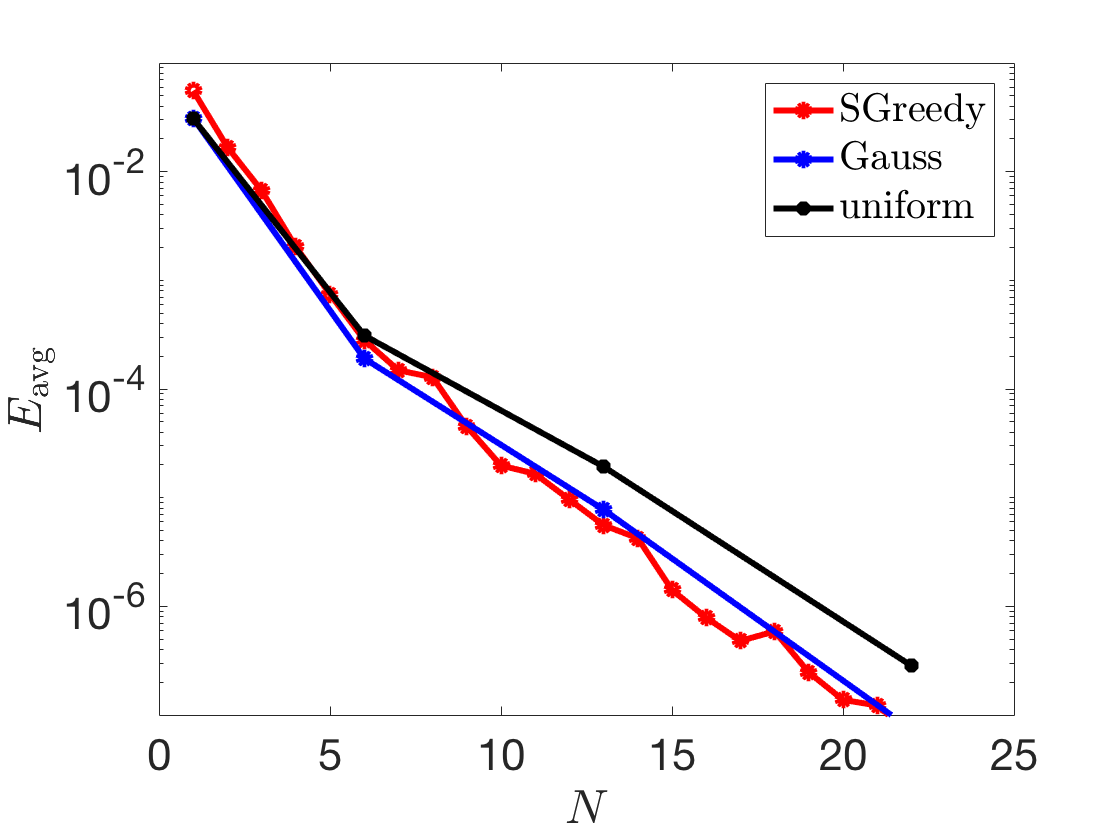}}  
  ~  
\subfloat[nonlinear,  ${\rm SNR} = \infty$] {\includegraphics[width=0.3\textwidth]
 {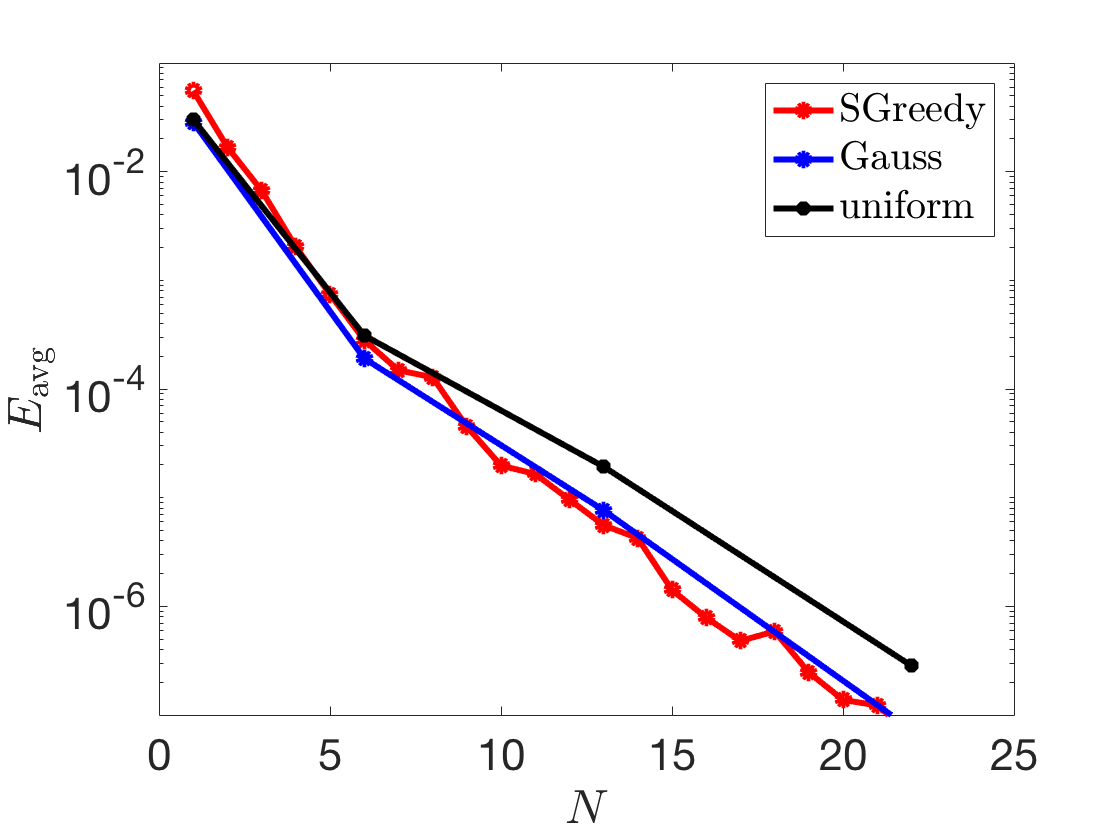}}

\subfloat[linear, ${\rm SNR} = 3$] {\includegraphics[width=0.3\textwidth]
 {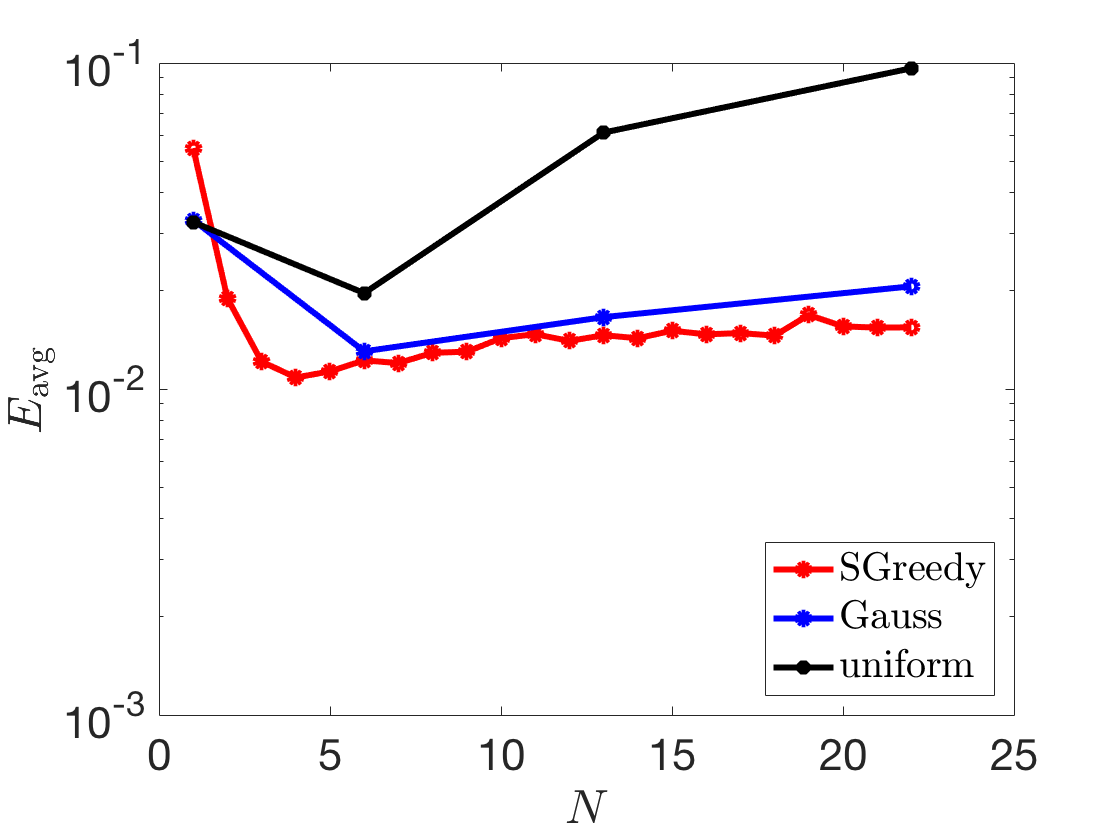}}  
  ~  
\subfloat[nonlinear,  ${\rm SNR} = 3$] {\includegraphics[width=0.3\textwidth]
 {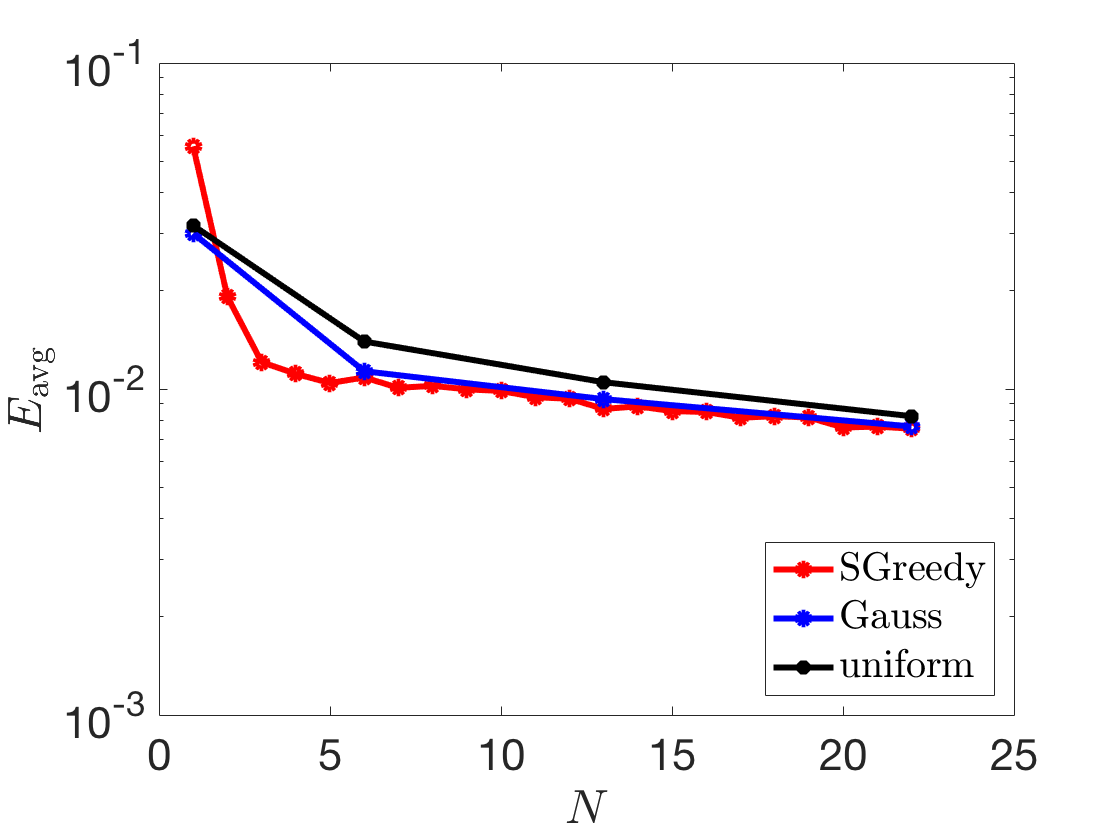}}
 
\caption{two-dimensional problem.
Behavior of $E_{\rm avg}$ with $N$ for  $M=N+3$, for linear and nonlinear PBDW, and several choices of measurement locations (unbiased case).
}
 \label{fig:Nconv_constrained_unbiased}
\end{figure}

\begin{figure}[h!]
\centering
\subfloat[linear, ${\rm SNR} = \infty$] {\includegraphics[width=0.3\textwidth]
 {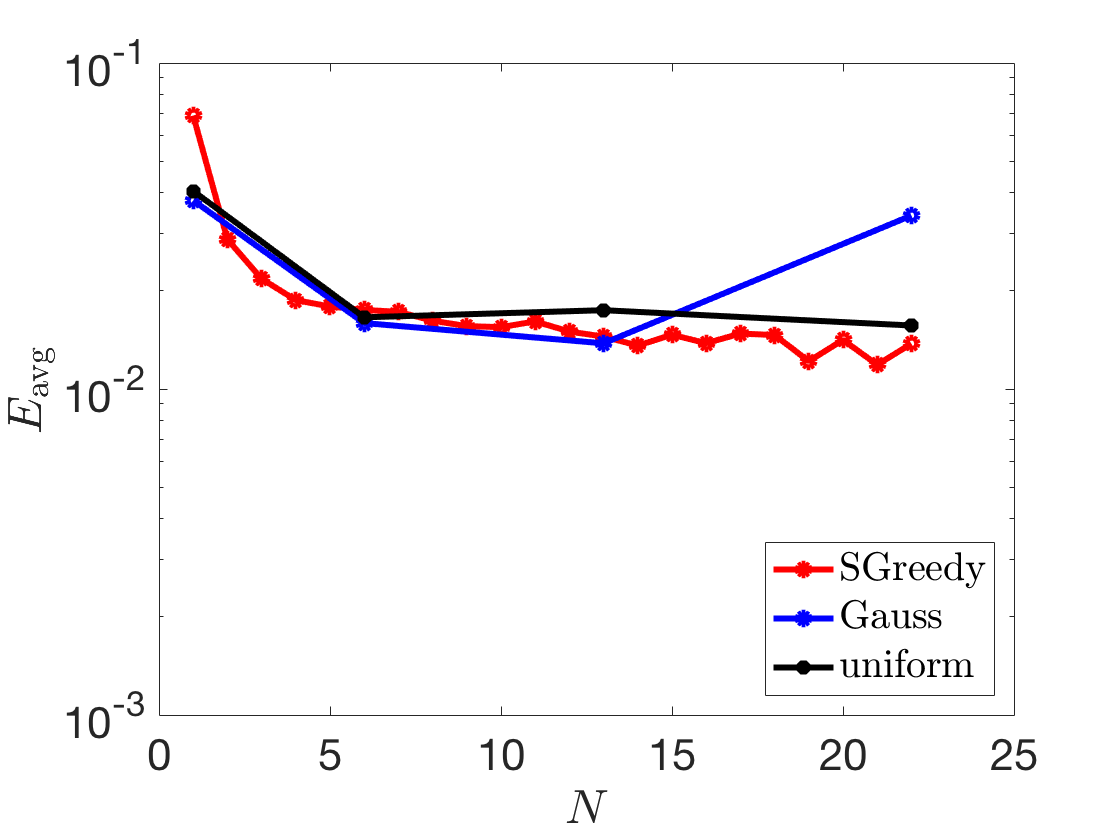}}  
  ~  
\subfloat[nonlinear,  ${\rm SNR} = \infty$] {\includegraphics[width=0.3\textwidth]
 {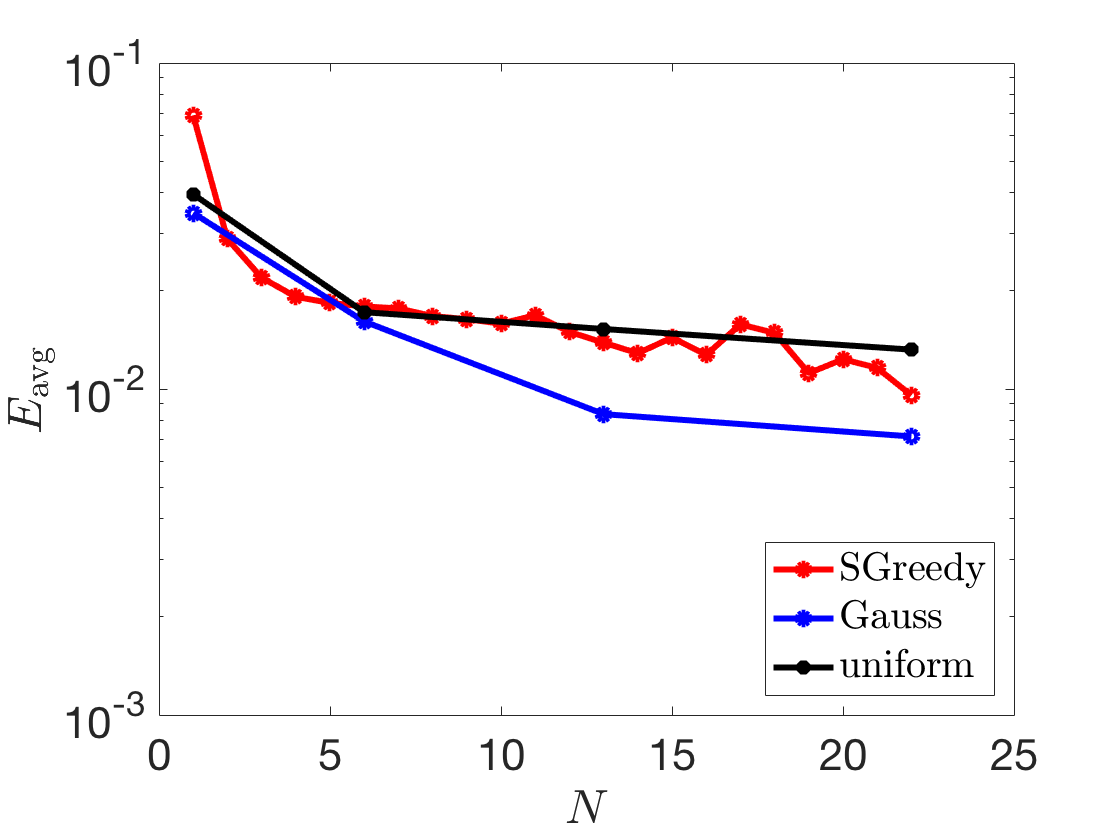}}

\subfloat[linear, ${\rm SNR} = 3$] {\includegraphics[width=0.3\textwidth]
 {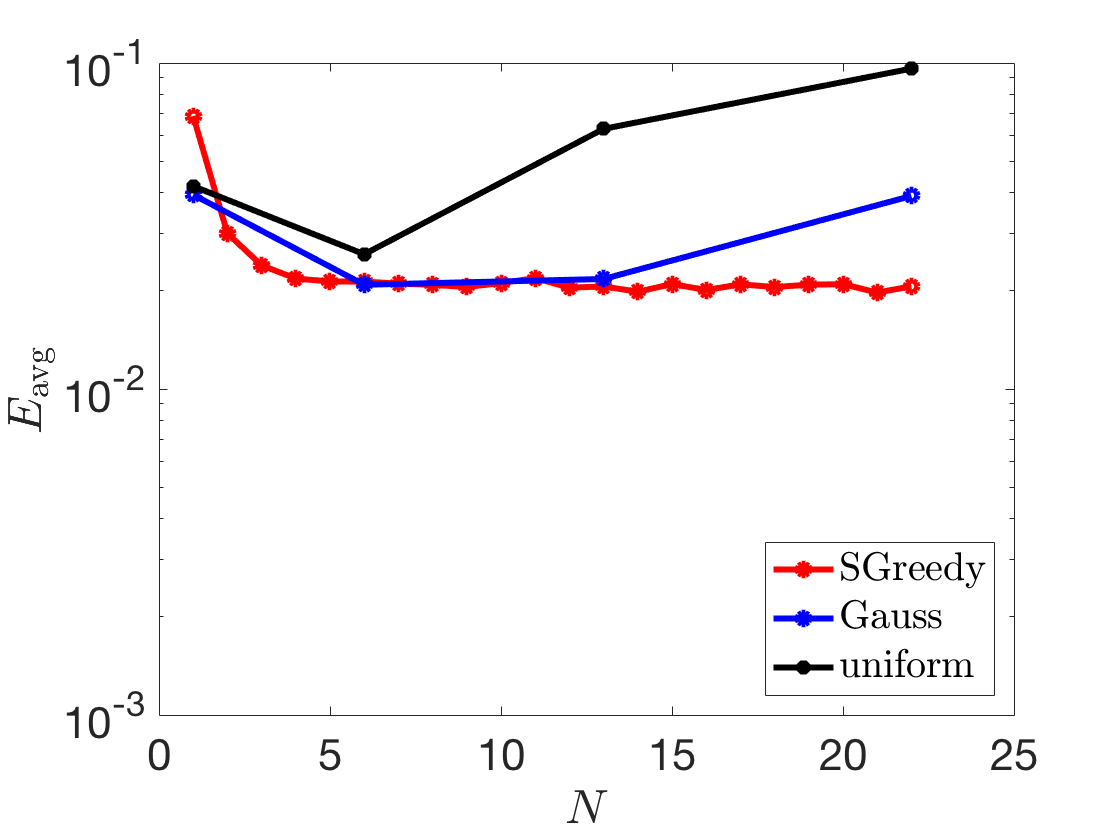}}  
  ~  
\subfloat[nonlinear,  ${\rm SNR} = 3$] {\includegraphics[width=0.3\textwidth]
 {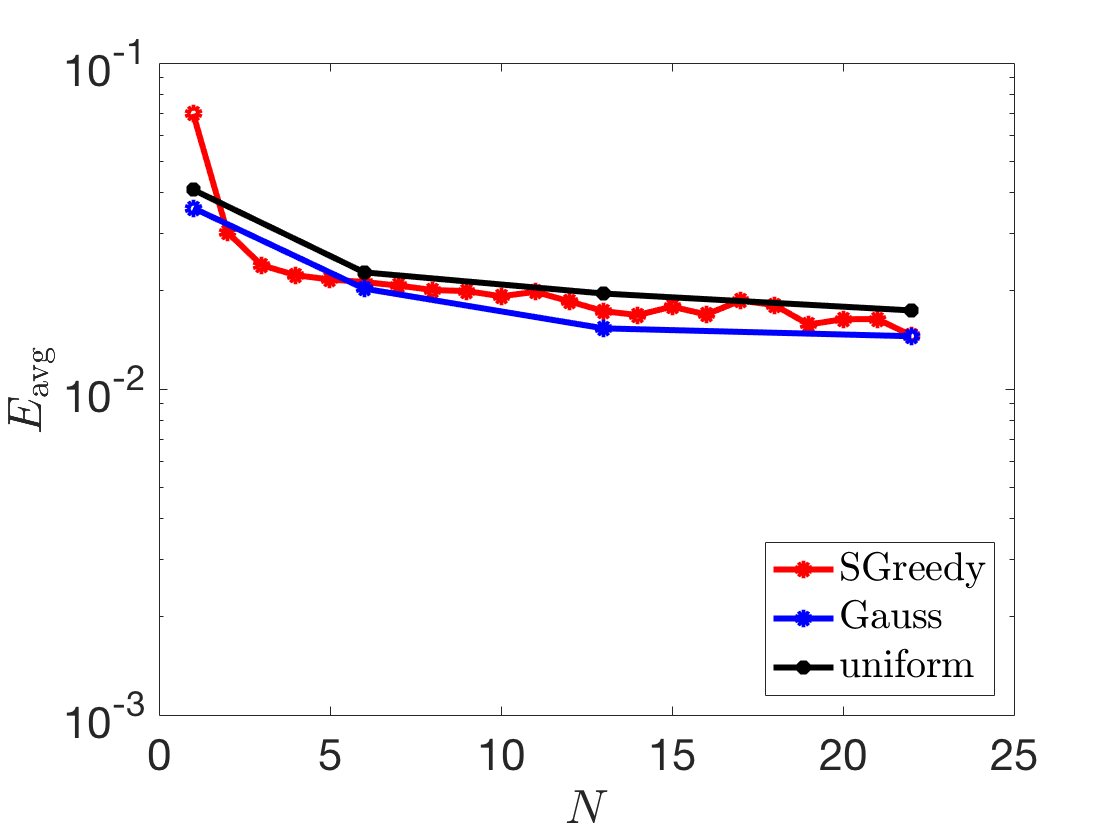}}
 
\caption{two-dimensional problem.
Behavior of $E_{\rm avg}$ with $N$ for  $M=N+3$, for linear and nonlinear PBDW, and several choices of measurement locations (biased case).
}
 \label{fig:Nconv_constrained_biased}
\end{figure}

\subsection{A three-dimensional problem}
\label{sec:3D_acoustics_numerics}

\subsubsection{Problem statement}
We  consider the three-dimensional  model problem:
\begin{equation}
\label{eq:acoustic_3D_exdom}
\left\{
\begin{array}{ll}
 - (1+ \epsilon {\texttt{i}})  \Delta  u_g(\mu)
\,   - (2\pi \mu)^2  u_g (\mu) = g  &  \mbox{in} \; \Omega; \\[3mm]
\partial_n  u_g (\mu)
= 0  &  \mbox{on} \;  \partial   \Omega; \\
\end{array}
\right.
\end{equation}
where $\epsilon=10^{-2}$,   $\Omega  = (-1.5,1.5) \times (0,3) \times (0,3) \setminus  \Omega^{\rm cut}$,
$\Omega^{\rm cut} = (-0.5,0.5) \times (0.25,0.5) \times (0,1)$.
Figure \ref{fig:acoustic_3D_extracted_dom} shows the geometry; the same test case has been considered in \cite{maday2017adaptive}. 
We define the bk manifold $\mathcal{M}^{\rm bk} = 
\{ u^{\rm bk}(\mu)
=u_{g^{\rm bk}}(\mu): \mu \in \mathcal{P}^{\rm bk} = [0.1,0.5] \}$, and we define the true field as the solution
to \eqref{eq:acoustic_3D_exdom} for some  $\mu^{\rm true} \in \mathcal{P}^{\rm bk}$
and $g=g^{\rm true}$, where 
$$
g^{\rm bk}(x) = 10 \, e^{ - \| x - p^{\rm bk} \|_2^2};
\quad
g^{\rm true}(x) = 10 \, e^{ - \| x - p^{\rm true} \|_2^2};
$$
and $p^{\rm bk} = [0,2,1]$, $p^{\rm true} = [-0.02,2.02,1]$.
{
Lack of knowledge of}  the input frequency $\mu$ constitutes the anticipated {  ignorance} in the system,
while the   incorrect location of the acoustic source (that is,  $p^{\rm bk} \neq p^{\rm true}$) constitutes 
unanticipated {  ignorance/model error} .
Computations are based on a P2 Finite Element (FE) discretization with roughly $\mathcal{N} = 16000$ degrees of freedom in $\Omega$. 

\begin{figure}[h]
\centering
\subfloat[   ] {\includegraphics[width=0.3\textwidth]
 {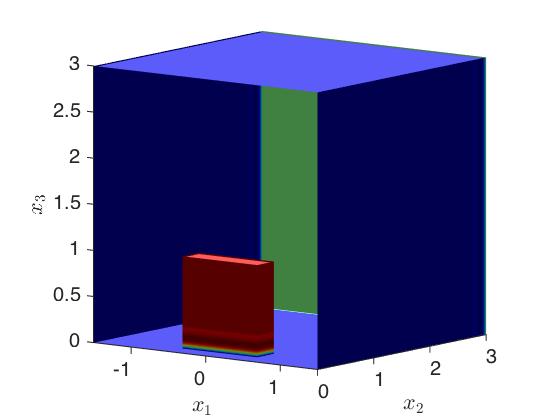}}
~~~~~
\subfloat[] {
\begin{tikzpicture}[x=1.0cm,y=1.0cm]
\linethickness{0.3 mm}
\linethickness{0.3 mm}
\draw  (0,0)--(0.25,0)--(0.25,1)--(0.5,1)--(0.5,0)-- (3,0)-- (3,3)-- (0,3)-- (0, 0);

\coordinate [label={center:  {\large {$\Omega$} }}] (E) at (1.3,2.3) ;
\coordinate [label={center:  {\large {$\Omega^{\rm cut}$} }}] (E) at (1.1,0.8) ;
\coordinate [label={center:  {\large {$x_2$}}}] (E) at (1.4, -0.3) ;
\coordinate [label={center:  {\large {$x_3$}}}] (E) at (3.3, 1.5) ;
\end{tikzpicture}
}
~~~~~
\subfloat[] {
\begin{tikzpicture}[x=1.0cm,y=1.0cm]
\linethickness{0.3 mm}
\linethickness{0.3 mm}
\draw  (-1.5,0)--(1.5,0)--(1.5,3)--(-1.5,3)--(-1.5,0);
\draw[densely dashed]     (-0.5,0)-- (-0.5,1)-- (0.5,1)-- (0.5,0);

\coordinate [label={center:  {\large {$\Omega^{\rm cut}$} }}] (E) at (0,0.4) ;
\coordinate [label={center:  {\large {$\Omega$}}}] (E) at (0, 2.3) ;

\coordinate [label={center:  {\large {$x_1$}}}] (E) at (0, -0.3) ;
\coordinate [label={center:  {\large {$x_3$}}}] (E) at (1.8, 1.5) ;
\end{tikzpicture}
}
\caption{three-dimensional problem: computational domain. }
\label{fig:acoustic_3D_extracted_dom}
\end{figure}
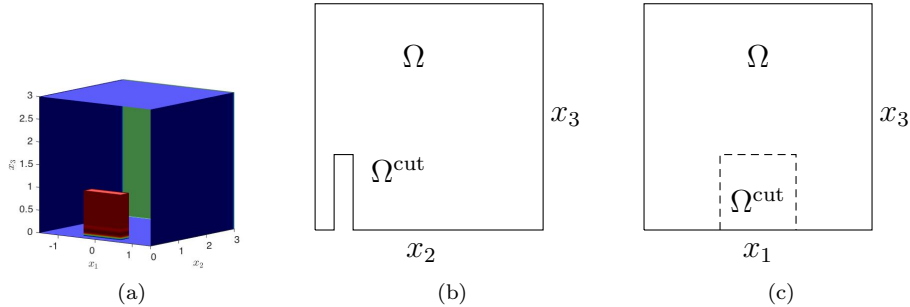

As in the previous example, we model the synthetic observations by a Gaussian convolution with standard deviation $r_{\rm w}$, see \eqref{eq:exp_observations}.
In order to simulate noisy observations, we add Gaussian homoscedastic disturbances
$$
y_m = \ell_m( u^{\rm true} ) + \epsilon_m^{\rm re} + {\rm i} \epsilon_m^{\rm im},
\quad
{\rm where} \; \;
\epsilon_m^{\rm re} \overset{\rm iid}{\sim} \mathcal{N}(0, \sigma_{\rm re}^2),  \; \;
\epsilon_m^{\rm im} \overset{\rm iid}{\sim} \mathcal{N}(0, \sigma_{\rm im}^2), 
$$
with 
$\sigma_{\rm re} = \frac{1}{\rm SNR} \; {\rm std} \left( \{ {\rm Re} \{
\ell \left( u^{\rm true}; \tilde{x}_j, \, r_{\rm w} \right) \}
  \}_{j=1}^{100}  \right),$
$   \sigma_{\rm im} = \frac{1}{\rm SNR} \; {\rm std} \big( \{ {\rm Im} \{ \ell \big( u^{\rm true}; $ 
$\tilde{x}_j, \, r_{\rm w} \big)$ $\} \}_{j=1}^{100}  \big),$
for given signal-to-noise ratio ${\rm SNR}>0$ and uniformly-randomly chosen  observation points
$\{  \tilde{x}_j \}_j \subset  \Omega$.
Furthermore, we measure performance by computing the average relative
$L^2$    error   $E_{\rm avg}$ \eqref{eq:Eavg}  over  $n_{\rm test}=10$  different choices of the parameter $\mu$ in $\mathcal{P}^{\rm bk}$.

We consider the ambient space $\mathcal{U}=H^1(\Omega)$ endowed with the inner product 
$$
(u, v) =  \int_{\Omega} \nabla u \cdot \nabla \bar{v} \, + \, u \, \bar{v} \, dx,
$$  
where $\bar{z}$ denotes the complex conjugate of $z \in \mathbb{C}$.
On the other hand, the background space $\mathcal{Z}_N$ is built using the Weak-Greedy algorithm based on the residual, as in \cite{maday2015parameterized}.

Since the solution is complex, we 
compute the solution to the nonlinear formulation by 
solving  the complex-valued counterpart 
of \eqref{eq:algebraic_pbdw_constrained} as 
a $2N$-dimensional real-valued quadratic problem for 
$\widehat{\mathbf{z}}_{\xi}^{\star} = 
[{\rm Re} \{ \hat{\mathbf{z}}_{\xi} \},    
{\rm Im} \{ \hat{\mathbf{z}}_{\xi} \}]$ (see Remark \ref{remark_complex}).
Given $N \leq N_{\rm max} := 30$, 
we estimate the constraints $\{ a_n, b_n \}_{n=1}^{2N}$   by evaluating a Galerkin Reduced Order Model (ROM) based on the reduced space $\mathcal{Z}_{N_{\rm max}}$ for $n_{\rm train} = 10^3$ parameters
$\mu^1,\ldots, \mu^{n_{\rm train}}  \overset{\rm iid}{\sim} {\rm Uniform}(\mathcal{P}^{\rm bk})$:
$$
\left\{
\begin{array}{l}
\displaystyle{
a_n := \min_{i=1,\ldots,n_{\rm train}} 
{\rm Re} \left\{
\left( u_{N_{\rm max}}^{\rm bk}(\mu^i), \, \zeta_n \right)  \right\},
}
\\[3mm]
\displaystyle{
b_n := \max_{i=1,\ldots,n_{\rm train}} 
{\rm Re} \left\{
\left( u_{N_{\rm max}}^{\rm bk}(\mu^i), \, \zeta_n \right) \right\},
}
\\[3mm]
\displaystyle{
a_{n+N} := \min_{i=1,\ldots,n_{\rm train}}  \,
{\rm Im} \left\{ \left( u_{N_{\rm max}}^{\rm bk}(\mu^i), \, \zeta_n \right) \right\},
}
\\[3mm]
\displaystyle{
b_{n+N} := \max_{i=1,\ldots,n_{\rm train}} \,
{\rm Im} \left\{  \left( u_{N_{\rm max}}^{\rm bk}(\mu^i), \, \zeta_n \right)   \right\},
}
\\
\end{array}
\right.
\quad
n=1,\ldots,N.
$$
Here $u_{N_{\rm max}}^{\rm bk}$ denotes the solution to the Galerkin ROM with 
$N_{\rm max}$ degrees of freedom. We refer to the Reduced Basis literature for further details concerning the generation and the evaluation of the ROM; we emphasize  that by resorting to the low-dimensional  ROM 
--- as opposed to the FE model --- 
to estimate lower and upper bounds for the background coefficients we significantly reduce the offline computational effort.

\subsubsection{Results}

Figure \ref{fig:constants3D} shows the behavior of $\Lambda_2$ and  $\Lambda_{\mathcal{U}}$ with $N$ for $M=N+2$   and   $M=N+5$, for two choices of the measurement locations: the SGreedy adaptive algorithm, and a random uniform algorithm in which $x_1,\ldots, x_M$ are sampled uniformly in $\Omega$ with the constraint that 
${\rm dist} \left( x_m, \, \{ x_{m'}  \}_{m'=1}^{m-1}  \right) \geq \delta = 0.02$,
for $m=1,\ldots,M$.
Results for the latter procedure are averaged over $K=10$ independent random choices of measurement locations.
We observe that the Greedy algorithm leads to a reduction in both  $\Lambda_2$ and  $\Lambda_{\mathcal{U}}$ compared to the random uniform algorithm.

\begin{figure}[h!]
\centering
\subfloat[ $M=N+2$] {\includegraphics[width=0.3\textwidth]
 {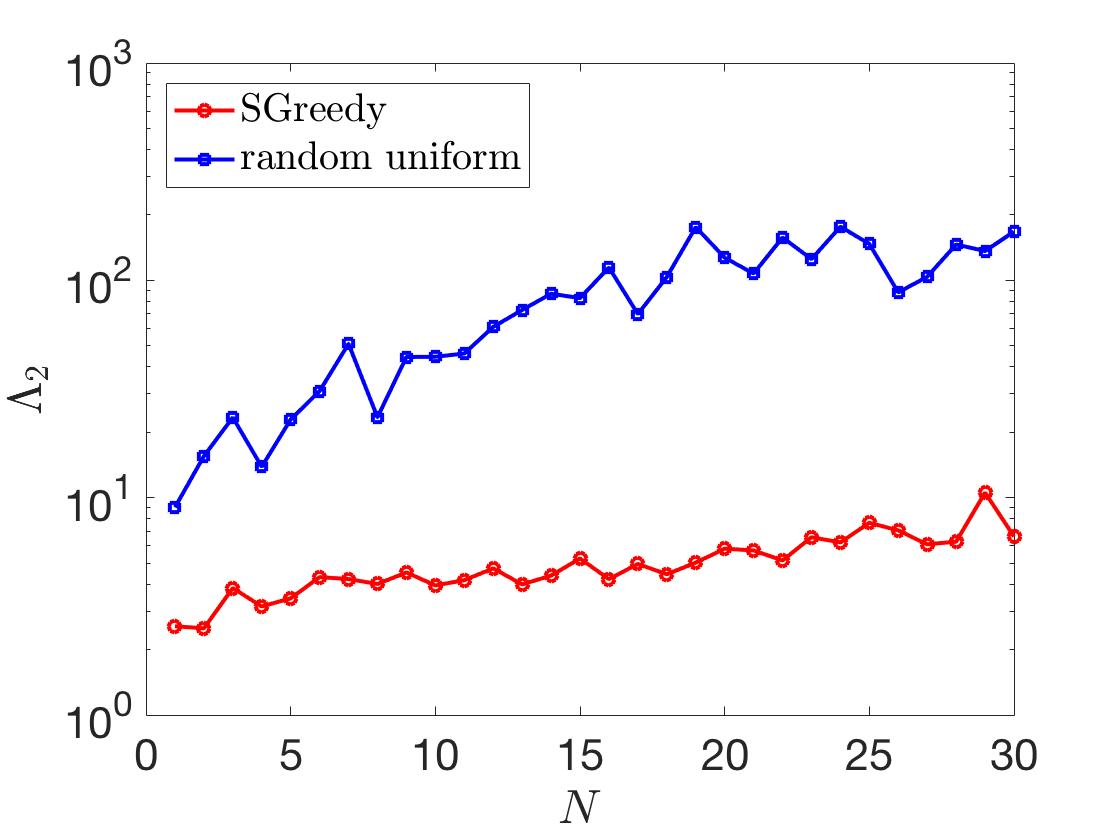}}  
  ~  
\subfloat[$M=N+2$] {\includegraphics[width=0.3\textwidth]
 {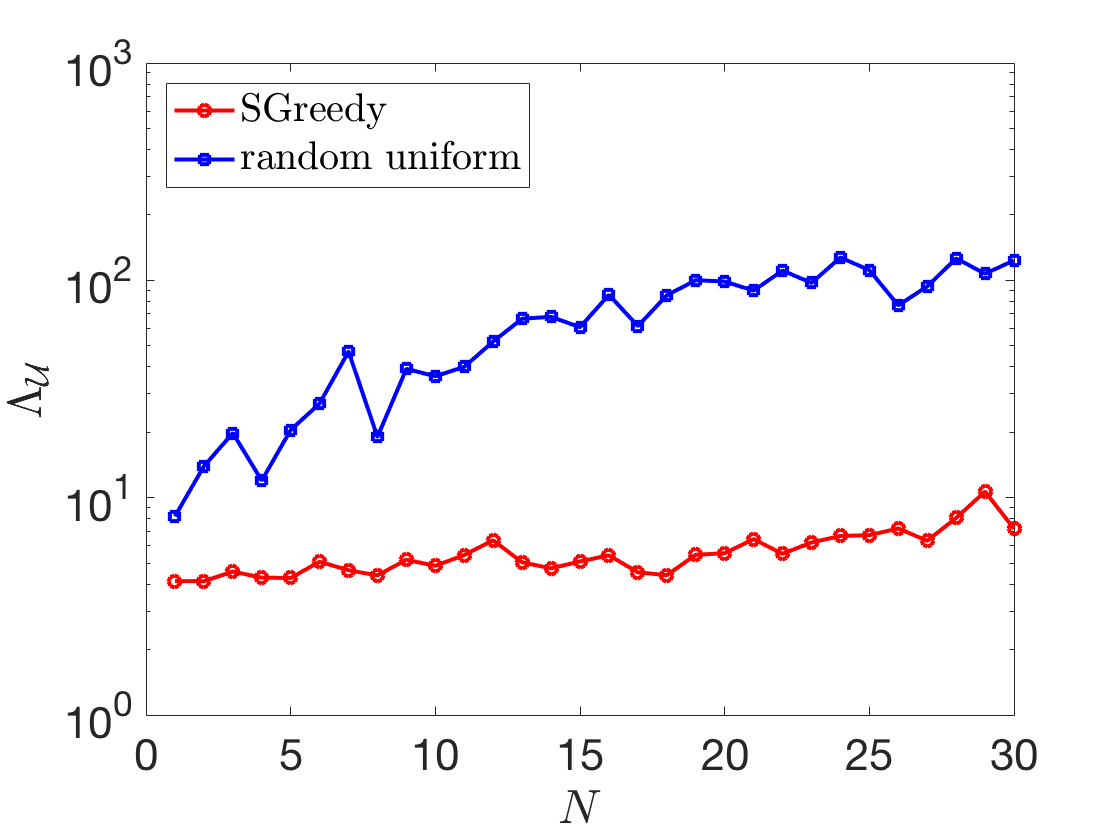}}  

\subfloat[ $M=N+5$] {\includegraphics[width=0.3\textwidth]
 {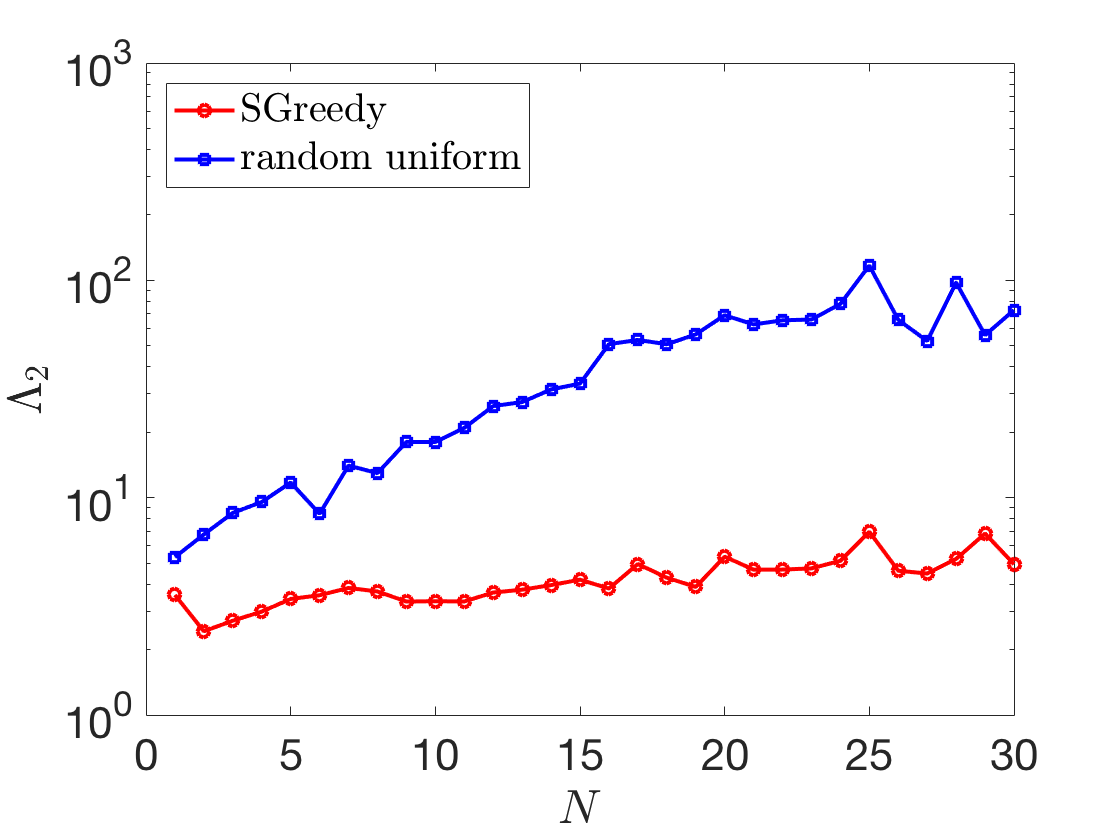}}  
  ~  
\subfloat[$M=N+5$] {\includegraphics[width=0.3\textwidth]
 {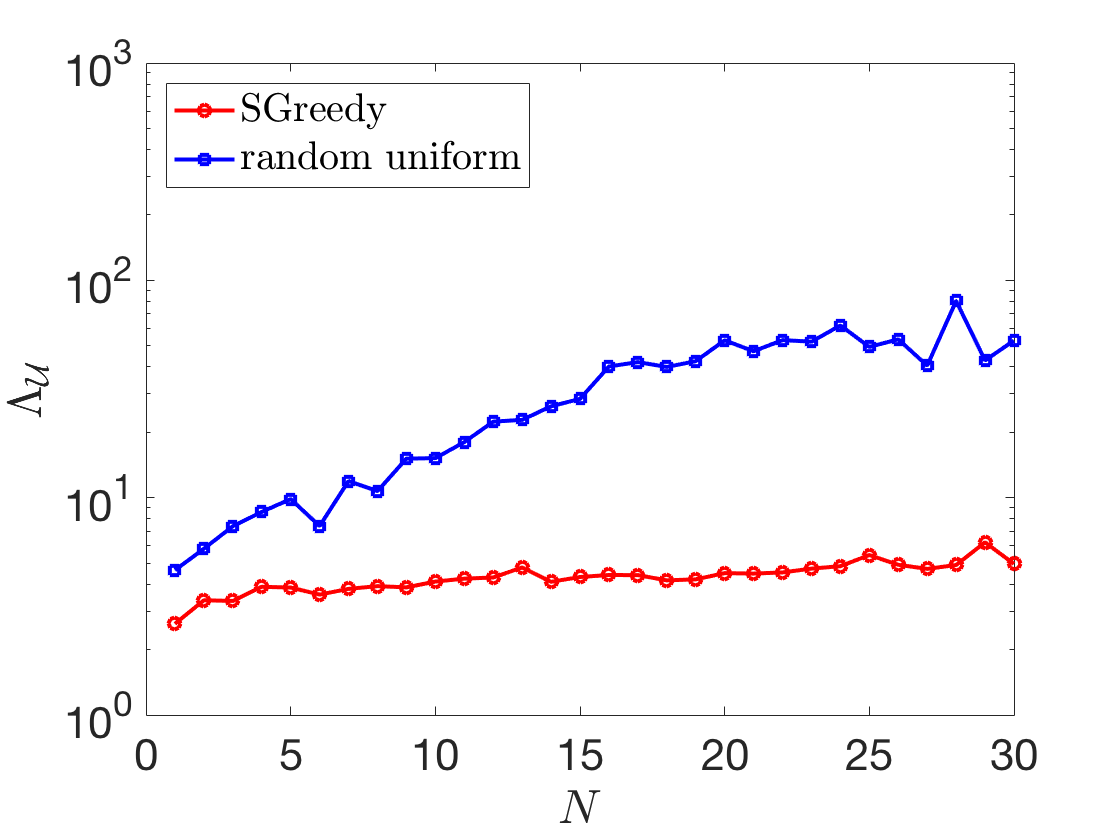}}  
 
\caption{three-dimensional problem. Behavior of $\Lambda_2$ and  $\Lambda_{\mathcal{U}}$ with $N$ for $M=N+2$ (Figures (a)-(b))  and   $M=N+5$ (Figures (c)-(d)).
 }
 \label{fig:constants3D}
\end{figure}

Figures \ref{fig:Nadapt_3D_unbiased} and  \ref{fig:Nadapt_3D_biased}, 
 show the behavior of the average relative error $E_{\rm avg}$ \eqref{eq:Eavg} with $N$ for
several fixed values of $M$,  for two choices of the observation centers, for linear and nonlinear PBDW, for two noise levels, and for both the biased and the unbiased case. In all cases, the regularization hyper-parameter $\xi$ is chosen using holdout validation, based on $I=M/2$ additional measurements. 
As for the previous model problem, we empirically find that the nonlinear formulation improves reconstruction performance for noisy measurements and for non-adaptive selections of measurement locations.

\begin{figure}[h!]
\centering
\subfloat[Greedy,  lin., ${\rm SNR} = \infty$] {\includegraphics[width=0.3\textwidth]
 {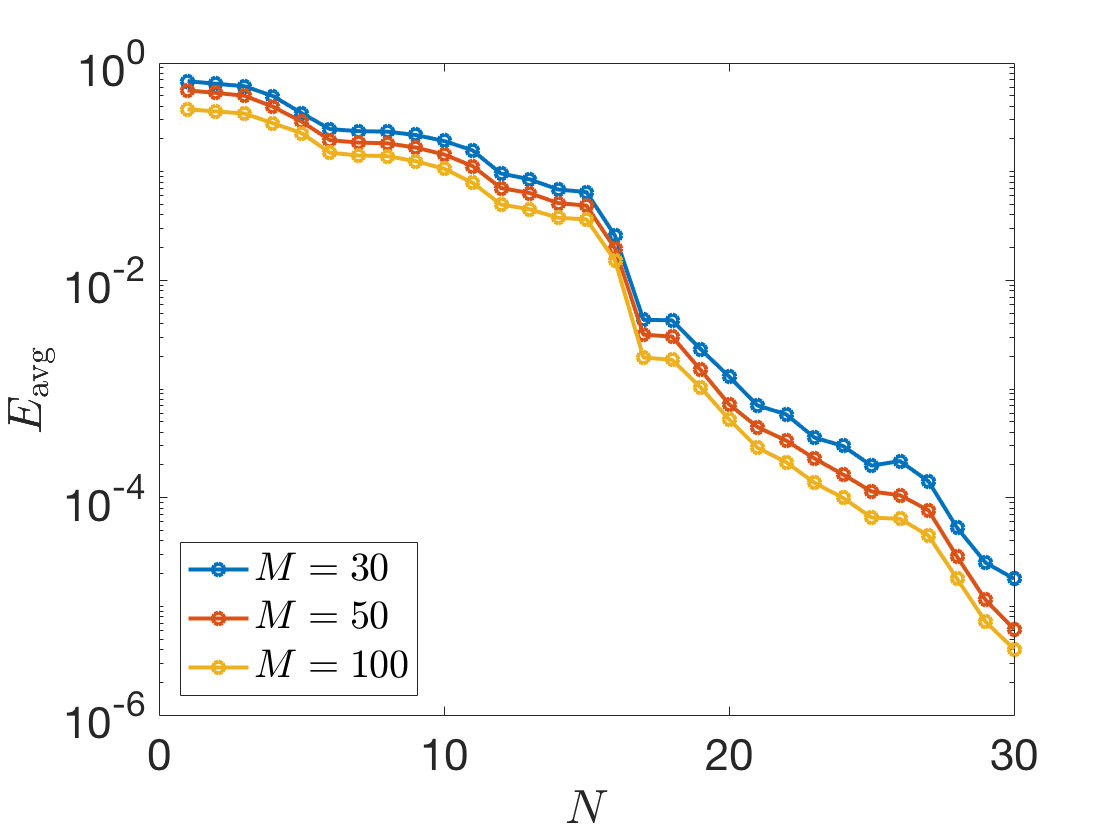}}  
  ~  
  \subfloat[Greedy,  nonlin., ${\rm SNR} = \infty$] {\includegraphics[width=0.3\textwidth]
 {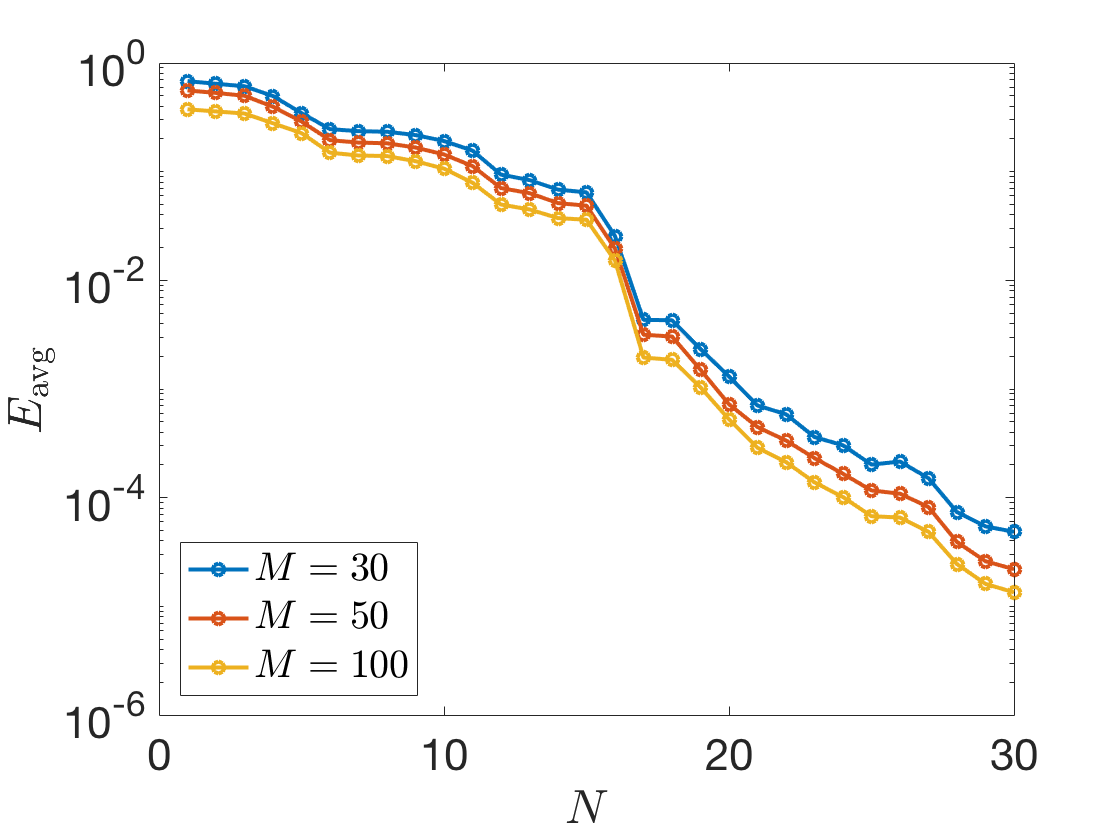}}  

  \subfloat[random,  lin., ${\rm SNR} = \infty$] {\includegraphics[width=0.3\textwidth]
 {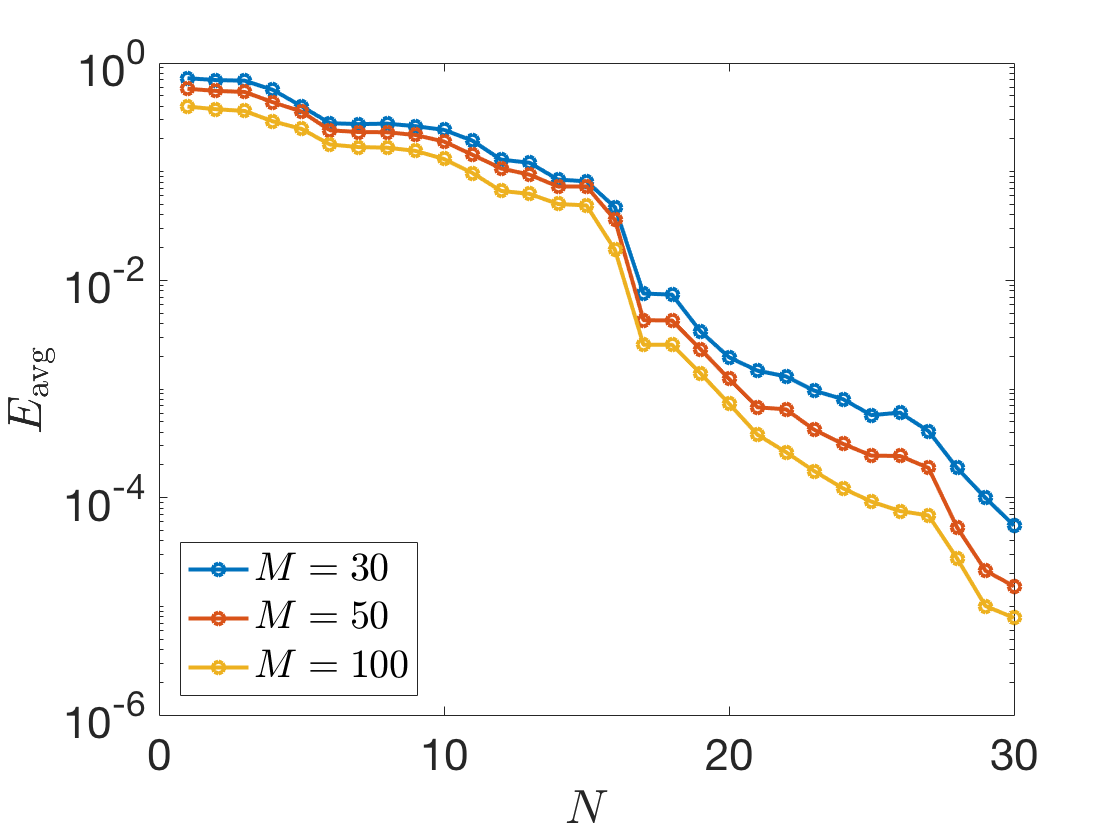}}  
  ~  
  \subfloat[random,  nonlin., ${\rm SNR} = \infty$] {\includegraphics[width=0.3\textwidth]
 {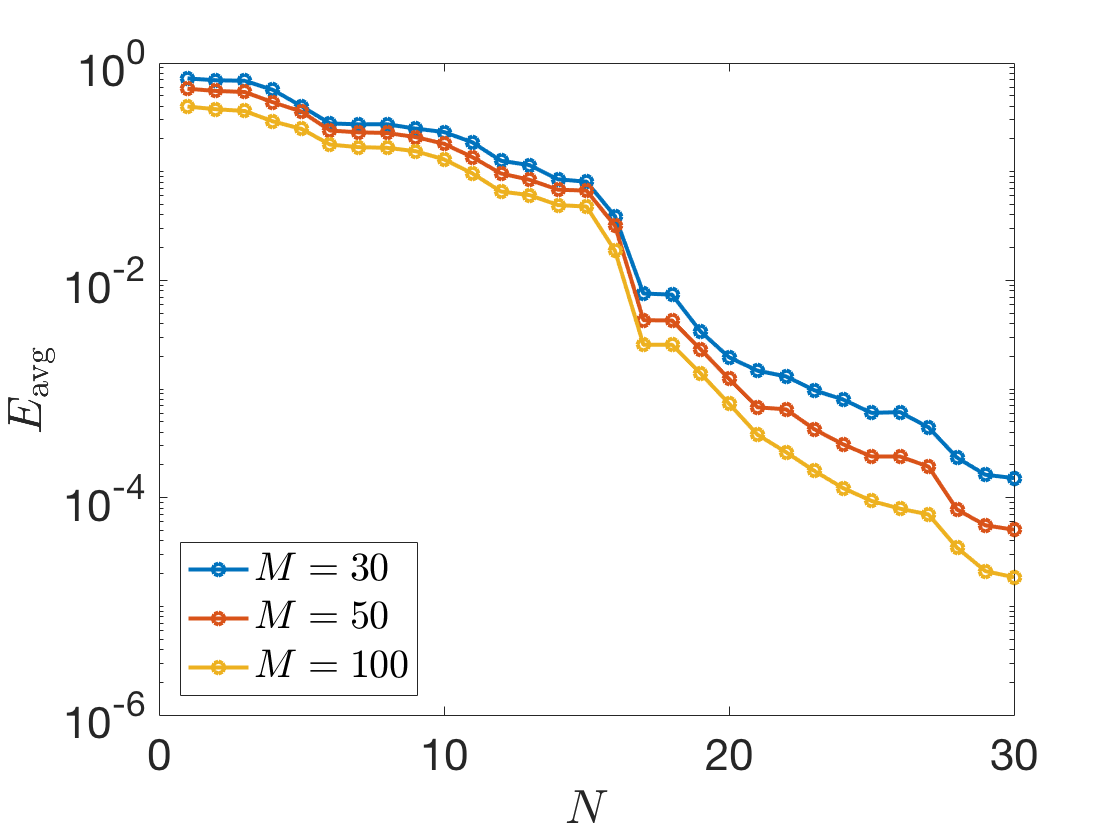}}  
 
 \subfloat[Greedy,  lin., ${\rm SNR}=3$] {\includegraphics[width=0.3\textwidth]
 {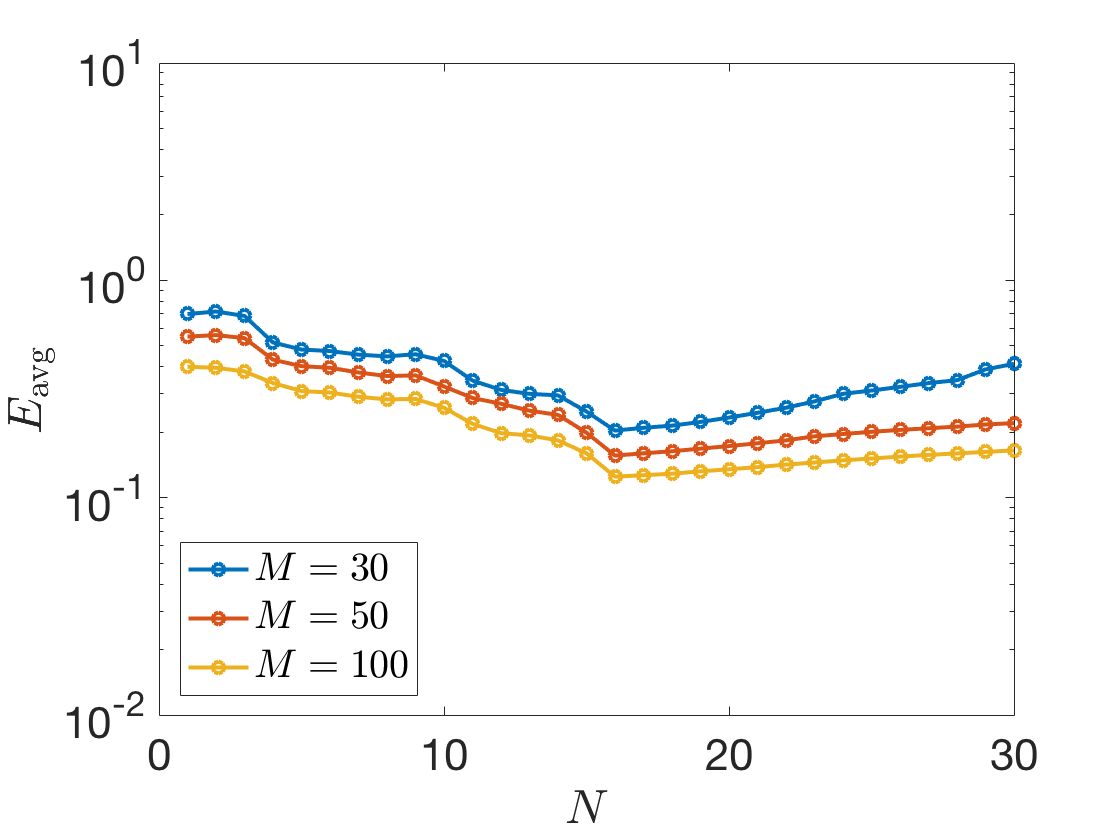}}  
  ~  
  \subfloat[Greedy,  nonlin., ${\rm SNR}=3$] {\includegraphics[width=0.3\textwidth]
 {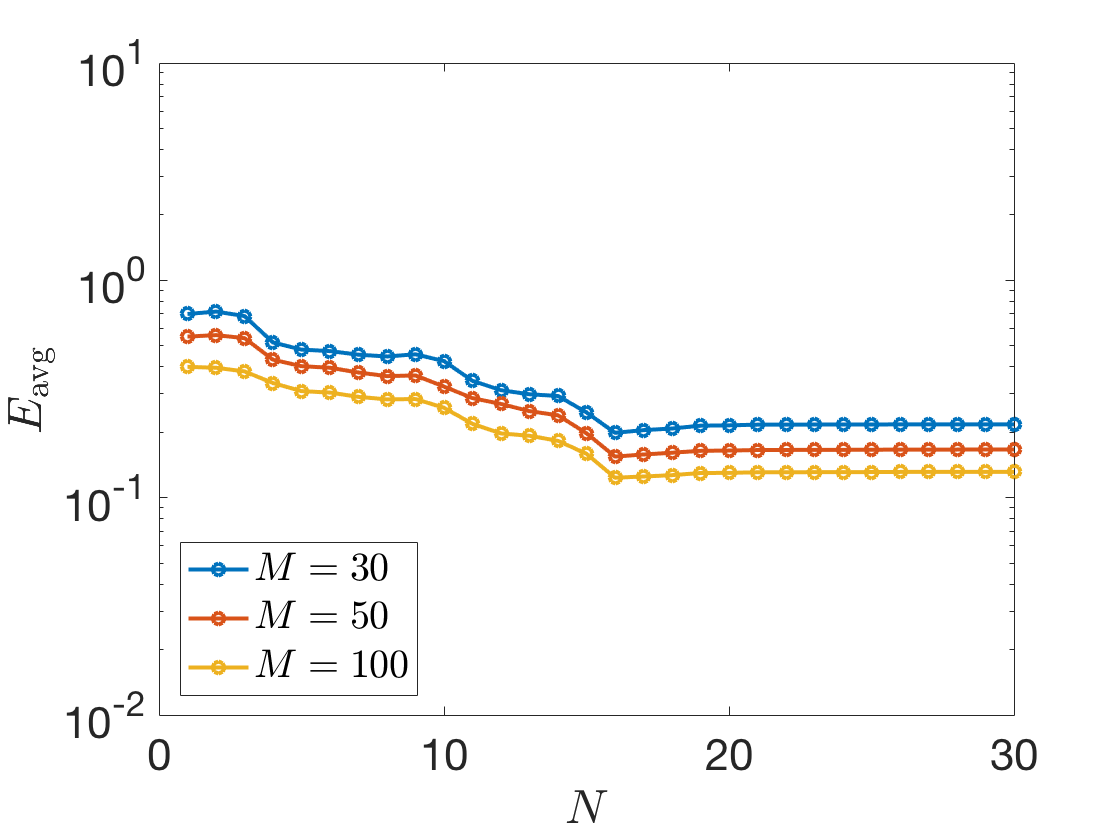}}  

  \subfloat[random,  lin., ${\rm SNR}=3$] {\includegraphics[width=0.3\textwidth]
 {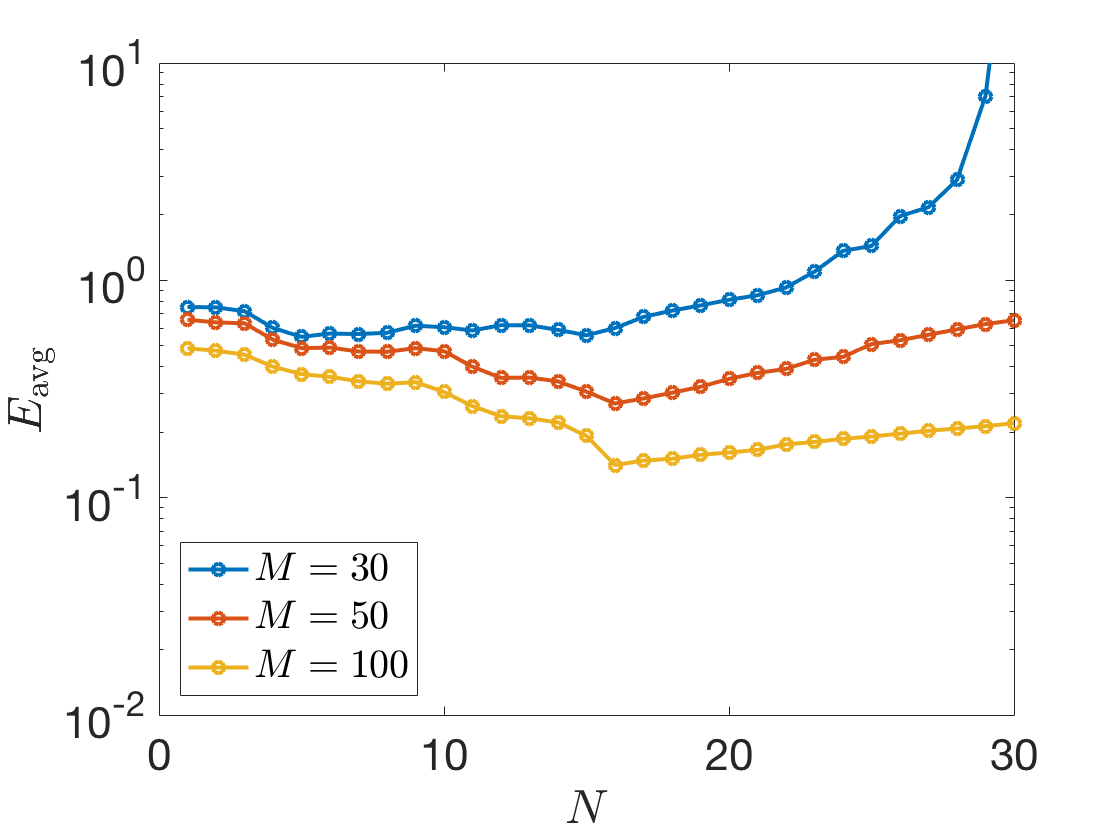}}  
  ~  
  \subfloat[random,  nonlin., ${\rm SNR}=3$] {\includegraphics[width=0.3\textwidth]
 {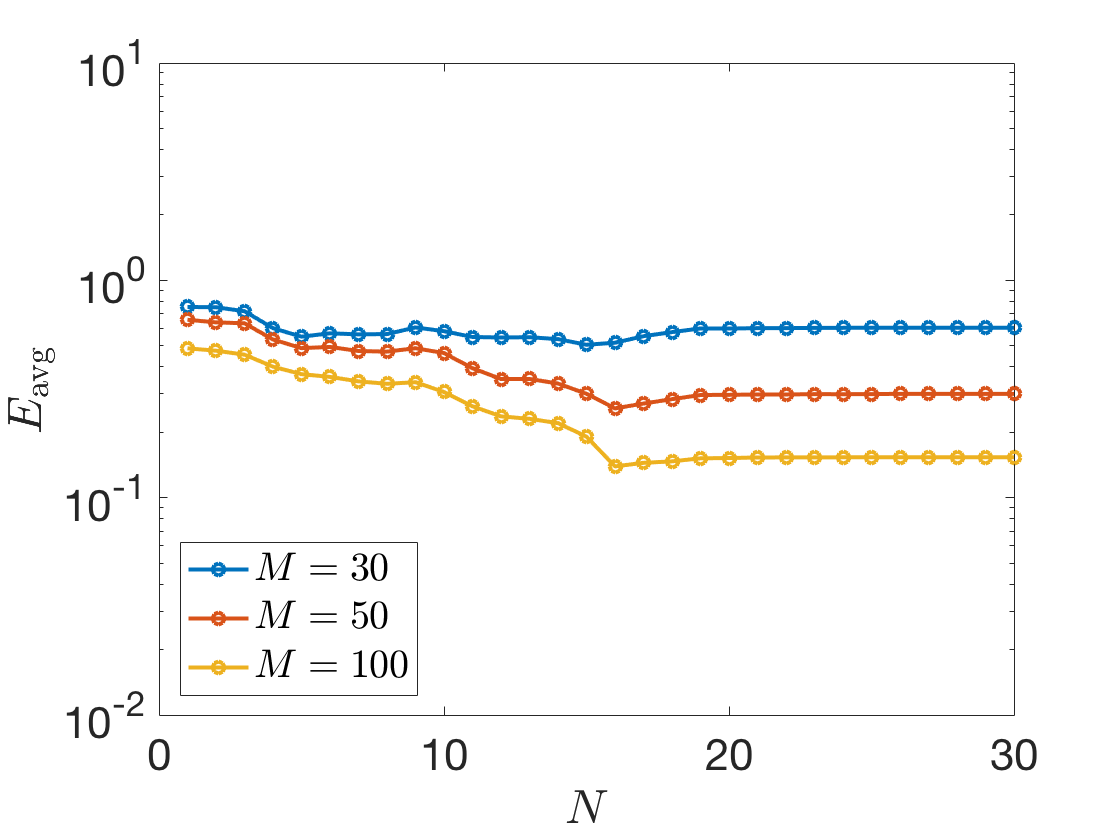}}  
 
\caption{three-dimensional problem.
Behavior of $E_{\rm avg}$ with $N$ for several values of $M$, for linear  and nonlinear PBDW  and two noise levels (unbiased case).
}
 \label{fig:Nadapt_3D_unbiased}
\end{figure}

\begin{figure}[h!]
\centering
\subfloat[Greedy,  lin., ${\rm SNR} = \infty$] {\includegraphics[width=0.3\textwidth]
 {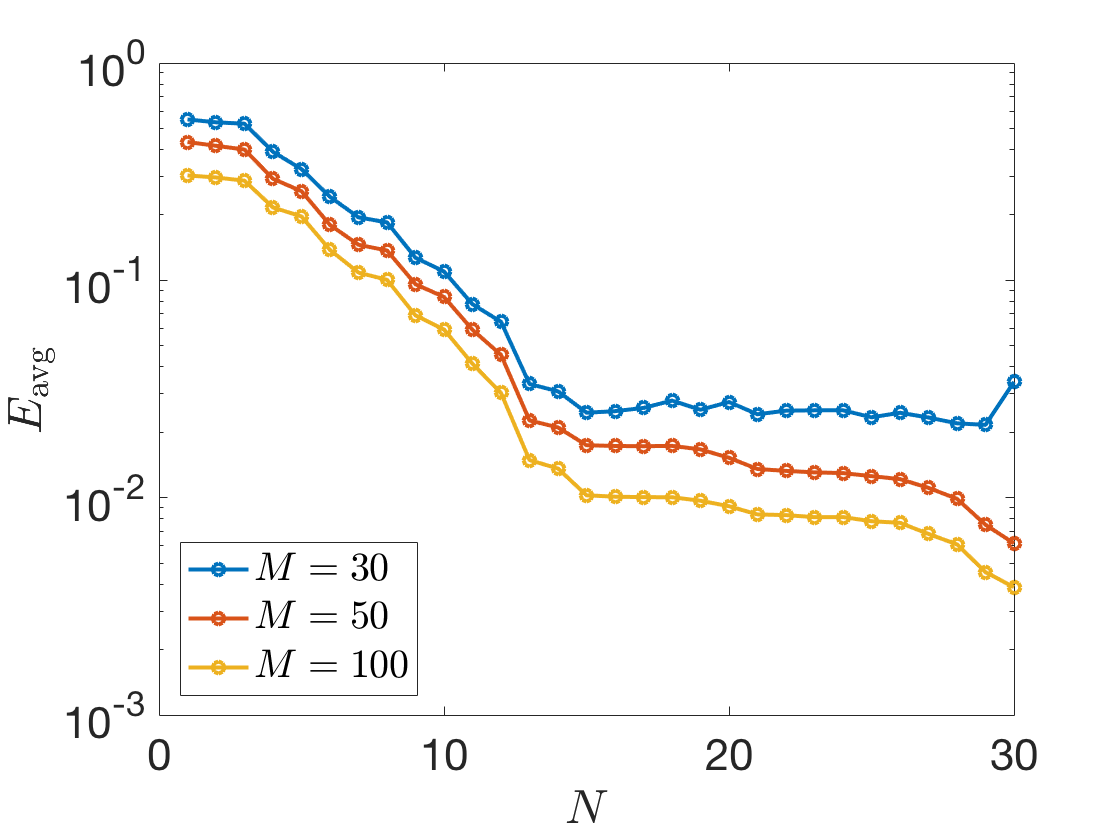}}  
  ~  
  \subfloat[Greedy,  nonlin., ${\rm SNR} = \infty$] {\includegraphics[width=0.3\textwidth]
 {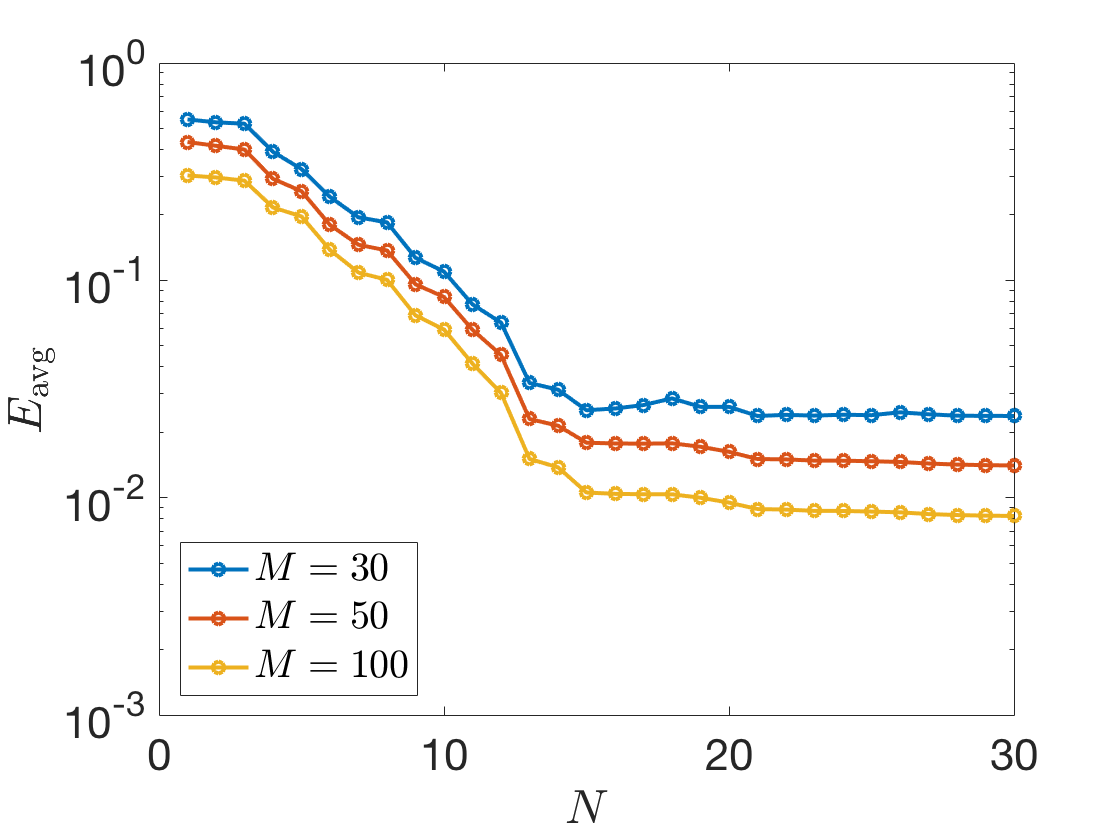}}  

  \subfloat[random,  lin., ${\rm SNR} = \infty$] {\includegraphics[width=0.3\textwidth]
 {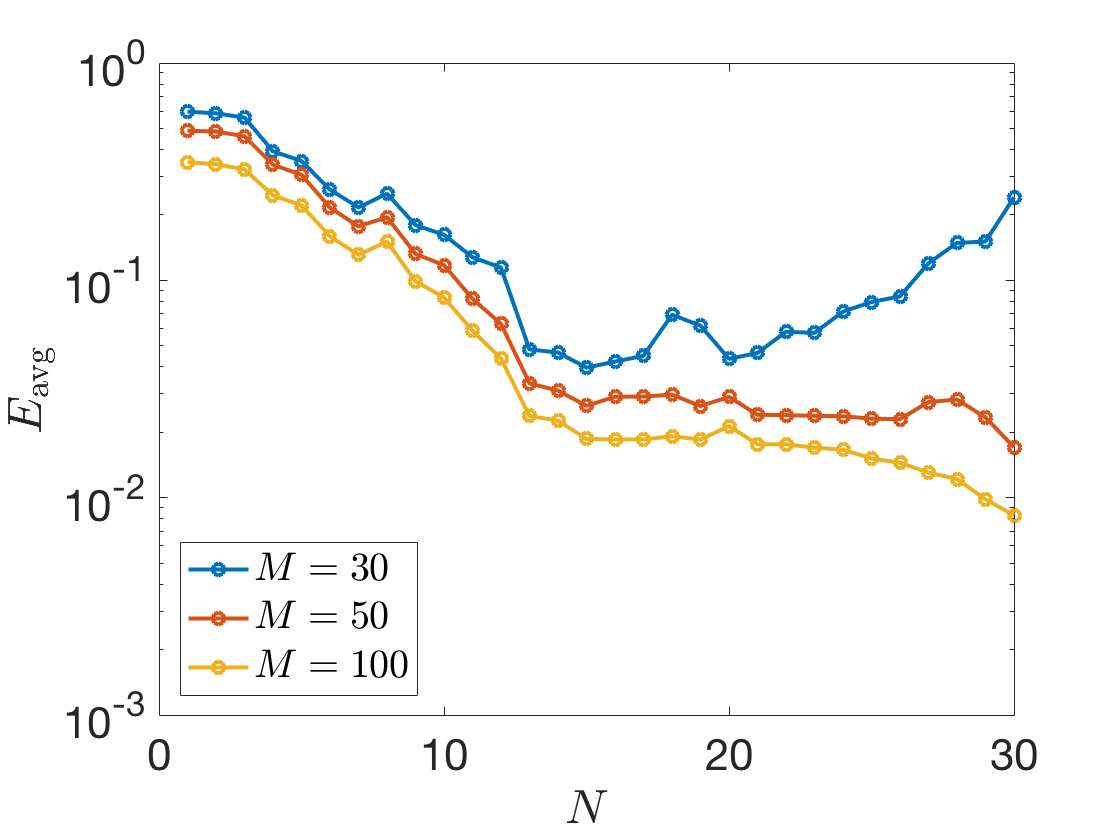}}  
  ~  
  \subfloat[random,  nonlin., ${\rm SNR} = \infty$] {\includegraphics[width=0.3\textwidth]
 {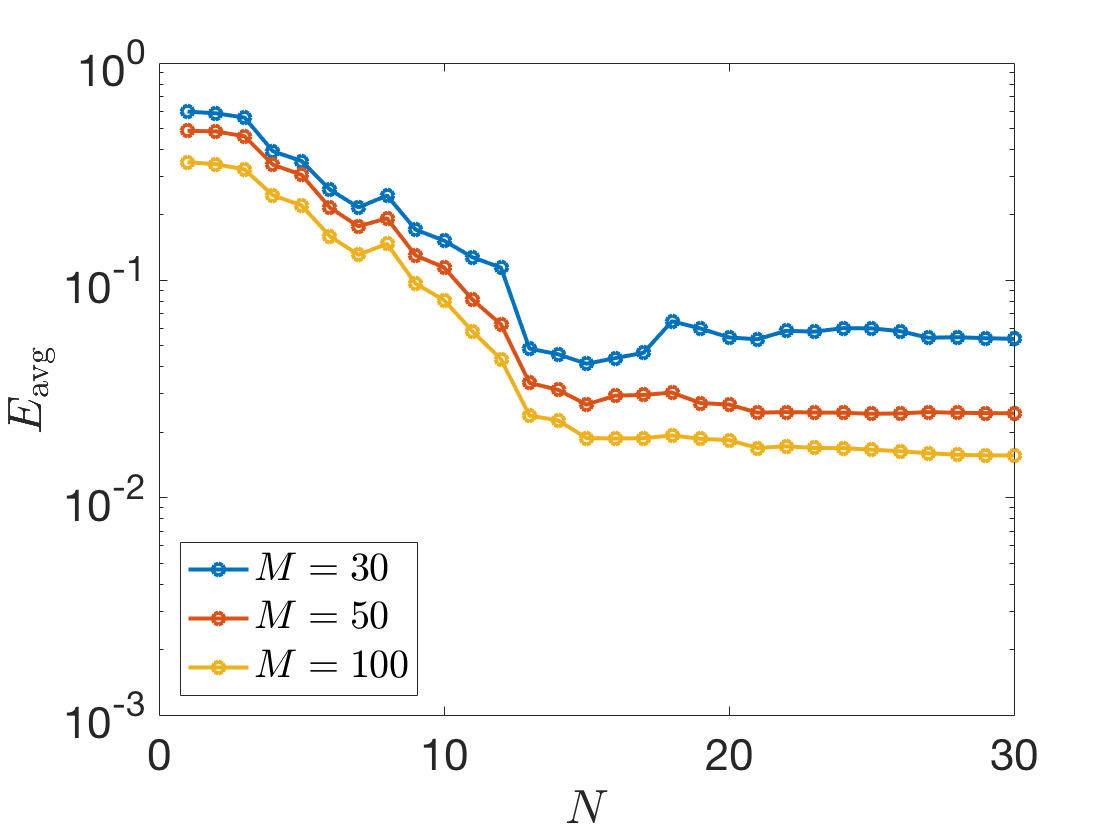}}  
 
\subfloat[Greedy,  lin., ${\rm SNR}=3$] {\includegraphics[width=0.3\textwidth]
 {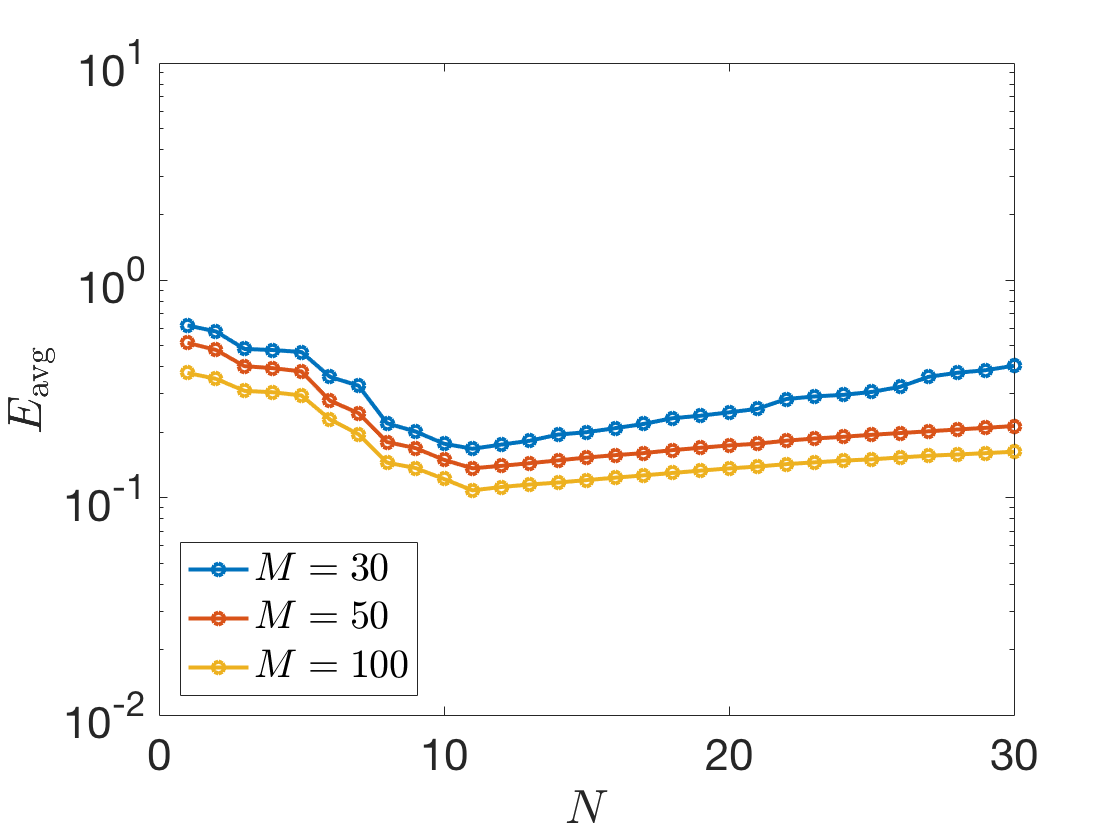}}  
  ~  
  \subfloat[Greedy,  nonlin., ${\rm SNR}=3$] {\includegraphics[width=0.3\textwidth]
 {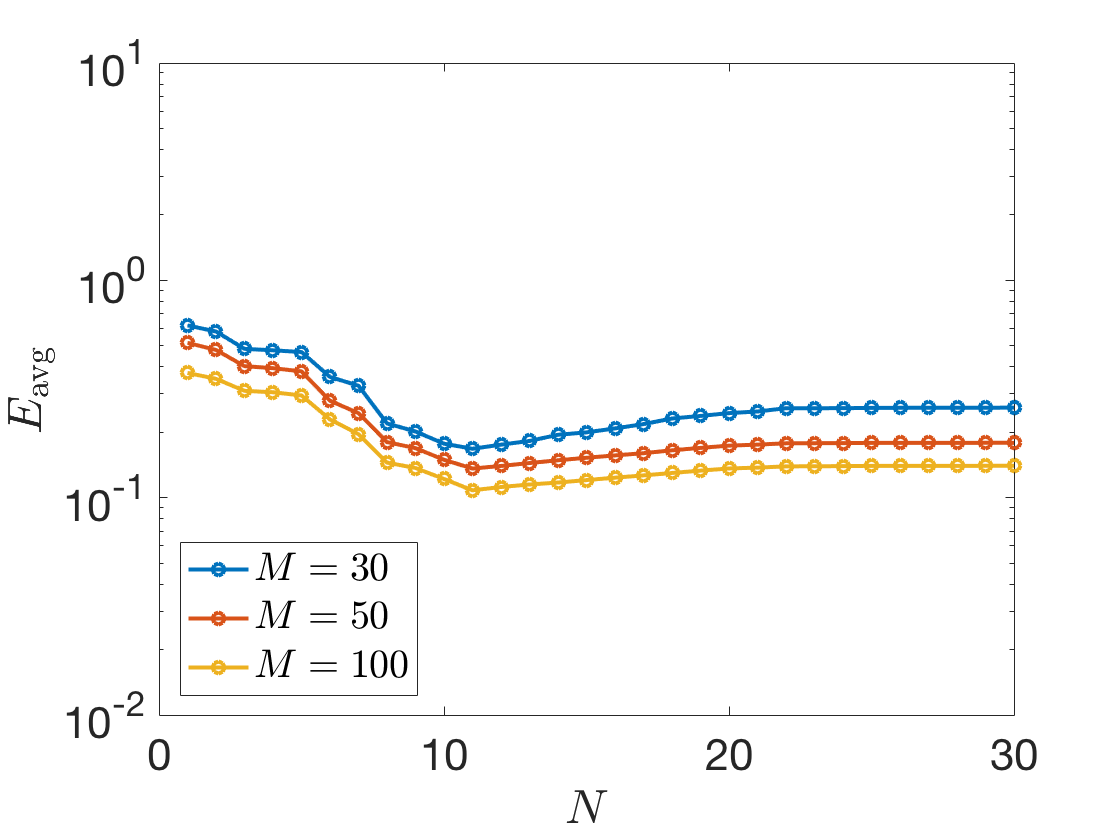}}  

  \subfloat[random,  lin., ${\rm SNR}=3$] {\includegraphics[width=0.3\textwidth]
 {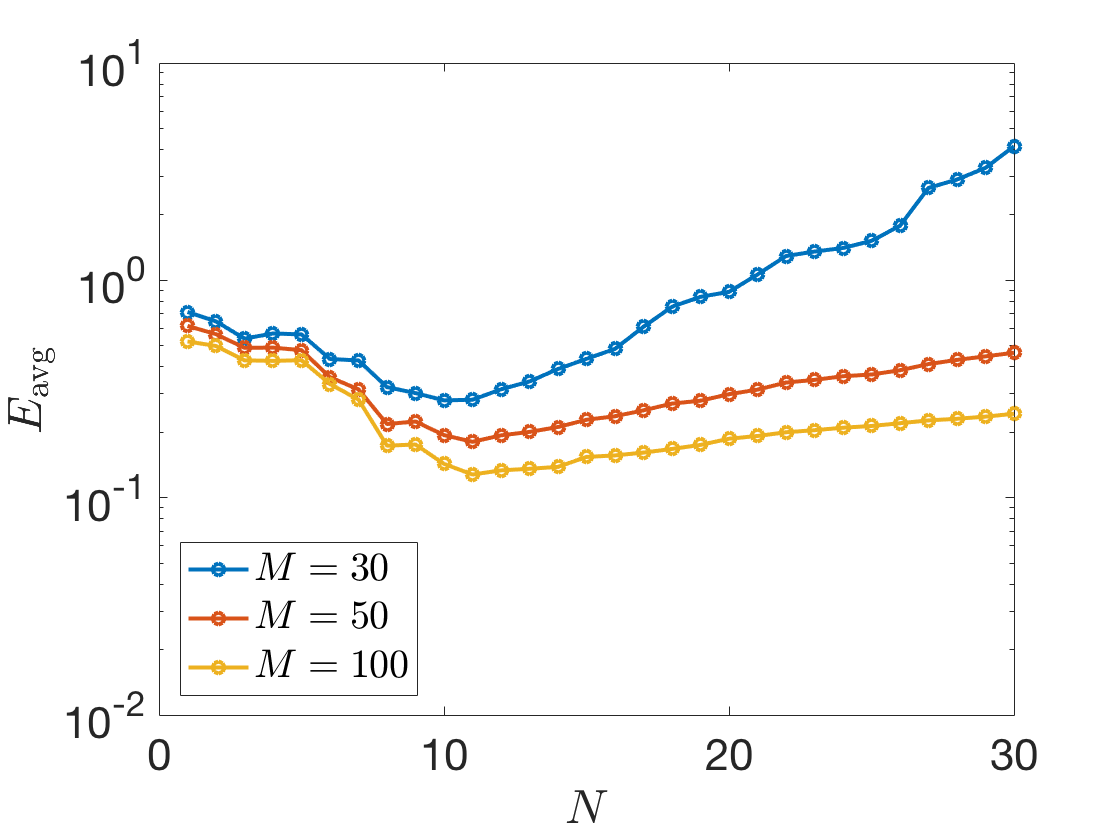}}  
  ~  
  \subfloat[random,  nonlin., ${\rm SNR}=3$] {\includegraphics[width=0.3\textwidth]
 {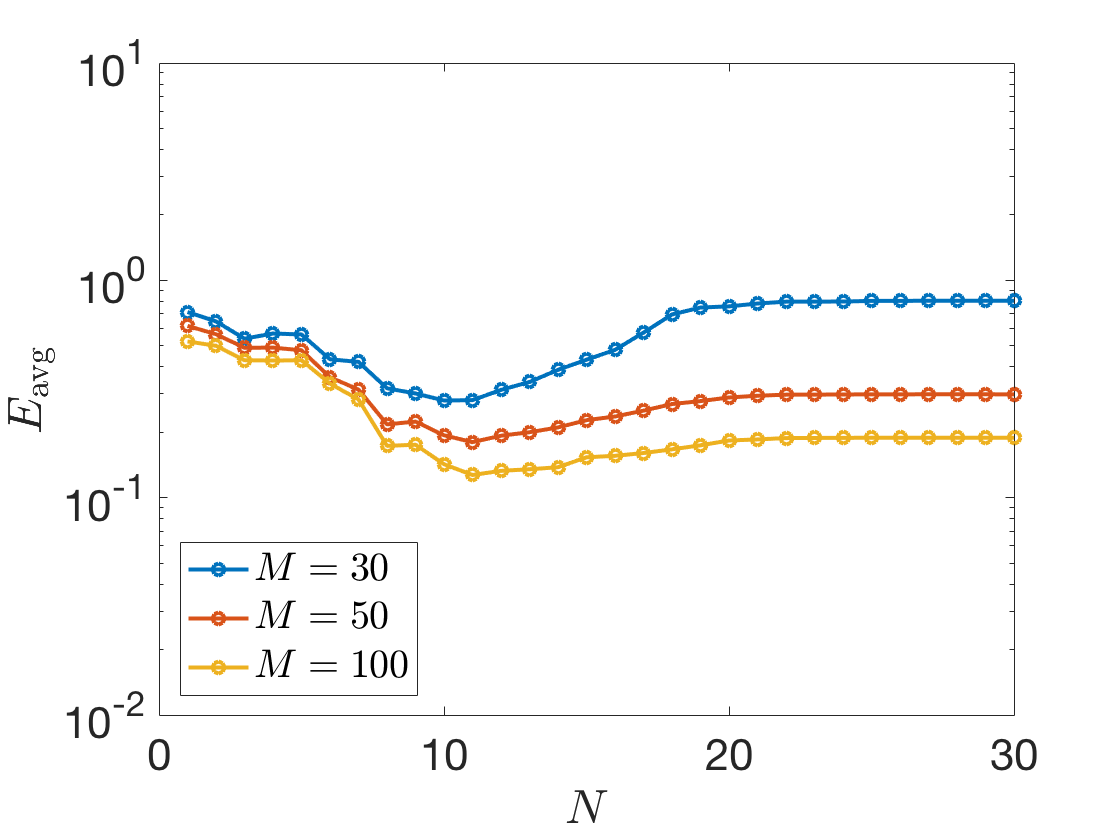}}  
\caption{three-dimensional problem.
Behavior of $E_{\rm avg}$ with $N$ for several values of $M$, for linear  and nonlinear PBDW  and two noise levels, (biased case).
}
 \label{fig:Nadapt_3D_biased}
\end{figure}

\section{Conclusions}
\label{sec:conclusions}
In this paper, we provided theoretical and empirical investigations of the performance of the PBDW approach.
First, we presented a mathematical analysis of the PBDW formulation.
For the linear case, we generalized the analysis in \cite{berger2017sampling} to obtain a complete \emph{a priori} error analysis for noisy measurements, and we also presented  two optimality results that motivate the approach.
For the nonlinear case, we showed a stability estimate that exploits a well-known result, first appeared in the inverse problem  literature.
The latter estimate suggests that the nonlinear formulation should be more robust to measurement error.
Second, we provided several numerical examples to compare the performance of linear PBDW with the performance of  nonlinear PBDW.

Results suggest that the box constraints  for the entries of the background vector $\hat{\mathbf{z}}_{\xi}$ improve the accuracy of the recovery algorithm, provided that measurements are polluted by a non-negligible disturbance. As regards the choice of the observation centers, the SGreedy method stabilizes the recovery algorithm for $N \approx M$ and $\Phi_N = \mathbb{R}^N$. On the other hand, at least for the numerical example considered in this work, in presence of box constraints, SGreedy does not lead to substantial improvements.
Finally, we also empirically found that the nonlinear formulation  is significantly less sensitive to the choice of $N$, particularly for noisy measurements.

\appendix
\section{Proof of Proposition \ref{th:representer_theorem}}
\label{sec:proof_long}

\begin{proof}
Given $\eta \in \mathcal{U}$, recalling the definition of $\mathcal{U}_M$, the Riesz theorem and the projection theorem, we find that
$\mathcal{J}_{\xi}(\mathbf{z}, \eta) = \mathcal{J}_{\xi}(\mathbf{z}, \Pi_{\mathcal{U}_M} \eta   ) + \xi \| \Pi_{\mathcal{U}_M^{\perp}} \eta  \|^2$: therefore, the optimal update $\hat{\eta}_{\xi}$ belongs to  $\mathcal{U}_M$.
Using a similar argument, we can also 
prove that $\hat{\eta}_{0} \in \mathcal{U}_M$.

Proof of \eqref{eq:two_stage_regularized_b} is straightforward and is here omitted; we now focus on \eqref{eq:two_stage_regularized_a}.
Towards this end, 
we  introduce the eigendecomposition $\mathbf{K}= \mathbf{U} \mathbf{D} \mathbf{U}^T$,
$\mathbf{D} = {\rm diag} \big(\lambda_1,\ldots,$ $\lambda_M  \big)$ and we observe that 
$\mathbf{W}_{\xi} = (\xi \mathbf{Id} + \mathbf{K})^{-1} $ satisfies
$\mathbf{W}_{\xi}  =\mathbf{U} \mathbf{D}_{\xi} \mathbf{U}^T$ 
with $\left( \mathbf{D}_{\xi} \right)_{m,m} = \frac{1}{\xi + \lambda_m(\mathbf{K})}$.
If we fix $\mathbf{z} \in \mathbb{R}^N$, it is easy to verify that the unique minimizer of $\mathcal{J}_{\xi, \mathbf{z}}(\cdot) = \mathcal{J}_{\xi}(\mathbf{z}, \cdot)$ is given by
$$
\eta(\mathbf{z}) = \sum_{m=1}^M \, \left( \boldsymbol{\eta}(\mathbf{z}) \right)_m q_m,
\quad
{\rm where} \; 
\boldsymbol{\eta}(\mathbf{z}) = \mathbf{W}_{\xi} \mathbf{y}^{\rm err}(\mathbf{z}).
$$
If we substitute the expression of $\eta$ in \eqref{eq:pbdw}, and we exploit 
the eigendecomposition of $\mathbf{K}$ and $\mathbf{W}_{\xi}$, 
we find
$$
\begin{array}{rl}
\displaystyle{
\mathcal{J}_{\xi}(\mathbf{z}, \eta(\mathbf{z})) =} &
\displaystyle{
\left(
\mathbf{y}^{\rm err}(\mathbf{z})
\right)^T
\left(
\xi \mathbf{W}_{\xi} \mathbf{K} \mathbf{W}_{\xi}
+
\left(
\mathbf{K} \mathbf{W}_{\xi}  - \mathbf{Id}
\right)^T
\left(
\mathbf{K} \mathbf{W}_{\xi}  - \mathbf{Id}
\right)
\right) \,
\mathbf{y}^{\rm err}(\mathbf{z})
} 
\\[3mm]
=
&
\displaystyle{
\left(
\mathbf{y}^{\rm err}(\mathbf{z})
\right)^T
\left(
\mathbf{Id} - \mathbf{K}  \mathbf{W}_{\xi} 
\right) \,
\mathbf{y}^{\rm err}(\mathbf{z})
\,= \,
\left(
\mathbf{y}^{\rm err}(\mathbf{z})
\right)^T
\left(
\xi \mathbf{W}_{\xi} 
\right) \,
\mathbf{y}^{\rm err}(\mathbf{z})
} 
\\
\end{array}
$$
which implies  \eqref{eq:two_stage_regularized_a}. Proof of \eqref{eq:two_stage_xi0} follows the exact same ideas and is here  omitted.

We now  prove \eqref{eq:bounds_limits}. Given $\mathbf{x} \in \mathbb{R}^M$,  exploiting the eigendecomposition of $\mathbf{W}_{\xi}$ we find
$$
\frac{1}{ \xi + \lambda_{\rm max}(\mathbf{K})   }
\| \mathbf{x} \|_2^2 
\leq
\| \mathbf{x} \|_{ \mathbf{W}_{\xi}}^2
\leq
\frac{1}{ \xi + \lambda_{\rm min}(\mathbf{K})   }
\| \mathbf{x} \|_2^2.
$$
By exploiting the latter, we obtain
$$
\begin{array}{rl}
\displaystyle{
\min_{\mathbf{z} \in \Phi_N} \, \| \mathbf{L} \mathbf{z} - \mathbf{y} \|_2^2
\geq } &
\displaystyle{
\left( \xi+ \lambda_{\rm min}(\mathbf{K}) \right)
\min_{\mathbf{z} \in \Phi_N} \, \| \mathbf{L} \mathbf{z} - \mathbf{y} \|_{ \mathbf{W}_{\xi}}^2
}
\\[3mm]
= &
\displaystyle{
\left( \xi+ \lambda_{\rm min}(\mathbf{K}) \right)
 \| \mathbf{L} \hat{\mathbf{z}}_{\xi} - \mathbf{y} \|_{ \mathbf{W}_{\xi}}^2 \geq 
\frac{\xi+ \lambda_{\rm min}(\mathbf{K}) }{ \xi + \lambda_{\rm max}(\mathbf{K})  }
 \| \mathbf{L} \hat{\mathbf{z}}_{\xi} - \mathbf{y} \|_2^2,} \\
\end{array}
$$
which is \eqref{eq:bounds_limits_a}.
By observing that 
the generalized eigenvalues of 
$\mathbf{W}_{\xi}^{1/2} $ $ \boldsymbol{\phi}_m = \lambda_m^{\rm gen} $ $\mathbf{K}^{-1/2}  \boldsymbol{\phi}_m$ are given by 
$\lambda_m^{\rm gen} = \frac{\sqrt{   \lambda_m(\mathbf{K})       }}{\sqrt{\xi +  \lambda_m(\mathbf{K})}}$ for $m=1,\ldots,M$, we obtain
$$
\frac{\lambda_{\rm min}(\mathbf{K})}{ \xi + \lambda_{\rm min}(\mathbf{K})   }
\| \mathbf{x} \|_{\mathbf{K}^{-1}}^2
\leq
\| \mathbf{x} \|_{\mathbf{W}_{\xi}}^2
\leq
\frac{\lambda_{\rm max}(\mathbf{K})}{ \xi + \lambda_{\rm max}(\mathbf{K})   }
\| \mathbf{x} \|_{\mathbf{K}^{-1}}^2, 
$$
and finally
$$
\begin{array}{l}
\displaystyle{
\min_{\mathbf{z} \in \Phi_N} \, \| \mathbf{L} \mathbf{z} - \mathbf{y} \|_{\mathbf{K}^{-1}}^2
\geq
\frac{\xi+ \lambda_{\rm max}(\mathbf{K}) }{ \lambda_{\rm max}(\mathbf{K})  }
\min_{\mathbf{z} \in \Phi_N} \, \| \mathbf{L} \mathbf{z} - \mathbf{y} \|_{  \mathbf{W}_{\xi}}^2
}
\\[3mm]
\hspace{0.3in}
\displaystyle{
=
\frac{\xi+ \lambda_{\rm max}(\mathbf{K}) }{ \lambda_{\rm max}(\mathbf{K})  }
 \| \mathbf{L} \hat{\mathbf{z}}_{\xi} - \mathbf{y} \|_{ 
 \mathbf{W}_{\xi}}^2
\geq
\frac{\lambda_{\rm min}(\mathbf{K}) }{\lambda_{\rm max}(\mathbf{K})  }
\,
\left(
\frac{\xi+ \lambda_{\rm max}(\mathbf{K}) }{ \xi + \lambda_{\rm min}(\mathbf{K})  }
\right)
\,
 \| \mathbf{L} \hat{\mathbf{z}}_{\xi} - \mathbf{y} \|_{\mathbf{K}^{-1}}^2,
}
\end{array}
$$
which is \eqref{eq:bounds_limits_b}.

Since $(\hat{\mathbf{z}}_{\xi},  \hat{\eta}_{\xi})$ minimizes \eqref{eq:pbdw} over all $(\mathbf{z}, \eta) \in \Phi_N \times \mathcal{U}$, we have
$$
\begin{array}{l}
\displaystyle{
\xi \|\hat{\eta}_{\xi}   \|^2 \leq J_{\xi}( \hat{\mathbf{z}}_{\xi},  \hat{\eta}_{\xi}  )
\leq
J_{\xi}(  \hat{\mathbf{z}}_{\infty},  0  ) =
 \min_{\mathbf{z} \in \Phi_N} \, \|  \mathbf{L} \mathbf{z} - \mathbf{y} \|_2^2;
}
\\[3mm]
\displaystyle{
\|  \boldsymbol{\ell}^o(\hat{\eta}_{\xi} ) + \mathbf{L} \hat{\mathbf{z}}_{\xi} - \mathbf{y}        \|_2^2
 \leq J_{\xi}( \hat{\mathbf{z}}_{\xi},  \hat{\eta}_{\xi}  )
\leq
J_{\xi}( \hat{\mathbf{z}}_{0},  \hat{\eta}_{0}  )
=
\xi \| \hat{\eta}_{0} \|^2,
}
\\
\end{array}
$$
which is \eqref{eq:bounds_limits_2}.

The  fourth and  fifth statements  follow directly from the algebraic formulation of the PBDW statement, and 
from well-known results in convex optimization: we omit the details.
\end{proof}

\section{Proof of Proposition \ref{th:optimality_gauss_markov} }
\label{sec:appendix_gauss_markov}

We state upfront that the proof follows the same idea of the well-known Gauss-Markov theorem (\cite{aitken_1936}) for linear unbiased estimators.

\begin{proof}
Without loss of generality, we assume that $ \{ \zeta_n \}_{n=1}^N$ is orthonormal; then,  we denote by $\hat{\mathbf{z}}_{\xi} (\mathbf{y})$ the vector of coefficients associated with the basis $ \{ \zeta_n \}_{n=1}^N$ and the solution to the PBDW statement for $\xi \in \{0, \infty\}$,  and we denote by ${\mathbf{z}}_{A} (\mathbf{y})$ the vector of coefficients associated with the algorithm $A$. 

We first prove that $\Lambda_2(A) \geq \Lambda_2(A^{\rm pbdw, \xi = \infty})$.
Since both PBDW and $A$ are linear,
recalling the definition of $\mathbf{L}$, $\mathbf{L}_{m,n} = \ell_m(\zeta_n)$, 
 we have that
$$
{\mathbf{z}}_{A} (\mathbf{y}) =
\hat{\mathbf{z}}_{\infty} (\mathbf{y})
\, + \, \mathbf{D} \, \mathbf{y},
\qquad
\hat{\mathbf{z}}_{\infty} (\mathbf{y})
= 
\left( \mathbf{L}^T \mathbf{L} \right)^{-1} \mathbf{L}^T \, \mathbf{y},
$$
for a proper choice of  $\mathbf{D} \in \mathbb{R}^{N,M}$.  Since 
$\Lambda_{\mathcal{U}}^{\rm bias}(A) = 0$, we must have
${\mathbf{z}}_{A} (\mathbf{L} \mathbf{z}) = \mathbf{z} $ for all $\mathbf{z} \in \mathbb{R}^N$; this implies that
$$
{\mathbf{z}}_{A} (\mathbf{L} \mathbf{z}) =
\mathbf{z} \, + \, \mathbf{D} \, \mathbf{L} \mathbf{z} \, 	\Rightarrow \, 
\mathbf{D} \, \mathbf{L} = 0.
$$
Recalling \eqref{eq:algebraic_lebesgue}, we shall prove that
$$
s_{\rm max} 
\left(
\left( \mathbf{L}^T  \, \mathbf{L}  \right)^{-1} \mathbf{L}^T  \, + \, \mathbf{D}
\right)
\geq
s_{\rm max} 
\left(
\left( \mathbf{L}^T  \, \mathbf{L}  \right)^{-1} \mathbf{L}^T 
\right).
$$
Towards this end,  we observe that\footnote{
We recall that the maximum singular value of a matrix $\mathbf{B}$ is the square root of the maximum eigenvalue of $\mathbf{B}^T \mathbf{B}$ or equivalently of 
$\mathbf{B}\mathbf{B}^T $.
}
$$
\begin{array}{l}
\displaystyle{
\left( \Lambda_{2}(A)  \right)^2   \, = \, 
 \sup_{\mathbf{y} \in \mathbb{R}^M} \, \frac{\big \|
 \left(   \left(  
 \mathbf{L}^T  \, \mathbf{L}  \right)^{-1} \mathbf{L}^T  \, + \, \mathbf{D} 
 \right)  \mathbf{y} \big\|_2^2}{\|\mathbf{y} \|_2^2} 
=
  \sup_{\mathbf{z} \in \mathbb{R}^N} \, \frac{
  \big \|
   \left(  \mathbf{L}  \left(    \mathbf{L}^T  \, \mathbf{L}  \right)^{-1}   \, + \, \mathbf{D}^T   \right)  \mathbf{z} \|_2^2}{\|\mathbf{z} \|_2^2} 
 }
\\[3mm]
\hspace{0.6in}
 =
\displaystyle{
 \sup_{\mathbf{z} \in \mathbb{R}^N} \, 
 \frac{ \mathbf{z}^T
\left(
\left(  \mathbf{L}^T  \, \mathbf{L}   \right)^{-1}  
\, + \, 
\left(  \mathbf{L}^T  \, \mathbf{L}   \right)^{-1}  \left(\mathbf{D}  \mathbf{L}    \right)^T
\, + \,
\mathbf{D}  \mathbf{L}  \, \left(  \mathbf{L}^T  \, \mathbf{L}   \right)^{-1}
\, + \,
\mathbf{D} \mathbf{D}^T
 \right)
 \mathbf{z}
 }
{\|\mathbf{z}\|_2^2}
}
\\[3mm]
\hspace{0.6in} 
 =
 \displaystyle{
 \sup_{\mathbf{z} \in \mathbb{R}^N} \, 
 \frac{ \mathbf{z}^T
\left(
\left(  \mathbf{L}^T  \, \mathbf{L}   \right)^{-1}  
\, + \,
\mathbf{D} \mathbf{D}^T
 \right)
 \mathbf{z}
 }
{\|\mathbf{z}\|_2^2}
\geq
 \sup_{\mathbf{z} \in \mathbb{R}^N} \, 
 \frac{ \mathbf{z}^T
\left(
\left(  \mathbf{L}^T  \, \mathbf{L}   \right)^{-1}  
 \right)
 \mathbf{z}
 }
{\|\mathbf{z}\|_2^2}
} \\[3mm]
\hspace{0.6in}
\displaystyle{
= \,
\left(
\Lambda_2(A^{\rm pbdw, \xi = \infty})
\right)^2 },
\end{array}
$$
which is the thesis.
Note that in the second-to-last step we used the fact that $\mathbf{D} \mathbf{D}^T$ is semi-positive definite.

We now prove that $\Lambda_{\mathcal{U}}(A) \geq \Lambda_{\mathcal{U}}(
\Pi_{\mathcal{Z}_N}
A^{\rm pbdw, \xi = 0})$. As for the previous case, we observe that
$$
{\mathbf{z}}_{A} (\mathbf{y}) =
\hat{\mathbf{z}}_{0} (\mathbf{y})
\, + \, \mathbf{E} \, \mathbf{y},
\qquad
\hat{\mathbf{z}}_{0} (\mathbf{y})
= 
\left( \mathbf{L}^T \mathbf{K}^{-1} \mathbf{L} \right)^{-1} \mathbf{L}^T \, \mathbf{K}^{-1} \, \mathbf{y},
$$
where the matrix $\mathbf{E} \in \mathbb{R}^{N,M}$  should satisfy
 $\mathbf{E} \mathbf{L} = 0$.
Exploiting Lemma \ref{th:properties_lebesgue_constants}, we find the desidered result:
$$
\begin{array}{rl}
\Lambda_{\mathcal{U}}(A)
= &
\displaystyle{
s_{\rm max} (\mathbf{A} \mathbf{K}^{1/2}) 
=
\sup_{\mathbf{z} \in \mathbb{R}^N} \,
\frac{
\big\| 
\mathbf{K}^{1/2} \left(
\mathbf{K}^{-1}
\mathbf{L} \left(\mathbf{L}^T \mathbf{K}^{-1} \mathbf{L}   \right)^{-1}
\,  +  \, \mathbf{E}^T
\right)
\mathbf{z}
   \big\|_2}{\| \mathbf{z} \|_2}
}
\\[3mm]
= &
\displaystyle{
\sup_{\mathbf{z} \in \mathbb{R}^N} \,
\frac{
\sqrt{ 
\mathbf{z}^T
\left(
\left(
\mathbf{L}^T \mathbf{K}^{-1} \mathbf{L}
\right)^{-1}
\, + \,
\mathbf{E} \mathbf{K} \mathbf{E}^T 
\right)
\mathbf{z}
 }
  }{\| \mathbf{z} \|_2}
\, \geq \,
\sup_{\mathbf{z} \in \mathbb{R}^N} \,
\frac{
\sqrt{ 
\mathbf{z}^T
\left( \mathbf{L}^T \mathbf{K}^{-1} \mathbf{L} \right)^{-1}
\, 
\mathbf{z}
 }
  }{\| \mathbf{z} \|_2}
}
\\[3mm]
= &
\displaystyle{
\Lambda_{\mathcal{U}}
\left( 
 \Pi_{\mathcal{Z}_N}A^{\rm pbdw, \xi = 0} \right).
}
\\[3mm]
\end{array}
$$
\end{proof}

\bibliographystyle{plain}
\bibliography{all_refs}
 
\end{document}